\documentclass[11pt,a4paper]{article}
\usepackage[margin=26mm]{geometry}
\usepackage[T1]{fontenc}
\usepackage{amsmath,amssymb,amsthm,mathtools}
\usepackage{microtype,enumitem,booktabs,array}
\usepackage{graphicx,float}
\usepackage[hidelinks]{hyperref}
\hypersetup{pdftitle={Integrable distortion II: the p-Weil-Petersson variational problem},pdfauthor={Gaven Martin and Cong Yao}}
\newcommand{\D}{\mathbb D}
\newcommand{\ID}{\mathbb D}
\newcommand{\Sone}{\mathbb S^1}
\newcommand{\IS}{\mathbb S^1}
\newcommand{\IR}{\mathbb R}
\newcommand{\IC}{\mathbb C}
\newcommand{\K}{\mathbb K}
\newcommand{\IK}{\mathbb K}
\newcommand{\I}{\mathcal I}
\newcommand{\E}{\mathcal E}

\newcommand{\area}{\,dA}
\newcommand{\bz}{\bar z}
\newcommand{\bw}{\bar w}
\DeclareMathOperator{\loc}{loc}
\DeclareMathOperator{\id}{id}
\newtheorem{mainthm}{Theorem}

\newtheorem{theorem}{Theorem}[section]
\newtheorem{proposition}[theorem]{Proposition}
\newtheorem{lemma}[theorem]{Lemma}
\newtheorem{corollary}[theorem]{Corollary}
\theoremstyle{definition}

\newtheorem{example}[theorem]{Example}
\newtheorem{remark}[theorem]{Remark}
\numberwithin{equation}{section}
\setlist{itemsep=4pt,topsep=5pt}
\allowdisplaybreaks[1]
\title{Integrable distortion II:\\ $p$-Weil--Petersson circle homeomorphisms\thanks{{\tiny
{\bf Keywords.} Teichm\"uller theory, quasiconformal mappings, finite distortion, extremal mappings, calculus of variations, polyconvexity
\newline
{\bf MSC Subject.} 30C62 31A05 49J10 35J60 } }}
\author{Gaven Martin \and Cong Yao}
\date{}
\begin{document}
\maketitle
\begin{abstract}
We study the hyperbolically weighted distortion energies
\[
 \I_p(f)=\int_\D
 \frac{|\mu_f|^p}{(1-|\mu_f|^2)^{\max\{2,p\}/2}}\,dA_\eta
\]
for finite-distortion disk homeomorphisms with prescribed boundary values, together with the associated inverse problem in the monotone Sobolev class. For $1<p\leq2$, the inverse density is polyconvex precisely when $p\geq p_{\mathrm c}=1+1/\sqrt2$. In this range, quasisymmetric data with finite-energy monotone barriers yield minimisers with holomorphic Ahlfors--Hopf differentials and strong local Sobolev convergence of minimising subsequences.

For $p_{\mathrm c}<p<2$, a finite-energy finite-distortion barrier yields a unique absolute monotone minimiser, which is a quasiconformal diffeomorphism and is asymptotically conformal. Together with the boundary-class theorem of Part~I, this characterises the admissible boundary data as $T^p$. The proof combines a selected barrier approximation, uniform scalar gradient estimates and a weighted Reich--Strebel comparison.

For $1<p<p_{\mathrm c}$, finite rank-one comparisons bound distortion on the positive-Jacobian set by a cutoff strictly below the ellipticity cutoff. An attained monotone minimum is either a quasiconformal diffeomorphism or has nonzero rank-one derivatives on a set of positive area. For $0<p\leq1$, minimisers with finite distortion, or homeomorphic minimisers with a locally Sobolev inverse, are M\"obius maps. We also obtain energy-dependent quasiconformal bounds for regular minimisers when $p>2$.
\end{abstract}

\begingroup
\small
\makeatletter
\renewcommand*\l@section[2]{\@dottedtocline{1}{0em}{2em}{\bfseries #1}{\bfseries #2}}
\makeatother
\tableofcontents
\endgroup

\section{Introduction}
This paper continues the study of hyperbolically weighted distortion energies on the disk. Write
\[
 \eta(z)=(1-|z|^2)^{-2},\qquad dA_\eta=\eta\,dA,
\]
where $dA$ is Lebesgue area and $\eta|dz|^2$ has curvature $-4$. With $\mu_f=f_{\bar z}/f_z$, we consider

\begin{equation}\label{intro-functional}
 \I_p(f)=\int_\D\frac{|\mu_f|^p}{(1-|\mu_f|^2)^{p_*/2}}\,dA_\eta,\qquad p_*=\max\{2,p\},
\end{equation}
over finite-distortion homeomorphisms $f:\D\to\D$ with prescribed boundary values, and the inverse functional $\E_p(h)=\I_p(h^{-1})$, which for $1\leq p\leq2$ is
\begin{equation}\label{intro-inverse}
 \E_p(h)=\int_\D\eta(h)|h_w|^{2-p}|h_{\bar w}|^p\,dA ,
\end{equation}
over continuous monotone surjections $h:\overline\D\to\overline\D$ with trace $\beta$, $h\in W^{1,2}(\D)$, $J_h\geq0$ and $\E_p(h)<\infty$. We denote this class by $\mathcal M_\beta$, suppressing its dependence on $p$. Part~I treated $p=2$, when the minimisers are hyperbolic harmonic maps, and proved the Schoen conjecture with explicit constants. Its boundary capacity theorem also identifies finite-energy finite-distortion boundary values with Guo's $p$-integrable Teichm\"uller space $T^p$ for every $p>1$ \cite[Theorem E]{MYWP}. Thus a finite-distortion barrier can be replaced by a quasiconformal one. For $p<2$, the distortion profile is concave, the inverse density loses polyconvexity below a sharp threshold, and the Euler--Lagrange equation is no longer the harmonic-map equation.

Write $r=|\mu|$ and, for ease of notation, define, for $1<p\leq2$, the {\em flux} functions
\begin{equation}\label{flux-functions}
 \begin{split}
 F_p(r)&=\tfrac12r^{p-1}[p+(2-p)r^2],\\
 H_p(r)&=p(p-1)-2p(2-p)r-(p-1)(2-p)r^2,\\
 Q_p(r)&=\frac{4F_p(r)}{(1+r)^2},\qquad
 \ell_p=2(p-1)^2-1,\qquad p_{\mathrm c}=1+\frac1{\sqrt2}.
 \end{split}
\end{equation}
For $1<p<2$, the polynomial $H_p$ is strictly decreasing on $[0,1]$ with $H_p(0)=p(p-1)>0$, and $H_p(1)=2\ell_p$ changes sign at $p_{\mathrm c}$. At $p=2$, $H_2\equiv2$. For $p<p_{\mathrm c}$ let $r_{\mathrm c}(p)\in(0,1)$ be its zero, and let $r_{\mathrm v}(p)\in(0,r_{\mathrm c}(p))$ be the unique root of $Q_p(r)=1$ there; explicitly $r_{\mathrm v}$ solves $2r^{p-1}[p+(2-p)r^2]=(1+r)^2$. Put $K_{\mathrm v}(p)=(1+r_{\mathrm v})/(1-r_{\mathrm v})$ and $K_{\mathrm c}(p)=(1+r_{\mathrm c})/(1-r_{\mathrm c})$. These functions govern everything that follows; Lemma~\ref{flux-dictionary} below explains how, and Figure~\ref{fig-cutoffs} shows the two cutoffs.

\begin{figure}[H]
\centering
\includegraphics[width=0.78\textwidth]{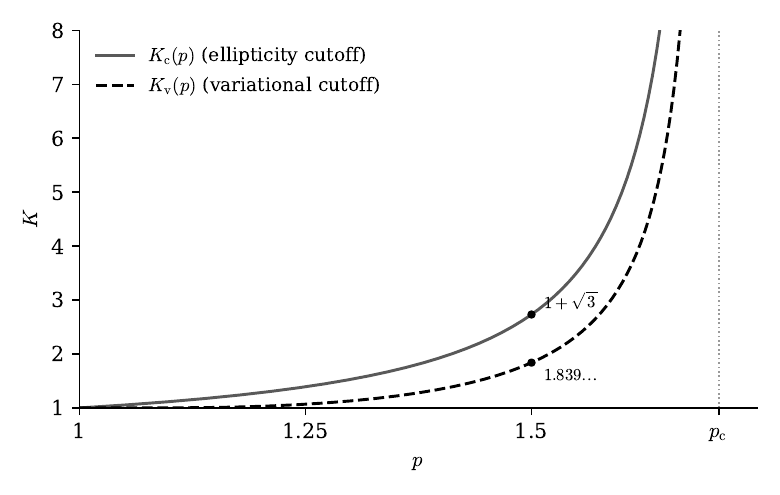}
\caption{The ellipticity cutoff $K_{\mathrm c}$ and the smaller variational cutoff $K_{\mathrm v}$ for $1<p<p_{\mathrm c}$. Both tend to $1$ as $p\downarrow1$ and to $\infty$ as $p\uparrow p_{\mathrm c}$. At $p=3/2$ they are $1+\sqrt3$ and the tribonacci constant $1.839\ldots$, respectively.}
\label{fig-cutoffs}
\end{figure}

\subsection{Main results}
A \emph{barrier} for the inverse trace $\beta$ is a map in $\mathcal M_\beta$ with finite energy; by \cite[Lemma 5.8]{MYWP} the inverse of any finite-distortion homeomorphic extension $F$ of $\beta^{-1}$ with $\I_p(F)<\infty$ is one, and then $\beta$ is automatically quasisymmetric \cite[Proposition 5.9]{MYWP}. A barrier of this kind is a \emph{finite-distortion barrier}; it is a \emph{quasiconformal barrier} if $F$ is quasiconformal. By \cite[Theorem E]{MYWP}, a finite-distortion barrier exists if and only if a quasiconformal barrier exists, since both are equivalent to $\beta^{-1}\in T^p$. Throughout, ``minimiser'' means a minimiser of $\E_p$ on $\mathcal M_\beta$ unless a smaller class is named.

\begin{mainthm}[Thresholds, existence and the lower range]\label{thm-A}
Let $1<p\leq2$.
\begin{enumerate}[label=\textup{(\roman*)}]
\item The inverse density $W_p(a,b)=|a|^{2-p}|b|^p$ on $\{|b|<|a|\}$ is polyconvex if and only if $p\geq p_{\mathrm c}$.
\item For $p_{\mathrm c}\leq p\leq2$ and every quasisymmetric $\beta$ with a barrier, $\E_p$ attains its minimum on $\mathcal M_\beta$. Every minimiser is stationary under compactly supported changes of domain and has holomorphic Ahlfors--Hopf differential
\begin{equation}\label{intro-hopf}
 \Phi_p=\tfrac12\eta(h)\bigl[p+(2-p)r^2\bigr]r^{p-2}h_w\overline{h_{\bw}},\qquad r=|h_{\bw}|/|h_w|,
\end{equation}
and every minimising sequence has a subsequence converging uniformly on $\overline\D$ and strongly in $W^{1,2}_{\loc}(\D)$ to a minimiser. This includes $p=p_{\mathrm c}$.
\item For $1<p<p_{\mathrm c}$, every minimiser $h$ of $\E_p$ in $\mathcal M_\beta$, or in any subclass preserved by compactly supported smooth changes of domain, satisfies $r\in[0,r_{\mathrm v}(p)]\cup\{1\}$ almost everywhere, with $r=0$ where $Dh=0$. Consequently an attained minimum in $\mathcal M_\beta$ is of exactly one of two kinds: either $h$ has finite distortion, and is then a $K_{\mathrm v}(p)$-quasiconformal $C^{2,\alpha}_{\loc}$ diffeomorphism for every $\alpha<\min\{1,(2-p)/(p-1)\}$, smooth off its conformal set; or $Dh$ has rank one on a set of positive area. There are smooth boundary data with finite-energy quasiconformal barriers for which only the second alternative is possible.
\item For $0<p\leq1$, every minimiser which has finite distortion, or is a homeomorphism with $h^{-1}\in W^{1,1}_{\loc}$, is a M\"obius transformation. Non-M\"obius smooth data with smooth quasiconformal barriers exist, so the homeomorphic problem with Sobolev inverse has no minimiser for them.
\end{enumerate}
\end{mainthm}

The threshold $p_{\mathrm c}$ also governs scalar ellipticity and convexity of the principal-stretch profile; the flux dictionary below relates these three calculations.

In the strict upper range, let $\mathfrak B_p$ denote the orientation-preserving circle homeomorphisms $e^{i\theta}\mapsto e^{i\varphi(\theta)}$ with $\varphi$ absolutely continuous, $\varphi'>0$ a.e. and $\log\varphi'\in B^{1/p}_{p,p}(\IS)$; by Wei--Matsuzaki \cite{WeiMatsuzakiIntegrable}, $\mathfrak B_p=T^p$, the $p$-integrable Teichm\"uller space of Guo \cite{Guo}.

\begin{mainthm}[The strict upper range]\label{thm-B}
Let $p_{\mathrm c}<p<2$ and let $f_0$ be an orientation-preserving circle homeomorphism. The following are equivalent:
\begin{enumerate}[label=\textup{(\alph*)},leftmargin=*,itemsep=1pt]
\item $f_0\in T^p=\mathfrak B_p$;
\item $f_0$ has a finite-distortion homeomorphic extension $G:\D\to\D$ with $\I_p(G)<\infty$;
\item $f_0$ has a quasiconformal extension of finite energy $\I_p$;
\item $f_0$ has a quasiconformal diffeomorphic extension of finite energy which minimises $\I_p$ among all finite-distortion homeomorphic extensions of $f_0$.
\end{enumerate}
The equivalence of \textup{(a)} and \textup{(b)} is \cite[Theorem E]{MYWP}; the variational content is \textup{(c)}$\Rightarrow$\textup{(d)}. When these hold, the minimiser $F_{(p)}$ in \textup{(d)} is unique, and its inverse $H_{(p)}$ is the unique minimiser of $\E_p$ in the full monotone class $\mathcal M_{f_0^{-1}}$; in particular the monotone, finite-distortion homeomorphic and quasiconformal inverse infima agree, and $\I_p(F_{(p)})\leq\I_p(G)$ for every extension $G$ as in \textup{(b)}, with equality only for $G=F_{(p)}$. Both maps are $C^{2,\alpha}_{\loc}$ diffeomorphisms for $0<\alpha<\min\{1,(2-p)/(p-1)\}$, real analytic off their conformal sets, globally quasiconformal and asymptotically conformal. The Hopf differential $\Phi_p[f_0]$ of $H_{(p)}$ is holomorphic and, with $q=p/(p-1)$,
\begin{equation}\label{intro-hopf-norm}
 C^{-1}\,\I_p(F_{(p)})\leq\int_\D|\Phi_p[f_0]|^q\eta^{1-q}\,dA\leq C\,\I_p(F_{(p)}),
\end{equation}
where $C$ depends only on $p$ and the dilatation of $F_{(p)}$. For fixed $f_0\in T^1$, the maps $F_{(p)},H_{(p)}$ and the differentials $\Phi_p[f_0]$ are right-continuous in $p$ on $(p_{\mathrm c},2)$, and continuous at every continuity point of the minimum energy $p\mapsto\I_p(F_{(p)})$.
\end{mainthm}

The proof has three stages. Polynomial truncation and parameter selection produce regular minimisers for the barrier density $|\mu|^p/(\lambda-|\mu|^2)$, $\lambda<1$. Uniform scalar gradient estimates on common Hopf patches give a diffeomorphic limit as $\lambda\uparrow1$. A weighted Reich--Strebel comparison then identifies this limit with the unique absolute monotone minimiser. The comparison uses convexity of the principal-stretch profile and weighted holomorphic approximation; it remains valid at $p=p_{\mathrm c}$, where the construction of a regular reference map is still missing.

The next theorem is the general mechanism behind the quasiconformality assertions: once a minimiser is regular, a scalar equation for the distortion gives global quasiconformality and asymptotic conformality from finite energy alone.

\begin{mainthm}[Regular minimisers are quasiconformal]\label{thm-C}
Let $1<p<\infty$. Suppose $h:\overline\D\to\overline\D$ is an orientation-preserving homeomorphism with quasisymmetric boundary values, is a $C^1$ diffeomorphism in $\D$ and is $C^2$ on $\{|\mu_h|>0\}$, and that $h$ is a finite-energy minimiser of $\E_p$ in a class preserved by compactly supported smooth changes of domain. Then $h$ and $f=h^{-1}$ are globally quasiconformal and asymptotically conformal. For $1<p<p_{\mathrm c}$ there is the bound $K_h=K_f\leq K_{\mathrm v}(p)$, depending only on $p$. For $p>2$ there is instead the energy-dependent estimate
\begin{equation}\label{early-superquadratic-qc-bound}
 \|\K_f-1\|_{L^\infty(\D)}\leq C_p\I_p(f)^{2/p},
\end{equation}
with $C_p$ depending only on $p$; hence $K_f\leq1+\Lambda+\sqrt{\Lambda(\Lambda+2)}$ with $\Lambda=C_p\I_p(f)^{2/p}$. No bound depending only on $p$ holds for $p_{\mathrm c}<p\leq2$.
\end{mainthm}

For $1<p\leq2$ the scalar equation is a divergence-form equation for $r=|\mu_h|$ in a natural coordinate of the Hopf differential, whose conductivity has the eigenvalues $H_p(\mp r)/(2r[p+(2-p)r^2])$; it is elliptic exactly where $H_p(r)>0$, and the explicit iteration of Section~\ref{scalar-equations} handles the linear degeneracy at the critical exponent. For $p>2$ the natural variable is different.

\begin{mainthm}[The uniformly elliptic equation above two]\label{thm-D}
Let $p>2$ and let $f$ be as in Theorem~\textup{\ref{thm-C}}. Then $\sigma=\log\bigl(|\mu_f|/\sqrt{1-|\mu_f|^2}\bigr)$ satisfies, weakly on $\{\mu_f\ne0\}$,
\[
 -\operatorname{div}(\mathsf B_p\nabla\sigma)=8\eta,\qquad (p-2)I\leq\mathsf B_p\leq(p+2)I,
\]
with $\mathsf B_p$ the explicit matrix \eqref{superquadratic-forward-log-equation}, and for every $z\in\D$,
\[
 \K_f(z)-1\leq C_p\left(\int_{\mathcal B_z}\frac{|\mu_f|^p}{(1-|\mu_f|^2)^{p/2}}\,dA_\eta\right)^{2/p},
\]
where $\mathcal B_z$ is the hyperbolic ball of fixed radius about $z$.
\end{mainthm}

For $p_{\mathrm c}<p<2$, it remains open whether a finite-energy monotone barrier forces the boundary data to lie in $T^p$. Further questions are regularity at $p=p_{\mathrm c}$ and attainment in the monotone class for $1<p<p_{\mathrm c}$. For the canonical family we prove right-continuity in the exponent; full continuity remains conditional on the recovery estimate \eqref{left-exponent-recovery}.

\subsection{A dictionary for the scalar flux}
The functions in \eqref{flux-functions} enter through three different calculations, and the following lemma records how they are related. It is proved in the sections indicated.

\begin{lemma}[Flux, ellipticity and variational selection]\label{flux-dictionary}
For $1<p\leq2$ and $0<r<1$,
\begin{equation}\label{flux-log-derivative}
 \frac{Q_p'}{Q_p}
 =\frac{(1-r)H_p(r)}{r(1+r)[p+(2-p)r^2]},
 \qquad Q_p(1)=1,\quad Q_p'(1)=0,\quad
 (\log Q_p)''(1)=-\ell_p/2.
\end{equation}
With the scalar parameter and conductivity defined by
\[
 t=\Theta_p(r)=\frac{\sqrt{F_p(r)}}{1-r},\qquad
 A_p(t)=\frac{1-r}{1+r},
\]
one has $(2tA_p(t))^2=Q_p(r)$.
The forward scalar equation is elliptic exactly where $H_p(r)>0$,
or equivalently where its flux $tA_p(t)$ is increasing.
For $p\geq p_{\mathrm c}$ this holds at every $r<1$, and
$Q_p(r)<1$. The same range is exactly the polyconvexity range of
the inverse density.

For $1<p<p_{\mathrm c}$, $H_p$ has one zero
$r_{\mathrm c}\in(0,1)$, whereas $Q_p(r)=1$ has one solution
$r_{\mathrm v}\in(0,r_{\mathrm c})$ besides the endpoint $1$.
Minimality forces $Q_p(r)\leq1$ at positive-Jacobian points.
Thus the excluded interval is $(r_{\mathrm v},1)$; the interval
where scalar ellipticity fails is the strictly smaller
$(r_{\mathrm c},1)$.
\end{lemma}
\begin{proof}
Differentiate \eqref{flux-functions}. The endpoint formula follows
from $H_p(1)=2\ell_p$. For $p<2$, $H_p$ is strictly decreasing;
$H_p(0)>0$ and the sign of $H_p(1)$ changes at $p_{\mathrm c}$.
These facts give the monotonicity and the two cutoffs, since
$Q_p(0)=0$ and $Q_p(1)=1$. The scalar eigenvalue identity
\eqref{forward-scalar-eigenvalues} gives the ellipticity assertion;
Corollary~\ref{sharp-polyconvexity-range} and
Theorem~\ref{weak-variational-cutoff} supply the two variational
assertions. At $p=2$ the same identities hold directly.
\end{proof}
\section{Preliminaries}\label{preliminaries}
\subsection{Conventions}\label{conventions}
We keep the conventions of Part~I \cite[Section 1]{MYWP}, which we recall briefly. Let $\ID$ be the unit disk, $\IS=\partial\ID$, and $dA$ Euclidean area. The Beltrami equation is $f_{\bar z}=\mu_ff_z$ with $|\mu_f|<1$ almost everywhere. A mapping of finite distortion is a $W^{1,1}_{\loc}$ map with locally integrable nonnegative Jacobian and $Df=0$ almost everywhere on $\{J_f=0\}$; at points where $Df=0$ we set $\mu_f=0$. We use the normalised Hilbert--Schmidt norm and distortion
\begin{equation}\label{db}
 \|Df\|^2=|f_z|^2+|f_{\bar z}|^2,\qquad
 \IK_f=\frac{1+|\mu_f|^2}{1-|\mu_f|^2},\qquad
 \|Df\|^2=\IK_f J_f,\qquad \IK_f-1=\frac{2|\mu_f|^2}{1-|\mu_f|^2},
\end{equation}
and the operator-norm distortion $K_f=(1+|\mu_f|)/(1-|\mu_f|)$; $K$ and $\K$ always denote these two quantities. ``Homeomorphism of finite distortion'' means orientation preserving. A continuous surjection of the closed disk is \emph{monotone} if every point fibre is connected; a \emph{hair} is a fibre containing more than one point. The hyperbolic density is
\begin{equation}\label{eta}
 \eta(z)=(1-|z|^2)^{-2},\qquad dA_\eta=\eta(z)\,dA(z),
\end{equation}
so that $\eta|dz|^2$ has curvature $-4$; where a formula is best written for curvature $-1$ we write $g_{\mathbb H}=4\eta|dz|^2$ explicitly. Quasisymmetry of a circle homeomorphism $e^{i\theta}\mapsto e^{iu(\theta)}$ is the adjacent-interval condition $M^{-1}\leq[u(\theta+t)-u(\theta)]/[u(\theta)-u(\theta-t)]\leq M$ for $0<t\leq\pi$, and symmetry means that this ratio tends uniformly to $1$ as $t\downarrow0$. Boundary traces of inverse maps $h$ are denoted $\beta$, those of forward maps $f$ are $\beta^{-1}$.

\paragraph{The energies.}
Set $p_*=\max\{2,p\}$ and $R_p(s)=s^{p/2}(1-s)^{-p_*/2}$ for $0\leq s<1$. For a map $F$ of finite distortion,
\begin{equation}\label{funct}
 \mathcal I_p(F)=\int_\ID R_p(|\mu_F|^2)\,dA_\eta,
\end{equation}
so that $\I_p(F)=2^{-p/2}\int_\D(\K_F-1)^{p/2}\,dA_\eta$ for $p\geq2$. Under the inverse and area formulas, $H=F^{-1}$ satisfies
\begin{equation}\label{inverseenergy}
 \mathcal I_p(F)=\mathcal E_p(H)
 :=\int_\ID\eta(H)J_H R_p(|\mu_H|^2)\,dA,
\end{equation}
and, writing $a=H_w$, $b=H_{\bar w}$, $J=|a|^2-|b|^2$,
\begin{equation}\label{inversedensity}
 \mathcal E_p(H)=
 \begin{cases}
 \displaystyle\int_\ID\eta(H)|b|^p|a|^{2-p}\,dA,&1\leq p\leq2,\\[4pt]
 \displaystyle\int_\ID\eta(H)|b|^p J^{1-p/2}\,dA,&p\geq2.
 \end{cases}
\end{equation}
A \emph{finite-energy barrier} is a competitor with the prescribed boundary values and finite energy.

\paragraph{M\"obius invariance and the inverse Sobolev estimate.}
Both energies are invariant under pre- and postcomposition by disk automorphisms, $\mathcal I_p(\psi\circ F\circ\phi)=\mathcal I_p(F)$, since $|\mu_{\psi\circ F\circ\phi}|=|\mu_F|\circ\phi$ and $\eta(\phi)|\phi'|^2=\eta$; the prescribed boundary data transform along with the map. The energies are ordered in $p$: since $0\leq r<1$,
\begin{equation}\label{subquadratic-energy-order}
 \frac{r}{1-r^2}\geq\frac{r^p}{1-r^2}\geq\frac{r^2}{1-r^2},
 \qquad 1\leq p\leq2,
\end{equation}
so $\I_1(F)\geq\I_p(F)\geq\I_2(F)$, and correspondingly $|a||b|\geq|a|^{2-p}|b|^p\geq|b|^2$ for $|b|\leq|a|$. Thus bounded $p$-energy controls the quadratic anti-holomorphic energy for every $1\leq p\leq2$, and with the pointwise identity $(\K_H-1)J_H=2|H_{\bar w}|^2$ and $\int_\D J_H\,dA=\pi$ this gives the inverse Sobolev estimate
\begin{equation}\label{sobolev}
 \int_\ID(|H_w|^2+|H_{\bar w}|^2)\,dA\leq\pi+2\mathcal E_p(H)\qquad(1\leq p\leq2),
\end{equation}
with $3\pi$ in place of $\pi$ for $p\geq2$; see \cite[Section 1.4]{MYWP}. It does not give a forward $W^{1,2}$ bound on $F$. For quasisymmetric boundary data \eqref{sobolev} gives the equicontinuity and monotone compactness recalled in Section~\ref{monotone-class}; the threshold $p_{\mathrm c}$ below concerns lower semicontinuity of the energy being minimised, not this Dirichlet bound.
\subsection{The monotone class and the results of Part~I that are used}\label{monotone-class}
Fix an orientation-preserving quasisymmetric boundary homeomorphism $\beta$. For $0<p\leq2$, let $\mathcal M_\beta$ consist of maps $h$ such that
\begin{enumerate}
\item $h:\overline\D\to\overline\D$ is a continuous monotone surjection, $h|_{\Sone}=\beta$, and $h(\D)\subset\D$;
\item $h\in W^{1,2}(\D)$ and $J_h\geq0$ almost everywhere;
\item $\E_p(h)=\int_\D\eta(h)W_p(h_w,h_{\bw})\area<\infty$, where $W_p(a,b)=|a|^{2-p}|b|^p$ on $|b|\leq|a|$, and $W_p(0,0)=0$.
\end{enumerate}
The definition permits nonzero rank-one gradients: when $|a|=|b|$, $W_p(a,b)=|a|^2$. Such a map need not have finite distortion, which requires $Dh=0$ on $J_h=0$. The following results are proved in Part~I; we restate them in the form used here. Write $\mathcal D(u)=\int_\D(|u_w|^2+|u_{\bw}|^2)\area$.

\begin{proposition}[Compactness and inverse calculus from Part~I]
\label{boundary-modulus}\label{no-boundary-hairs}\label{monotone-compactness}\label{compact-containment}\label{finite-distortion-inverse-calculus}
The following statements hold.
\begin{enumerate}[label=\textup{(\roman*)},leftmargin=*]
\item If continuous monotone surjections $u_j:\overline\D\to\overline\D$ in $W^{1,2}(\D)$ have bounded $\mathcal D(u_j)$ and their boundary maps admit $K_0$-quasiconformal extensions fixing $1,i,-i$, then
\begin{equation}\label{no-hairs-log-modulus}
 |u_j(x)-u_j(y)|^2\leq
 \frac{4\pi(\mathcal D(u_j)+C_0)}{\log(R_0/|x-y|)}
 \qquad(|x-y|<R_0),
\end{equation}
where $R_0,C_0$ depend only on $K_0$ \cite[Lemma 2.1]{MYWP}.
\item \emph{NoHairs.} Let $h:\overline\D\to\overline\D$ be a continuous monotone surjection in $W^{1,2}(\D)$ with quasisymmetric boundary trace $\beta$. If
\begin{equation}\label{orbit-bound}
 \|\psi\circ h\circ\phi^{-1}\|_{W^{1,2}(\D)}\leq C_*
\end{equation}
for every pair of disk automorphisms for which $\psi\circ\beta\circ\phi^{-1}$ fixes $1,i,-i$, then $h^{-1}(\Sone)=\Sone$ \cite[Theorem 2.2]{MYWP}.
\item For $0<p\leq2$ and $h\in\mathcal M_\beta$,
\begin{equation}\label{trace-jacobian}
 \int_\D J_h\,dA=\pi,\qquad
 \mathcal D(h)\leq\pi+2\E_p(h),
\end{equation}
and $\E_p$ is invariant under pre- and postcomposition by disk automorphisms. Every sequence in $\mathcal M_\beta$ with bounded energy has a subsequence converging uniformly on $\overline\D$ and weakly in $W^{1,2}(\D)$ to a continuous monotone surjection $h$ with trace $\beta$, $J_h\geq0$ and $h^{-1}(\Sone)=\Sone$. For each $C\Subset\D$, all $h_j(C)$ lie in a common compact subset of $\D$. This does not yet assert finite energy of the limit. The same compactness and containment conclusions hold for $W^{1,2}$ homeomorphisms with fixed quasisymmetric trace and bounded $\E_p$, for every $p\geq1$ \cite[Lemma 2.5 and Corollary 2.3]{MYWP}.
\item If $F:\D\to\D$ is a finite-distortion homeomorphism with $\I_p(F)<\infty$, $p\geq1$, then $h=F^{-1}\in W^{1,2}(\D)$ has finite distortion. For $1\leq p\leq2$,
\begin{equation}\label{finite-distortion-inverse-energy}
 \mathcal D(h)\leq\int_\D\K_F\,dA\leq\pi+2\I_p(F),
 \qquad \E_p(h)\leq\I_p(F).
\end{equation}
Thus $h\in\mathcal M_\beta$ for its boundary trace $\beta$ in this range \cite[Lemma 5.8]{MYWP}.
\end{enumerate}
\end{proposition}

\begin{lemma}[From local to global invertibility]\label{local-to-global}\label{hair-homeomorphism}
Let $h:\overline\D\to\overline\D$ be a continuous monotone surjection whose boundary trace is a homeomorphism. If $h$ is a local homeomorphism on $\D$, it is a homeomorphism of $\overline\D$. The same conclusion holds if $h^{-1}(\Sone)=\Sone$ and there is a continuous $g:\D\to\D$ with $h\circ g=\id_\D$; then $g=h^{-1}$. A local $C^k$ diffeomorphism in the first case is a global $C^k$ diffeomorphism on $\D$.
\end{lemma}
\begin{proof}
In the first case invariance of domain gives $h(\D)\subset\D$. Every interior fibre is connected by monotonicity and discrete by local injectivity, hence a singleton. The boundary fibres are singletons as well. In the second case $g$ is injective, so $g(\D)$ is open. If $g(z_j)\to w\in\D$, then $z_j=h(g(z_j))\to h(w)\in\D$ and continuity gives $w=g(h(w))$. Thus $g(\D)$ is also closed in $\D$, hence equals $\D$. In both cases compactness gives a homeomorphism of the closed disk; the inverse function theorem gives the final assertion.
\end{proof}

For $p\geq1$, let $T^p$ denote the circle homeomorphisms admitting a quasiconformal extension $F$ with $\mu_F\in L^p(dA_\eta)$ \cite{Guo,Cui,TT}, and let $\mathcal T_p$ denote those admitting a finite-distortion homeomorphic extension with finite $\I_p$.
\begin{proposition}[Boundary classes and barriers from Part~I]
\label{finite-energy-boundary-moduli}\label{smooth-barrier}\label{finite-distortion-boundary-class}
A finite-distortion homeomorphism $F$ with $\I_p(F)<\infty$, $p\geq1$, extends homeomorphically to the closed disk; its trace is quasisymmetric and symmetric, with modulus bounds depending only on $p$ and $\I_p(F)$ \cite[Proposition 5.9]{MYWP}. For $p>1$, $\mathcal T_p=T^p$: after three-point normalisation, the trace has a quasiconformal extension with $L^p(dA_\eta)$ Beltrami norm bounded in terms of $p$ and $\I_p(F)$ \cite[Theorem E]{MYWP}. Consequently a finite-distortion barrier exists exactly when a quasiconformal barrier exists. Every smooth orientation-preserving circle diffeomorphism admits a smooth quasiconformal extension with $|\mu_F(z)|=O(1-|z|)$, hence finite $\I_p$ for every $p>1$ \cite[Lemma 4.8]{MYWP}.
\end{proposition}

We also use the quadratic existence theorem of Part~I:

\begin{theorem}[{\cite[Theorem A]{MYWP}}]\label{quadratic-wp-existence}
Let $\beta$ be an orientation-preserving Weil--Petersson circle homeomorphism. The inverse energy $\E_2$ has a unique minimiser in $\mathcal M_\beta$. This minimiser is a real analytic hyperbolic harmonic diffeomorphism, is globally quasiconformal, and is asymptotically conformal.
\end{theorem}

\subsection{Inner variations and the weighted Ahlfors--Hopf differential}\label{inner-variations}
Let $p>1$. Suppose $H\in W^{1,2}(\ID)$ is a homeomorphism of finite distortion with $J_H>0$ almost everywhere, $\mathcal E_p(H)<\infty$, and is stationary under compactly supported smooth changes of domain. Every minimiser in this class is stationary under these variations.

For $\varphi\in C_c^\infty(\ID,\IC)$ set $g_t(w)=w+t\varphi(w)$ and
\[
 H_t=H\circ g_t^{-1}.
\]
For sufficiently small real $t$, $g_t$ is a diffeomorphism with
\[
 J_{g_t}=|1+t\varphi_w|^2-t^2|\varphi_{\bar w}|^2>0.
\]
After changing variables by $g_t$, the factors $\eta(H)$ and $J_H$ in \eqref{inverseenergy} are fixed. The chain rule gives, with $s=|\mu_H|^2$,
\[
 s_t:=|\mu_{H_t}(g_t(w))|^2
 =\frac{|\mu_H(1+t\varphi_w)-t\varphi_{\bar w}|^2}
 {|1+t\varphi_w-\overline{\mu_H}t\varphi_{\bar w}|^2},
\]
and hence
\[
 \dot s_0=-2(1-s)\operatorname{Re}(\overline{\mu_H}\varphi_{\bar w}).
\]
Define
\begin{equation}\label{ap}
 a_p(s)=(1-s)^2R_p'(s)=
 \begin{cases}
 s^{p/2-1}\bigl[\frac p2+(1-\frac p2)s\bigr],&1<p\leq2,\\[2pt]
 \frac p2\bigl(\frac{s}{1-s}\bigr)^{p/2-1},&p\geq2.
 \end{cases}
\end{equation}
Although $a_p$ itself diverges at zero when $1<p<2$, the product below extends by zero there, since its magnitude vanishes like $|\mu_H|^{p-1}$. Differentiation yields
\begin{equation}\label{variation}
 \left.\frac{d}{dt}\mathcal E_p(H_t)\right|_{t=0}
 =-2\operatorname{Re}\int_\ID\Phi_p\varphi_{\bar w}\,dA,
 \qquad
 \Phi_p=\eta(H)a_p(s)H_w\overline{H_{\bar w}}.
\end{equation}
Here the differentiation can be justified locally by an integrable bound. For $p\geq2$, the transformed ratio $|H_{t,\bar w}|^2/J_{H_t}$ and its $t$-derivative are bounded, after the change of variables, by a constant times $1+|b|^2/J$. The differentiated density is therefore controlled on the variation support by $C\eta(H)(J+J R_p(s))$. For $1<p<2$, use \eqref{inversedensity}: its derivative with respect to $(a,b)$ is bounded by $C(|a|+|b|)$ on $|b|\leq|a|$, including $b=0$. Multiplying by the derivative of the transformed gradient gives a bound $C\eta(H)(|a|^2+|b|^2)$. The support is compact, $\eta(H)$ is bounded there, and $H\in W^{1,2}$.

\begin{theorem}[Conditional variational conclusion]\label{hopf}
Under the preceding hypotheses, $\Phi_p\in L^1_{\loc}(\ID)$ is holomorphic. In particular, for $p\geq2$, the coefficient
\begin{equation}\label{Phifinal}
 \eta(H)\left(\frac{|\mu_H|^2}{1-|\mu_H|^2}\right)^{p/2-1}
 H_w\overline{H_{\bar w}}
\end{equation}
is holomorphic.
\end{theorem}
\begin{proof}
On compact subsets,
\[
 |\Phi_p|\leq C_p\eta(H)J_H\bigl(1+R_p(s)\bigr),
\]
by splitting into $s\leq1/2$ and $s>1/2$. Thus it is locally integrable. Stationarity and the replacement of $\varphi$ by $i\varphi$ give $\int\Phi_p\varphi_{\bar w}\,dA=0$. Weyl's lemma implies holomorphy. Formula~\eqref{Phifinal} differs by the nonzero constant $p/2$.
\end{proof}

At $p=2$, \eqref{Phifinal} is the hyperbolic Hopf differential. At $p=1$ the density is not differentiable at $H_{\bar w}=0$ so an endpoint treatment needs a subgradient or variational inequality there. We expect there are no ordinary Euler--Lagrange equations there.

Theorem~\ref{hopf} concerns stationary homeomorphisms. Theorem~\ref{monotone-existence} proves existence and the same Hopf equation for a monotone inverse minimiser when $p_{\mathrm c}\leq p\leq2$. Below $p_{\mathrm c}$, Proposition~\ref{lower-range-boundary-obstruction} prevents an unrestricted diffeomorphic-minimiser conclusion.
\section{Convexity and the threshold \texorpdfstring{$p_{\mathrm c}$}{p-c}}\label{convexity}
The direct method for $p\geq2$ rests on convexity of the inverse density in the derivative entries and the determinant. This section determines exactly when that convexity is available for $p<2$, and shows what fails below the threshold.

\subsection{Polyconvexity of the inverse density}
For fixed target value and $p\geq2$, the function
\[
 (b,J)\longmapsto |b|^p J^{1-p/2},\qquad J>0,
\]
is convex. To see this, set $r=p/2\geq1$: $t^rJ^{1-r}$ is the perspective of the convex function $t^r$, is increasing in $t\geq0$, and may be composed with $t=|b|^2$. This gives the relevant convexity in derivative entries and the determinant. It does not by itself establish compactness in the homeomorphism class or the existence of a minimiser.

For $1\leq p<2$ the inverse density is different. The choice $p_*=2$ does \emph{not} guarantee polyconvexity throughout this range. The distinction can be checked directly in distortion variables. With the norm convention \eqref{db}, the function
\[
 (A,t)\longmapsto\frac{\|A\|^2}{t},\qquad t>0,
\]
is convex, so $\K(A)=\|A\|^2/\det A$ is polyconvex. However, with $r=|\mu|$, the present forward density is
\begin{equation}\label{distortion-composition-test}
 \frac{r^p}{1-r^2}=\Psi_p(\K),\qquad
 \Psi_p(S)=\frac12(S-1)^{p/2}(S+1)^{1-p/2},\qquad
 \frac{\Psi_p''(S)}{\Psi_p(S)}
 =\frac{p(p-2)}{(S^2-1)^2}.
\end{equation}
Thus $\Psi_p$ is increasing but strictly concave for $1<p<2$, and the convex nondecreasing composition rule cannot be applied. This differs from the convex outer function $S\mapsto S^m$, $m>1$, in the $L^m$-distortion problem. For that problem the inverse density $\|A\|^{2m}/t^{m-1}$ is a convex perspective. Here it is instead $t\Psi_p(\|A\|^2/t)$ on $t=\det A$. Concavity of $\Psi_p$ alone does not of course disprove polyconvexity; the rank-one test below decides the lower range.

On diagonal matrices $\mathrm{diag}(\lambda,\nu)$, $\lambda>\nu>0$, the inverse derivative density is
\[
 W_p(\lambda,\nu)=\frac{(\lambda-\nu)^p(\lambda+\nu)^{2-p}}4.
\]
For $p=1$, $W_1=(\lambda^2-\nu^2)/4$, which is concave in the rank-one direction varying $\nu$. Even away from the endpoint,
\begin{equation}\label{rankone}
 \left.\frac{d^2}{d\nu^2}W_{3/2}(3,\nu)\right|_{\nu=1}
 =-\frac{1}{16\,2^{3/2}}<0.
\end{equation}
Thus the direct-method argument for $p\geq2$ cannot simply be carried over to all smaller exponents.

This is a test of polyconvexity, not merely convexity in the matrix entries: along $A(\nu)=\operatorname{diag}(3,\nu)$, both $A(\nu)$ and $\det A(\nu)=3\nu$ depend affinely on $\nu$. A density convex in the minors must therefore have a convex restriction to this path. Equation~\eqref{rankone} violates that necessary condition from Ball's framework \cite{BallPolyconvex}. Multiplying the density by a positive frozen target weight, whether hyperbolic, bounded or constant, preserves the negative sign.

There is nevertheless a positive result in the upper subquadratic range. Set
\begin{equation}\label{convexity-threshold}
 p_{\mathrm c}=1+\frac1{\sqrt2},\qquad
 c_p=\frac{p-1}{2},\qquad \ell_p=2(p-1)^2-1.
\end{equation}
We regard $a,b\in\IC$ as four real gradient variables, with $J(a,b)=|a|^2-|b|^2$.

\begin{lemma}[Convexity after adding the Jacobian]\label{jacobian-convexification}
For $p_{\mathrm c}\leq p\leq2$, define
\begin{equation}\label{convex-extension}
 G_p(a,b)=
 \begin{cases}
 |a|^{2-p}|b|^p+c_p(|a|^2-|b|^2),&|b|\leq|a|,\\[2pt]
 \tfrac12(|a|^2+|b|^2),&|b|\geq|a|,
 \end{cases}
 \qquad G_p(0,0)=0.
\end{equation}
Then $G_p$ is a strictly convex $C^1$ function on $\IC^2$, of quadratic growth. More precisely,
\[
 G_p(a,b)-\frac{\ell_p}{2}(|a|^2+|b|^2)
\]
is convex. Consequently, on $J\geq0$, the inverse density satisfies
\begin{equation}\label{polyconvex-inverse}
 W_p(a,b)=|a|^{2-p}|b|^p=G_p(a,b)-c_pJ(a,b).
\end{equation}
Thus $W_p$ is the restriction of a polyconvex density and quadratic convexity constant is positive for $p>p_{\mathrm c}$; strict convexity also holds at $p=p_{\mathrm c}$.
\end{lemma}
\begin{proof}
Write $x=|a|$, $y=|b|$, $d=p-1$, $q=2-p$, and first take $0<r=y/x<1$. For
\[
 g_p(x,y)=x^{2-p}y^p+\tfrac d2(x^2-y^2)
\]
the scalar Hessian satisfies
\begin{align*}
 (g_p)_{xx}&=d(1-q r^p)\geq d^2,\\
 (g_p)_{yy}&=d(p r^{p-2}-1)\geq d^2,\\
 0\leq(g_p)_{xy}&=pq r^{p-1}\leq pq.
\end{align*}
Hence $D^2g_p\geq(d^2-pq)I=\ell_p I$. Its first derivatives obey
\[
 \frac{(g_p)_x}{x}=q r^p+d\geq d,
 \qquad
 \frac{(g_p)_y}{y}=p r^{p-2}-d\geq1.
\]
On $y=x$, the value is $x^2$ and the gradient is $(x,x)$. These match the exterior expression $(x^2+y^2)/2$. The derivatives extend continuously at $y=0$ for $p>1$ and at the origin by homogeneity. Restricting to line segments across the interface therefore shows convexity on the whole nonnegative quadrant, also after subtracting $\ell_p(x^2+y^2)/2$.

Since $0\leq\ell_p\leq d\leq1$, this corrected scalar function is nondecreasing in each coordinate. Composing it with the two Euclidean norms preserves convexity, by the triangle inequality followed by coordinatewise monotonicity and scalar convexity. This proves the assertion on all of $\IC^2$, not just on the nonconvex set $|b|\leq|a|$. Quadratic growth and the $C^1$ assertion follow directly from the formulas.

It remains to check strict convexity at the critical exponent. Here $d^2=pq=1/2$, and for $0<r<1$ the three Hessian bounds above are strict: $(g_p)_{xx}>1/2$, $(g_p)_{yy}>1/2$, and $(g_p)_{xy}<1/2$. The radial Hessian is therefore positive definite. The two tangential eigenvalues $(g_p)_x/x$ and $(g_p)_y/y$ are also positive. Hence the full Hessian of $G_p$ is positive definite on $0<|b|<|a|$; on $|b|>|a|$ it is the identity. Along a line segment, each of $|b|^2=0$ and $|a|^2-|b|^2=0$ holds either identically or at only finitely many parameter values. If neither holds identically, every nonconstant segment has a subsegment on which the Hessian is positive definite, so convexity precludes an affine restriction to the whole segment. If one holds identically, the restriction is respectively $c_p|a|^2$ or $(|a|^2+|b|^2)/2$, a strictly convex quadratic on that nonconstant segment. This proves strict convexity at $p=p_{\mathrm c}$ without a positive quadratic lower bound.
\end{proof}

The correction does not make $\Psi_p(\K)$ convex as a function of $\K$ as that function remains concave for $p<2$. It does give the different convexity in gradient variables and the Jacobian needed for the direct method. Section~\ref{direct-existence} handles the transported weight $\eta(h)$ by localisation. No subtraction of two infinite hyperbolic integrals is involved. The closed scalar forward equation in Section~\ref{forward-scalar-equation} confirms the same threshold directly in its principal eigenvalues.

\begin{corollary}[Sharp polyconvexity range for the inverse density]\label{sharp-polyconvexity-range}
For $1<p\leq2$, the density $W_p(a,b)=|a|^{2-p}|b|^p$ on $\det A>0$ is polyconvex if and only if $p\geq1+1/\sqrt2$.
\end{corollary}
\begin{proof}
Sufficiency is Lemma~\ref{jacobian-convexification}. For necessity, fix $0<r<1$ and differentiate along the rank-one path $A(\nu)=\operatorname{diag}(1+r,\nu)$ at $\nu=1-r$. Direct differentiation gives
\[
 \left.\frac{d^2}{d\nu^2}W_p(A(\nu))\right|_{\nu=1-r}
 =\frac14r^{p-2}
 \bigl[p(p-1)-2p(2-p)r-(p-1)(2-p)r^2\bigr].
\]
At $r=1$ the bracket is $2[2(p-1)^2-1]$. For $p<1+1/\sqrt2$ it is negative, and remains negative at some $r<1$, where the matrix has positive determinant. Polyconvexity is therefore impossible in that range.
\end{proof}

\begin{remark}[Equal leading growth and a negative second variation]\label{rankone-growth-versus-stability}
The common linear growth in distortion does not imply the same variational stability. Put $\rho_p(r)=r^p/(1-r^2)$. For $1<p<2$, concavity of $u^{p/2}$ and its tangent at $u=1$ give
\begin{equation}\label{bounded-forward-density-difference}
 0\leq\rho_p(r)-\rho_2(r)\leq\frac{2-p}{2},\qquad
 \lim_{r\uparrow1}\bigl(\rho_p(r)-\rho_2(r)\bigr)=\frac{2-p}{2},
\end{equation}
so $\rho_p(r)\sim\rho_2(r)\sim\K/2$ as $r\uparrow1$, while near $r=0$ the ratio $\rho_p/\rho_2=r^{p-2}$ diverges. On the inverse side, along the rank-one path $A_t=\operatorname{diag}(1,t)$, $r=(1-t)/(1+t)$,
\[
 W_p(A_t)=\tfrac14(1-t)^p(1+t)^{2-p}
 =\frac14-\frac{p-1}{2}t+\frac{\ell_p}{4}t^2+O(t^3):
\]
The quadratic coefficient $\ell_p/4$ is negative exactly for $1<p<p_{\mathrm c}$, despite the common value at the limiting rank-one matrix.

\end{remark}
\subsection{Why compactness does not settle the lower range}\label{lower-range-energy-defect}
The obstruction below $p_{\mathrm c}$ is more specific than absence of an available convexity proof. With $H_p$ as in \eqref{flux-functions},
\[
 \left.\frac{d^2}{d\nu^2}W_p(1+r,\nu)\right|_{\nu=1-r}
 =\frac14r^{p-2}H_p(r),\qquad
 H_p(1)=2\ell_p<0\quad(1<p<p_{\mathrm c}).
\]
Choosing $r<1$ close enough to one therefore gives a negative rank-one second variation. At $p=3/2$ the explicit choice in \eqref{rankone} suffices.

This produces failure of weak lower semicontinuity even for fixed-boundary smooth disk diffeomorphisms with a smooth diffeomorphic limit. To see it, take a smooth disk diffeomorphism $h_0$, equal to the identity near the circle and affine with derivative $A=\operatorname{diag}(1+r,1-r)$ on a small interior square $U$. Such a map is obtained from the time-one flow of a compactly supported vector field equal to $(\log(1+r)x,\log(1-r)y)$ on a sufficiently small neighbourhood. Choose nonzero real $\chi\in C_c^\infty(U)$ and set
\[
 h_n(x,y)=h_0(x,y)+\left(0,\frac{\epsilon}{n}\chi(x,y)\sin(ny)\right).
\]
For sufficiently small fixed $\epsilon>0$, these are uniformly bi-Lipschitz orientation-preserving diffeomorphisms with the same boundary trace. They converge uniformly and weakly in $W^{1,2}$ to $h_0$, while
\[
 Dh_n=A+\epsilon\chi\cos(ny)\,e_2\otimes e_2+O(n^{-1})
 \quad\hbox{on }U.
\]
The density is strictly concave along this rank-one direction near $A$. Its linear oscillation averages to zero and its quadratic oscillation has a strictly negative average on $\{\chi\ne0\}$. As $\eta(h_n)\to\eta(h_0)$ uniformly on $U$, Taylor's formula and oscillatory averaging give
\[
 \limsup_{n\to\infty}\E_p(h_n)<\E_p(h_0).
\]
Thus weak lower semicontinuity fails below $p_{\mathrm c}$ even along smooth diffeomorphisms with a diffeomorphic limit. This does not settle attainment for any specified boundary data.

\begin{proposition}[A restriction on diffeomorphic minimisers below the threshold]\label{lower-range-boundary-obstruction}
Let $1<p<p_{\mathrm c}$, and let $r_{\mathrm c}(p)\in(0,1)$ be the positive zero of $H_p$ from \eqref{flux-functions}.
Every orientation-preserving $C^1$ diffeomorphic local minimiser of $\E_p$, with respect to compactly supported homeomorphic perturbations, satisfies
\begin{equation}\label{necessary-qc-bound}
 |\mu_h|\leq r_{\mathrm c}(p),\qquad
 K_h\leq K_{\mathrm c}(p):=\frac{1+r_{\mathrm c}(p)}{1-r_{\mathrm c}(p)}.
\end{equation}
There are smooth circle diffeomorphisms with finite-energy quasiconformal barriers for $\I_p$ whose boundary-value problem has no such diffeomorphic minimiser. This does not exclude minimisers of lower regularity.
\end{proposition}
\begin{proof}
At a nonconformal point, rotations and a positive scalar factor put the derivative into the form $\operatorname{diag}(1+r,1-r)$. Homogeneity removes the scalar factor from the sign of the second derivative. The variation is the real rank-one direction $e_2\otimes e_2$ at that diagonal matrix. Its displayed second derivative must be nonnegative at a local minimum. Indeed compactly supported oscillations of small amplitude preserve the homeomorphism property of a $C^1$ diffeomorphism; their high-frequency second variation freezes the positive target weight and tests precisely this derivative. Negative sign on a small neighbourhood would decrease the energy. Thus $H_p(r)\geq0$. Since $H_p$ is strictly decreasing on $[0,1]$, $H_p(0)>0$ and $H_p(1)<0$, this is \eqref{necessary-qc-bound}. Conformal points already satisfy it.

Smooth boundary values need not have any extension with a prescribed fixed quasiconformal bound. To see this, fix three boundary points and choose smooth orientation-preserving circle diffeomorphisms fixing them while the image of a fourth fixed point approaches one of the three. Extensions with a common bound $K_{\mathrm c}(p)$ would contradict compactness and the common boundary modulus of continuity of the inverses of normalised quasiconformal maps \cite{AIM}. Consequently some of these smooth boundary maps, and their inverses, have no $K_{\mathrm c}(p)$-quasiconformal extension.

Each smooth circle diffeomorphism nevertheless has a finite-energy quasiconformal barrier for every $p>1$, by Proposition~\ref{smooth-barrier}. Combining these barriers with boundary values requiring distortion greater than $K_{\mathrm c}(p)$ proves the last assertion.
\end{proof}

\section{Existence in the monotone class for \texorpdfstring{$p_{\mathrm c}\leq p\leq2$}{p-c <= p <= 2}}\label{direct-existence}
The direct method in \cite[Section 5]{MYLp} starts with the inverse Dirichlet bound, obtains compactness in an enlarged class, and then proves lower semicontinuity. There the exponent is a power $m>1$ of the distortion and the inverse density $\|Dh\|^{2m}/J_h^{m-1}$ is convex. Here the exponent $p$ is the power of $|\mu|$ and the inverse density is not convex; Lemma~\ref{jacobian-convexification} gives the replacement, and the transported weight $\eta(h)$ is handled by localisation.

\subsection{The weighted lower-semicontinuity step}
\begin{lemma}\label{weighted-lsc}
Suppose $p_{\mathrm c}\leq p\leq2$ and $h_j,h$ have the convergence, orientation and compact-containment properties in Proposition~\ref{monotone-compactness}. Then
\begin{equation}\label{weighted-lsc-local}
 \int_\D\chi\eta(h)W_p(Dh)\area
 \leq\liminf_j\int_\D\chi\eta(h_j)W_p(Dh_j)\area
\end{equation}
for every nonnegative continuous compactly supported $\chi$. In particular $\E_p(h)\leq\liminf_j\E_p(h_j)$.
\end{lemma}
\begin{proof}
On the support of $\chi$, $\eta(h_j)\to\eta(h)$ uniformly and these weights are uniformly bounded. Also $G_p(Dh_j)\leq C_p\|Dh_j\|^2$, so the uniform $W^{1,2}$ bound permits replacing $\eta(h_j)$ by $\eta(h)$ in the convex part of the integral, at a cost $o(1)$. Convexity gives
\[
 \int\chi\eta(h)G_p(Dh)\area
 \leq\liminf_j\int\chi\eta(h_j)G_p(Dh_j)\area.
\]

The nonnegative measures $J_{h_j}\area$ have total mass $\pi$ and converge distributionally to $J_h\area$. They consequently converge against continuous compactly supported test functions. Applying this to $\chi\eta(h)$ and again replacing the weight uniformly yields
\begin{equation}\label{weighted-jacobian-convergence}
 \int\chi\eta(h_j)J_{h_j}\area
 \longrightarrow\int\chi\eta(h)J_h\area.
\end{equation}
Subtract $c_p$ times this convergent quantity using \eqref{polyconvex-inverse}. This proves \eqref{weighted-lsc-local}. Exhaustion by compact cutoffs and nonnegativity of $W_p$ give the global assertion. All determinant subtractions take place on compact sets; the two global integrals obtained by formally splitting $G_p-c_pJ$ may each be infinite.
\end{proof}

\subsection{A minimiser and its Hopf differential}
\begin{theorem}[Direct monotone existence]\label{monotone-existence}
Let $p_{\mathrm c}\leq p\leq2$, let $\beta$ be quasisymmetric, and suppose $\mathcal M_\beta$ contains a finite-energy competitor. Then $\E_p$ attains its minimum on $\mathcal M_\beta$. If the competitor is supplied as the inverse of a finite-energy finite-distortion extension of $\beta^{-1}$, quasisymmetry follows automatically from Proposition~\ref{finite-energy-boundary-moduli}. Every minimiser is stationary under compactly supported smooth changes of domain and has a holomorphic differential
\begin{equation}\label{monotone-hopf}
 \Phi_p=\frac12\eta(h)\bigl[p+(2-p)r^2\bigr]
 r^{p-2}h_w\overline{h_{\bw}},
 \qquad r=\frac{|h_{\bw}|}{|h_w|}\in[0,1].
\end{equation}
At $h_{\bw}=0$ the product is interpreted by continuity as zero; set $r=0$ where $Dh=0$. The theorem permits $r=1$ at nonzero rank-one gradients and does not yet assert local injectivity.
\end{theorem}
\begin{proof}
Take a minimising sequence in $\mathcal M_\beta$. Proposition~\ref{monotone-compactness} gives a proper monotone limit with all the required properties except possibly finite energy. Lemma~\ref{weighted-lsc} then gives that property and proves minimality.

For $g_t(w)=w+t\varphi(w)$, $\varphi\in C_c^\infty(\D,\IC)$ and small real $t$, the maps $h_t=h\circ g_t^{-1}$ remain monotone surjections with the same trace, remain proper, and have nonnegative Jacobian. Their energy is unchanged off a compact set and remains finite on that set by the quadratic growth of $W_p$ and the boundedness of $\eta(h)$. Thus both signs of $t$ are admissible.

On the orientation-preserving cone, the first derivatives of $W_p$ are bounded by $C_p(|a|+|b|)$, with the continuous interpretation at $b=0$ and at the origin. The differentiated integrand after changing variables by $g_t$ is bounded by $C\eta(h)\|Dh\|^2$ on the variation support. Differentiation under the integral is therefore valid. Formula~\eqref{variation} extends to nonzero rank-one gradients by continuity from $|b|<|a|$, and is zero where $Dh=0$. Hence, 
\[
 0=-2\operatorname{Re}\int_\D\Phi_p\varphi_{\bw}\area.
\]
The bound $|\Phi_p|\leq C_p\eta(h)\|Dh\|^2$ makes $\Phi_p$ locally integrable. Replacing $\varphi$ by $i\varphi$ and using Weyl's lemma proves holomorphy.
\end{proof}
\begin{corollary}[Local higher Sobolev regularity of the weak minimiser]\label{monotone-higher-sobolev}
Every minimiser in Theorem~\ref{monotone-existence} has
\[
 h_{\bar w}\in L^\infty_{\loc}(\D),\qquad
 h\in W^{1,q}_{\loc}(\D)\quad(1<q<\infty).
\]
In particular it has Lusin's $N$ property locally. These conclusions allow rank-one derivatives and do not assert local injectivity.
\end{corollary}
\begin{proof}
Taking absolute values in \eqref{monotone-hopf} gives $|\Phi_p|=\eta(h)|h_w|^2F_p(r)$ with $F_p$ as in \eqref{flux-functions}, and $F_p(r)\geq r^p\geq r^2$ on $[0,1]$. Hence $|h_{\bar w}|^2\leq|\Phi_p|/\eta(h)\leq|\Phi_p|$, which is locally bounded by holomorphy; the inequality holds also at rank-one derivatives by continuity. For a cutoff $\chi\in C_c^\infty(\D)$, the function $(\chi h)_{\bar w}=\chi h_{\bar w}+\chi_{\bar w}h$ is bounded with compact support, so the $L^q$ boundedness of the Beurling transform gives $(\chi h)_w\in L^q$ for every $1<q<\infty$ \cite[Chapters 4--5]{AIM}. The $N$ property follows from any exponent $q>2$ by the Sobolev oscillation estimate and a covering of a null set.
\end{proof}

\begin{corollary}[Strong convergence of minimising sequences]\label{strong-monotone-minimising}
For $p_{\mathrm c}\leq p\leq2$, any minimising sequence in $\mathcal M_\beta$ has a subsequence converging uniformly on $\overline\D$ and strongly in $W^{1,2}_{\loc}(\D)$ to a minimiser.
\end{corollary}
\begin{proof}
Extract the subsequence above. Its total energies converge to $\E_p(h)$ by minimality and lower semicontinuity. For $0\leq\chi\leq1$ continuous and compactly supported, the local lower-semicontinuity inequality also applies to $1-\chi$ by exhaustion. Subtracting this latter inequality from convergence of the total energies shows convergence of the $\chi$-weighted energies. Equation~\eqref{weighted-jacobian-convergence} then gives
\[
 \int\chi\eta(h)G_p(Dh_j)\area
 \longrightarrow\int\chi\eta(h)G_p(Dh)\area,
\]
where replacing $\eta(h_j)$ by $\eta(h)$ costs $o(1)$.

We use the Radon--Riesz principle of weak convergence together with convergence of the energy, as in \cite{MYRadonRiesz}. The following direct argument includes the critical exponent and requires no positivity assumption on the limiting Jacobian. Regard gradients as pairs $A=(a,b)\in\IC^2$, with $|A|^2=|a|^2+|b|^2$, and define the convexity defect
\[
 \Delta_p(A,B)=\frac{G_p(A)+G_p(B)}2
                -G_p\!\left(\frac{A+B}2\right).
\]
By Lemma~\ref{jacobian-convexification}, this is positive whenever $A\ne B$. For $0<\epsilon<1$, continuity and compactness give a constant $\delta_{p,\epsilon}>0$ such that
\[
 \Delta_p(A,B)\geq\delta_{p,\epsilon}
 \quad\text{if}\quad |A|^2+|B|^2=1,
 \qquad |A-B|^2\geq\epsilon.
\]
The two-homogeneity of $G_p$ and $|A-B|^2\leq2(|A|^2+|B|^2)$ consequently imply, for all $A,B$,
\begin{equation}\label{critical-radon-riesz-defect}
 |A-B|^2\leq\epsilon(|A|^2+|B|^2)
       +\frac{2}{\delta_{p,\epsilon}}\Delta_p(A,B).
\end{equation}
Put $dm=\chi\eta(h)\area$. Weak lower semicontinuity applied to the midpoint gradients, together with the convergence of the $G_p$ integrals already proved, gives
\[
 0\leq\limsup_j\int\Delta_p(Dh_j,Dh)\,dm
 \leq\int G_p(Dh)\,dm
       -\liminf_j\int G_p\!\left(\frac{Dh_j+Dh}2\right)\,dm
 \leq0.
\]
Integrate \eqref{critical-radon-riesz-defect}. The gradients have a uniform $L^2(dm)$ bound, so first letting $j\to\infty$ and then $\epsilon\downarrow0$ proves $Dh_j\to Dh$ in $L^2(dm)$. Taking $\chi=1$ on any prescribed compact subset proves the assertion. When $p>p_{\mathrm c}$, the positive constant $\ell_p$ gives the shorter quadratic-defect argument; the strict-convexity argument above is needed at $p=p_{\mathrm c}$.
\end{proof}

\begin{corollary}[Recovery of the homeomorphic problem under local regularity]\label{monotone-to-homeomorphic-minimum}
Let $h$ be a minimiser from Theorem~\ref{monotone-existence}. If $h$ is a local $C^1$ diffeomorphism in $\D$, then it is a global diffeomorphism there and extends homeomorphically to $\overline\D$. It minimises also over the homeomorphic inverse class, and $f=h^{-1}$ minimises the forward energy in the corresponding class for which the inverse calculus holds. If, in addition, $h$ is $C^2$ on $\{r>0\}$, Proposition~\ref{subquadratic-asymptotic} makes $h$ asymptotically conformal and globally quasiconformal.
\end{corollary}
\begin{proof}
Lemma~\ref{local-to-global} gives the global topological assertion; no boundary information beyond the trace is needed. Every homeomorphic competitor with the stated inverse calculus belongs to $\mathcal M_\beta$, and the regular minimiser itself now belongs to the original class. The two infima therefore agree. The inverse energy identity and then Proposition~\ref{subquadratic-asymptotic} give the remaining assertions.
\end{proof}

The proof of Corollary~\ref{strong-monotone-minimising} adapts the weighted Radon--Riesz mechanism of \cite{MYRadonRiesz} to the convexified density $G_p$ and allows $J_h=0$.

\section{Regularity of quasiconformal stationary maps on the elliptic branch}\label{forward-regularity}
We prove diffeomorphic regularity for quasiconformal stationary maps on the elliptic branch, using the potential and difference-quotient method of \cite[Section 2]{MYLp}, and then describe the compatibility required for analyticity at Hopf zeros.

\subsection{The forward stress equation}
\begin{proposition}[Regularity on the elliptic branch]\label{forward-regularity-proposition}
Let $1<p<2$, and let $f$ be a quasiconformal homeomorphism which is stationary under compactly supported smooth inner variations for
\[
 \int_\D\eta(z)e(|\mu_f(z)|)\area,
 \qquad e(r)=\frac{r^p}{1-r^2},
\]
where $\eta$ is smooth and strictly positive. Suppose $|\mu_f|\leq R<1$ almost everywhere and $H_p(r)>0$ on $[0,R]$. Then, for every
\begin{equation}\label{forward-holder-exponent}
 0<\alpha<\min\left\{1,\frac{2-p}{p-1}\right\},
\end{equation}
the coefficient $\mu_f$ and its modulus belong to $C^{1,\alpha}_{\loc}$, and $f$ is a $C^{2,\alpha}_{\loc}$ diffeomorphism. The coefficient and map are smooth on $\{\mu_f\ne0\}$. The condition on $H_p$ is automatic for $p_{\mathrm c}\leq p<2$ and any $R<1$; below $p_{\mathrm c}$ it holds whenever $R<r_{\mathrm c}(p)$.
\end{proposition}
\begin{proof}
Write $\mu=\mu_f$, $r=|\mu|$, $L=\log\eta$, and introduce the {\em complex stress} $U$,
\begin{equation}\label{forward-complex-stress}
 T(r)=\frac{1-r^2}{2}e'(r)
     =\frac{r^{p-1}[p+(2-p)r^2]}{2(1-r^2)},
 \qquad U=T(r)\frac{\overline\mu}{r},
\end{equation}
with $U=0$ at $\mu=0$. Both $U$ and $e(r)$ are locally bounded. For $g_t(z)=z+t\varphi(z)$, vary $f$ by $f\circ g_t^{-1}$ and change variables by $g_t$. The density variation contributes $-2\eta\operatorname{Re}(U\varphi_{\bz})$, while the weight and area contribute $2e(r)\operatorname{Re}((\eta\varphi)_z)$. Stationarity, also with $i\varphi$, therefore gives
\begin{equation}\label{forward-weighted-distribution}
 \int\eta U\varphi_{\bz}\area
 =\int e(r)(\eta\varphi)_z\area,
 \qquad\varphi\in C_c^\infty(\D,\IC).
\end{equation}
Differentiation under the integral is legitimate because the quasiconformal bound remains strictly below one under sufficiently small variations, and the density is $C^1$ as a function of $\mu$ for $p>1$. Thus
\begin{equation}\label{forward-stress-distribution}
 \boxed{\quad U_{\bz}=(e(r))_z-L_{\bz}U\quad}
\end{equation}
in distributions, with no derivative of $\mu$ assumed beforehand.

The function $T$ is strictly increasing on $[0,1)$, being a product of the increasing positive functions $r^{p-1}$, $p+(2-p)r^2$ and $(1-r^2)^{-1}$. Define $\mathcal B(t)=e(T^{-1}(t))$. Its radial derivative obeys
\begin{equation}\label{forward-stress-contraction}
 \mathcal B'(T(r))=\frac{e'(r)}{T'(r)},\qquad
 \frac{T'}{e'}-1
 =\frac{(1-r)^2H_p(r)}{2r[p+(2-p)r^2]}.
\end{equation}
At $r=0$ the first ratio has limit zero. Consequently the radial map $a\mapsto\mathcal B(|a|)$ is $k$-Lipschitz on a neighbourhood of $\{|a|\leq T(R)\}$ for some $k<1$. Extend it to a globally $k$-Lipschitz map by continuing the radial profile as a constant past a slightly larger radius.

We first justify one weak derivative of $U$. On a relatively compact disk, take a local Cauchy correction $G$ with
\[
 G_{\bz}=L_{\bz}U,\qquad
 G\in W^{1,s}_{\loc}\quad(1<s<\infty).
\]
Such a correction is obtained from the bounded right-hand side with a cutoff. Equation~\eqref{forward-stress-distribution} gives a local potential $\mathcal F$ satisfying
\[
 \mathcal F_z=U+G,\qquad
 \mathcal F_{\bz}=\mathcal B(|U|)
                =\mathcal B(|\mathcal F_z-G|).
\]
The compatibility of these two derivatives is exactly the distributional equation; this is the potential construction of \cite[Lemma 2.1]{MYLp}. Here $\mathcal F\in W^{1,2}_{\loc}$, in fact locally Lipschitz. Its difference quotients satisfy
\[
 |(\delta_t\mathcal F)_{\bz}|
 \leq k|(\delta_t\mathcal F)_z|+k|\delta_tG|.
\]
Localise with a smooth cutoff and use equality of the two Wirtinger derivative norms for compactly supported Sobolev functions. For disks $V\Subset V'$ this gives
\[
 \|D\delta_t\mathcal F\|_{L^2(V)}
 \leq C\bigl(\|\delta_t\mathcal F\|_{L^2(V')}
                 +\|\delta_tG\|_{L^2(V')}\bigr).
\]
The right-hand side is bounded independently of small $t$. Hence $\mathcal F\in W^{2,2}_{\loc}$ and $U=\mathcal F_z-G\in W^{1,2}_{\loc}$.

We may now expand the weak equation by the Sobolev chain rule. Put
\[
 c=\tfrac12\mathcal B'(|U|),\qquad \omega=U/|U|\quad(U\ne0).
\]
Solving the equation and its conjugate for $U_{\bz}$ yields the inhomogeneous Beltrami equation
\begin{equation}\label{forward-stress-beltrami}
 \boxed{\quad
 U_{\bz}=\mathfrak a\,U_z+\mathfrak b\,\overline{U_z}
                    +\mathfrak f,\quad}
\end{equation}
where
\[
 \mathfrak a=\frac{c\overline\omega}{1-c^2},\qquad
 \mathfrak b=\frac{c^2\omega^2}{1-c^2},\qquad
 \mathfrak f=-\frac{L_{\bz}U+c|U|L_z}{1-c^2}.
\]
All three extend continuously by zero at $U=0$, and
\begin{equation}\label{forward-stress-ellipticity}
 |\mathfrak a|+|\mathfrak b|
 =\frac{\mathcal B'(|U|)}{2-\mathcal B'(|U|)}
 \leq\frac{k}{2-k}<1.
\end{equation}
Local higher integrability for this uniformly elliptic equation, with bounded inhomogeneous term, gives $U\in W^{1,s}_{\loc}$ for some $s>2$, hence local H\"older continuity. The coefficient functions are locally Lipschitz functions of $U$, also at zero: with $\beta=1/(p-1)>1$ one has $\mathcal B'(t)=O(t^\beta)$ and $\mathcal B''(t)=O(t^{\beta-1})$. Schauder estimates first give $U\in C^{1,\gamma}_{\loc}$ for some $\gamma>0$; the resulting local Lipschitz bound makes the coefficients and inhomogeneous term Lipschitz in $z$, so a second application gives $U\in C^{1,\gamma}_{\loc}$ for every $\gamma<1$. These are the standard local linear Beltrami estimates; see \cite[Chapter 15]{AIM}.

Finally recover
\begin{equation}\label{forward-stress-reconstruction}
 \mu=\frac{T^{-1}(|U|)}{|U|}\,\overline U.
\end{equation}
Near zero the scalar inverse has the form
\[
 T^{-1}(t)=(2/p)^\beta t^\beta B(t^{2\beta}),
 \qquad B\ \hbox{real analytic near }0,\quad B(0)=1.
\]
This follows by applying the analytic inverse-function theorem to $T(r)^{2\beta}$ as a function of $r^2$. Thus both reconstruction maps for $\mu$ and $r$ are $C^{1,\alpha}$ for the exponents in \eqref{forward-holder-exponent}. The linear equation $f_{\bz}=\mu f_z$ gives $f\in C^{2,\alpha}_{\loc}$. Its Jacobian is strictly positive: the nonvanishing theorem for a homeomorphic solution with H\"older Beltrami coefficient applies, for example \cite[Theorem 1.1]{ACFJK} with the linear field $\mathcal H(z,\xi)=\mu(z)\xi$. Smooth bootstrap applies wherever $U\ne0$. This proves the proposition.
\end{proof}
\subsection{The implicit Hopf relation and its ellipticity}\label{implicit-hopf}
The same threshold appears when the inverse Hopf equation is written as an implicit nonlinear Beltrami equation, following \cite[Section 3]{MYHigher}. Put $x=|h_w|$, $y=|h_{\bw}|$ and $r=y/x$. On a nonzero Hopf level,
\[
 x^2F_p(r)=c,\qquad c=|\Phi|/\eta(h).
\]
Freeze $w$ and the target value $h$ when differentiating with respect to the gradient, and set
\[
 C_p(r)=p+(2-p)r^2,\qquad
 D_p(r)=p(p-1)+(2-p)(p+1)r^2.
\]
\begin{lemma}\label{implicit-hopf-lemma}
Along a nonzero Hopf level,
\begin{equation}\label{implicit-hopf-ellipticity}
 \frac{dy}{dx}=r-\frac{2F_p(r)}{F_p'(r)}
       =\frac{r[D_p(r)-2C_p(r)]}{D_p(r)}<0,
 \qquad
 1-\left|\frac{dy}{dx}\right|
       =\frac{(1-r)H_p(r)}{D_p(r)}.
\end{equation}
Thus the tangential derivative of the constitutive map $h_w\mapsto h_{\bw}$ has magnitude $r$, its radial derivative has magnitude $|dy/dx|>r$ for $0<r<1$ and $p<2$, and the map is a strict contraction exactly where $H_p(r)>0$.
\end{lemma}
\begin{proof}
Differentiate $x^2F_p(r)=c$ with $r=y/x$ and use $rF_p'(r)/F_p(r)=D_p(r)/C_p(r)$; the identity $(1+r)D_p-2rC_p=(1-r)H_p(r)$ gives the second formula.
\end{proof}
For $p\geq p_{\mathrm c}$ a local quasiconformal bound $r\leq k<1$ therefore gives a strict contraction. Below $p_{\mathrm c}$ a general Hopf solution needs strict separation from $r_{\mathrm c}(p)$, which local quasiconformality alone does not give; for example $p=3/2$, $r=1/2$ gives $dy/dx=-35/34$. For a minimiser the cutoff $r\leq r_{\mathrm v}(p)<r_{\mathrm c}(p)$ of Section~\ref{variational-cutoff} gives that separation once the rank-one set is null. The endpoint $r=1$ is also degenerate when $\Phi\ne0$: since $F_p(1)=F_p'(1)=1$, the level curve $x^2F_p(y/x)=c$ reaches the rank-one line $x=y=\sqrt c$ with slope $-1$. This is a specific difference from superlinear distortion profiles, such as $\K^m$ or $e^{\K}$, whose nonzero level curves approach the rank-one direction only at the origin, and it is what the regular limiting construction of Section~\ref{lambda-programme} has to exclude.

\subsection{Holomorphic Hopf data and real analyticity at a Hopf zero}\label{analyticity-scope}
Holomorphy of the differential and real analyticity of its solution are different assertions. For a classical local diffeomorphic solution, the implicit Hopf equation is an analytic nonlinear elliptic system on every open set where $\Phi_p\ne0$ and $H_p(r)>0$. Its coefficients are real analytic there, including the hyperbolic metric, and its ellipticity is uniform on sufficiently small compact neighbourhoods. The derivative test is \eqref{implicit-hopf-ellipticity}. Standard analytic elliptic regularity therefore makes the solution real analytic on that set \cite{MorreyAnalytic}. For $p_{\mathrm c}\leq p\leq2$, the strict inequality holds whenever $0<r<1$.

For $p=2$, a local diffeomorphism with holomorphic Hopf differential satisfies the hyperbolic harmonic-map equation
\[
 h_{w\bw}+\frac{2\overline h}{1-|h|^2}h_wh_{\bw}=0.
\]
This is an analytic elliptic equation, also at $h_{\bw}=0$. If initially $h$ is only $C^1$ through the isolated Hopf zeros and $C^2$ elsewhere, the equation extends weakly across those points: the boundary terms on circles of radius $\epsilon$ tend to zero because the first derivatives are bounded. Elliptic regularity then gives real analyticity throughout. If $\Phi_2\equiv0$, the positive Jacobian gives $h_{\bw}\equiv0$, and the map is holomorphic.

For $1<p<2$, the constitutive relation has a different behaviour at a conformal point. The following conclusion follows directly from the Hopf equation and requires no regularity theorem.

\begin{proposition}[The compatibility required at a Hopf zero]\label{analytic-hopf-zeros}
Let $1<p<2$. Suppose $h$ is an orientation-preserving $C^1$ local diffeomorphism near $w_0$ and satisfies \eqref{monotone-hopf} with holomorphic $\Phi_p\not\equiv0$. If $\Phi_p$ has a zero of order $m\geq1$ at $w_0$, then, with $\zeta=w-w_0$,
\begin{equation}\label{hopf-zero-profile}
 h_{\bw}(w)=C\,\overline\zeta^{\,m}
 |\zeta|^{\,m(2-p)/(p-1)}(1+o(1)),\qquad C\ne0.
\end{equation}
A necessary condition for $h$ to be real analytic at $w_0$ is
\begin{equation}\label{analytic-zero-compatibility}
 \frac{m(2-p)}{p-1}\ \in\ \{0,2,4,\ldots\}.
\end{equation}
This is a necessary condition, not a sufficient analytic regularity theorem.
\end{proposition}
\begin{proof}
Put $a=h_w$, $b=h_{\bw}$. Since $J_h(w_0)>0$, $a(w_0)\ne0$, and the zero of $\Phi_p$ forces $b(w_0)=0$. Write the equation as
\[
 \Phi_p=L(w)|b|^{p-2}\overline b,\qquad
 L(w)=\tfrac12\eta(h)[p+(2-p)r^2]|a|^{2-p}a.
\]
The coefficient $L$ is continuous and $L(w_0)\ne0$. Setting $X=\Phi_p/L$ gives the exact inversion
\[
 b=\overline X\,|X|^{(2-p)/(p-1)}.
\]
As $\Phi_p=\gamma\zeta^m(1+O(\zeta))$, $\gamma\ne0$, this proves \eqref{hopf-zero-profile}.

If $h$ is real analytic, the nonzero real analytic function $b$ has a first nonzero homogeneous Taylor polynomial of some integer degree $N$. The modulus in \eqref{hopf-zero-profile} forces $N=m/(p-1)$. Its angular factor is $e^{-im\theta}$. A homogeneous polynomial in $\zeta,\overline\zeta$ of degree $N$ can have exactly this angular factor only if it is a multiple of
\[
 \zeta^{(N-m)/2}\overline\zeta^{(N+m)/2}.
\]
Thus $N-m=m(2-p)/(p-1)$ must be a nonnegative even integer, proving \eqref{analytic-zero-compatibility}.
\end{proof}

For example, a simple Hopf zero at $p=3/2$ forces the leading term $C\overline\zeta|\zeta|$, excluding $C^3$ regularity; at $p=9/5$ it forces $C\overline\zeta|\zeta|^{1/4}$. These are conditional consequences of a specified zero. Our regularity conclusions across the conformal set are $C^{2,\alpha}$; if $\Phi_p\equiv0$, the map is conformal and analytic.

\section{Below the threshold}\label{below-threshold}
For $1<p<p_{\mathrm c}$ the inverse density is not polyconvex and the scalar equations of Section~\ref{scalar-equations} are not elliptic for all distortions. Minimality nevertheless gives strong information: finite changes of domain impose a cutoff on the distortion of a minimiser, strictly stronger than the Legendre--Hadamard test, except on a possible rank-one set. All assertions in this section are conditional on the existence of the minimiser.

\subsection{Weak localisation and a stronger variational cutoff}\label{variational-cutoff}
Finite changes of domain give more information than the Legendre--Hadamard test. In this subsection a minimiser may be a global minimiser in $\mathcal M_\beta$, or in any subclass which is preserved by compactly supported smooth orientation-preserving changes of domain. This is enough to have minimality among such competitors which are uniformly close to the given map. This is a condition of strong local minimality; stationarity, or minimality only against small changes of the gradient, is not enough.

\begin{lemma}[Localisation by changes of domain]\label{weak-inner-localisation}
Let $0<p<2$, and let $h\in\mathcal M_\beta$ have the minimality property just stated. At almost every $w_0\in\D$, put $A=Dh(w_0)$. For any bounded square $Q$ and any smooth orientation-preserving diffeomorphism $\psi:Q\to Q$ equal to the identity near $\partial Q$,
\begin{equation}\label{frozen-inner-inequality}
 |Q|W_p(A)\leq\int_Q W_p(A D\psi(x))\,dx.
\end{equation}
In particular, no injectivity or differentiability beyond the stated continuity and Sobolev regularity of $h$ is needed for this necessary condition.
\end{lemma}
\begin{proof}
Choose $w_0$ to be an $L^2$-Lebesgue point of $Dh$. For small $\epsilon>0$, define a diffeomorphism of $\D$ by
\[
 \psi_\epsilon(w_0+\epsilon x)=w_0+\epsilon\psi(x)
 \quad(x\in Q),\qquad
 \psi_\epsilon=\id\quad\hbox{off }w_0+\epsilon Q.
\]
The competitor $h_\epsilon=h\circ\psi_\epsilon$ has the same trace, image and monotonicity as $h$, and $J_{h_\epsilon}\geq0$. On the variation support its image lies in a compact subset of $\D$, so the weight is bounded and positive. The quadratic bound for $W_p$ and the Sobolev chain rule imply finite energy there; elsewhere the energy is unchanged. Thus the competitor is admissible. Continuity gives $h_\epsilon\to h$ uniformly, even though $D\psi$ need not be close to $I$.

The $L^2$-Lebesgue property, followed by a change of variables through the fixed bi-Lipschitz map $\psi$, gives
\[
 Dh(w_0+\epsilon\,\cdot)\longrightarrow A,
 \qquad
 Dh(w_0+\epsilon\psi(\cdot))D\psi
       \longrightarrow A D\psi
 \quad\hbox{strongly in }L^2(Q).
\]
Here the inverse-Jacobian bound for this fixed $\psi$ is finite; no bound uniform over all test diffeomorphisms is being asserted. The target values in both integrals converge uniformly to $h(w_0)$.

To justify convergence of the densities, extend $W_p$ continuously to all real matrices by
\[
 \widetilde W_p(M)=\tfrac14(\sigma_1-\sigma_2)^p
                            (\sigma_1+\sigma_2)^{2-p},
 \qquad \sigma_1\geq\sigma_2\geq0.
\]
This agrees with $W_p$ when $\det M\geq0$ and has quadratic growth. Strong $L^2$ convergence, continuity and uniform integrability therefore give strong $L^1$ convergence of the densities. Divide the local minimality inequality by $\epsilon^2$, pass to the limit and cancel $\eta(h(w_0))>0$. This proves \eqref{frozen-inner-inequality}. The set of $L^2$-Lebesgue points used here is independent of $Q$ and $\psi$.
\end{proof}

\begin{lemma}[A rank-one Weierstrass condition]\label{rankone-weierstrass}
Let $1<p<2$. At a point covered by Lemma~\ref{weak-inner-localisation} with $J_h>0$, rotate and scale the derivative to $A=\operatorname{diag}(1,\tau)$, $0<\tau<1$. Set
\[
 \phi_p(\tau)=\tfrac14(1-\tau)^p(1+\tau)^{2-p}.
\]
Then
\begin{equation}\label{weierstrass-endpoint}
 \phi_p(0)\geq\phi_p(\tau)-\tau\phi_p'(\tau).
\end{equation}
The conformal case $\tau=1$ satisfies this inequality directly.
\end{lemma}
\begin{proof}
Rotation invariance, degree-two homogeneity and the freedom to rotate the test square justify the normalisation. Fix $0<\delta<\tau$ and a small $\theta>0$ such that
\[
 s_\theta=\frac{\tau-\theta\delta}{1-\theta}<1.
\]
The matrices $B=\operatorname{diag}(1,\delta)$ and $C=\operatorname{diag}(1,s_\theta)$ have positive determinant, are rank-one connected and have average $A$ with proportions $\theta,1-\theta$.

Here is an explicit way to realise the comparison in \eqref{frozen-inner-inequality}. Let $q$ be periodic with mean-zero derivative taking the values $\delta-\tau$ and $s_\theta-\tau$ in those proportions, and let $0\leq\chi\leq1$ be a smooth cutoff on the square, equal to one except near its boundary. Use
\[
 \psi_n(x,y)=\left(x,\ y+\frac{\chi(x,y)}{n\tau}q(ny)\right).
\]
After smoothing $q$, these are smooth diffeomorphisms for sufficiently large $n$: their second component is strictly increasing in $y$, since
\[
 \det D\psi_n=1+\frac{\chi}{\tau}q'(ny)+O(n^{-1})
       \geq\min\{1,\delta/\tau\}+O(n^{-1})>0.
\]
They fix a neighbourhood of the boundary and map the square onto itself. In the region $\chi=1$, $A D\psi_n$ alternates between $B$ and $C$, up to the smoothing. First take $n\to\infty$, then shrink the boundary layer, and then remove the smoothing. All gradients in these passages are bounded for the fixed $\delta,\theta$; the boundary-layer contribution tends to zero. Consequently
\[
 \phi_p(\tau)\leq\theta\phi_p(\delta)
                         +(1-\theta)\phi_p(s_\theta).
\]
Letting $\theta\downarrow0$ gives
\[
 \phi_p(\delta)\geq\phi_p(\tau)
                       +(\delta-\tau)\phi_p'(\tau).
\]
Finally let $\delta\downarrow0$. Every competitor used before this last numerical limit is a diffeomorphism with positive Jacobian; a singular competitor has not been inserted into the admissible class.
\end{proof}

\begin{theorem}[The weak variational cutoff]\label{weak-variational-cutoff}
Let $1<p<p_{\mathrm c}$, and let $h\in\mathcal M_\beta$ have the minimality property of Lemma~\ref{weak-inner-localisation}. With the flux function of \eqref{flux-functions}, the selection condition is
\begin{equation}\label{variational-Q}
 Q_p(r)\leq Q_p(1)=1.
\end{equation}
There is a unique $r_{\mathrm v}(p)\in(0,r_{\mathrm c}(p))$ such that $Q_p(r_{\mathrm v})=1$. Almost everywhere,
\begin{equation}\label{weak-phase-alternative}
 Dh=0,\qquad\hbox{or}\qquad
 J_h>0\ \hbox{ and }\ |\mu_h|\leq r_{\mathrm v}(p),
 \qquad\hbox{or}\qquad
 J_h=0\ \hbox{ and }\ \operatorname{rank}Dh=1.
\end{equation}
Equivalently, with $r=0$ where $Dh=0$, one has
\[
 r\in[0,r_{\mathrm v}(p)]\cup\{1\}\quad\hbox{almost everywhere}.
\]
In particular the entire interval $(r_{\mathrm v}(p),1)$ is excluded, not just the interval where the classical equation changes type.
\end{theorem}
\begin{proof}
For $\tau=(1-r)/(1+r)$,
\[
 \phi_p'(\tau)=-\frac{2(p-1+\tau)}{1-\tau^2}\phi_p(\tau),
 \qquad
 4\bigl(\phi_p(\tau)-\tau\phi_p'(\tau)\bigr)=Q_p(r).
\]
Thus Lemma~\ref{rankone-weierstrass} gives $Q_p(r)\leq1$ wherever $J_h>0$. The dictionary identity \eqref{flux-log-derivative} reads
\begin{equation}\label{Q-derivative-sign}
 \frac{Q_p'(r)}{Q_p(r)}
  =\frac{(1-r)H_p(r)}{r(1+r)[p+(2-p)r^2]}.
\end{equation}
Since $H_p$ decreases through zero at $r_{\mathrm c}<1$, $Q_p$ increases strictly on $(0,r_{\mathrm c})$ and decreases strictly on $(r_{\mathrm c},1)$. Also $Q_p(0)=0$ and $Q_p(1)=1$, so its maximum is strictly larger than one. This proves the existence and uniqueness of $r_{\mathrm v}$ and the claimed restriction for $r<1$. At $J_h=0$, a nonzero orientation-preserving-cone matrix has rank one and $r=1$; the zero matrix has our convention $r=0$.
\end{proof}

For $p=3/2$, the smaller root is
\[
 r_{\mathrm v}=0.2955977425\ldots<2\sqrt3-3=r_{\mathrm c}.
\]
Indeed $x=\sqrt{r_{\mathrm v}}$ is the positive root of $x^3+x^2+x-1=0$. Hence
\[
 K_{\mathrm v}(p):=\frac{1+r_{\mathrm v}(p)}{1-r_{\mathrm v}(p)},
 \qquad K_{\mathrm v}(3/2)=1.8392867552\ldots.
\]
In fact $K_{\mathrm v}(3/2)=1/x$ is the tribonacci constant: multiplying $(1+x^2)/(1-x^2)=1/x$ by $x(1-x^2)$ gives exactly $x^3+x^2+x=1$.
The cutoff is necessary; no sufficiency or sharpness for the full variational problem is claimed. In the forward scalar equation of Proposition~\ref{forward-scalar-reduction}, the flux is $tA_p(t)=\sqrt{Q_p(r)}/2$. The new restriction is exactly that this flux does not exceed its limit $1/2$ as $r\uparrow1$. The ellipticity cutoff locates the maximum of the flux, whereas minimality selects a smaller range.

\begin{corollary}[Quasiconformality once the rank-one set is null]\label{cutoff-finite-distortion}
Under Theorem~\ref{weak-variational-cutoff}, if
\[
 |\{J_h=0,\ Dh\ne0\}|=0,
\]
then $h$ is a $K_{\mathrm v}(p)$-quasiconformal homeomorphism. On its nonconformal branch $H_p(r)\geq H_p(r_{\mathrm v})>0$, giving a strict separation from the change of type which depends only on $p$. In particular this applies if finite distortion of $h$ is known.
\end{corollary}
\begin{proof}
The alternative \eqref{weak-phase-alternative} now gives $|h_{\bw}|\leq r_{\mathrm v}|h_w|$ almost everywhere, including the zero-gradient set. A nonconstant $W^{1,2}_{\loc}$ solution of this bounded-distortion inequality is quasiregular and therefore open and discrete \cite{AIM}. Its fibres are also connected by monotonicity, so they are singletons. The map is consequently a quasiconformal homeomorphism, with the stated operator-dilatation bound. Strict decrease of $H_p$ gives the final assertion. This argument establishes quasiconformality without first assuming differentiable invertibility.
\end{proof}
\begin{corollary}[A quasiconformal minimiser in the range $1<p<p_{\mathrm c}$ is a diffeomorphism]\label{qc-minimiser-regularity}
Let $1<p<p_{\mathrm c}$. Every finite-energy minimiser of $\I_p$ over the quasiconformal homeomorphisms with prescribed boundary values is a $C^{2,\alpha}_{\loc}$ diffeomorphism for the exponents in \eqref{forward-holder-exponent}. Its inverse has the same regularity. Both maps are smooth away from their conformal sets.
\end{corollary}
\begin{proof}
For the inverse minimiser, the finite changes of domain used in Theorem~\ref{weak-variational-cutoff} remain quasiconformal competitors. Its rank-one set is null, so that theorem gives $|\mu_h|\leq r_{\mathrm v}(p)$. The same bound holds for $f$, and $r_{\mathrm v}<r_{\mathrm c}$. Forward inner variations are also admissible. Proposition~\ref{forward-regularity-proposition} therefore applies with any $r_{\mathrm v}<R<r_{\mathrm c}$, and the inverse-function theorem gives the assertion for $h$.
\end{proof}
The regularity across the conformal set is deliberately stated as $C^{2,\alpha}$ for the map. The fractional power in \eqref{forward-stress-reconstruction} prevents an automatic $C^\infty$ conclusion there, consistently with Proposition~\ref{analytic-hopf-zeros}. Away from that set, the full smooth bootstrap is available. Thus second derivatives needed by the later equations have now been justified for a quasiconformal minimiser.
\begin{corollary}[Dichotomy for an attained minimum below the threshold]\label{attained-minimum-dichotomy}
Let $1<p<p_{\mathrm c}$ and suppose $h\in\mathcal M_\beta$
attains the minimum of $\E_p$ over $\mathcal M_\beta$.
Exactly one of the following holds.
\begin{enumerate}[label=\textup{(\roman*)}]
\item The map $h$ has finite distortion. This is a
$K_{\mathrm v}(p)$-quasiconformal $C^{2,\alpha}_{\loc}$
diffeomorphism, for every
$0<\alpha<\min\{1,(2-p)/(p-1)\}$. Its inverse has the same
regularity, and both maps are smooth off their conformal sets.
\item The set
\[
 S_h=\{w\in\D:J_h(w)=0,\ Dh(w)\ne0\}
\]
has positive area. On this set the derivative has rank one almost
everywhere, and $h$ is not a mapping of finite distortion.
\end{enumerate}
Consequently, if $\beta$ has no $K_{\mathrm v}(p)$-quasiconformal
extension, any attained monotone minimum must satisfy \textup{(ii)}.
Such smooth boundary data exist and admit finite-energy
quasiconformal barriers.
\end{corollary}
\begin{proof}
Since $h\in W^{1,2}$ and $J_h\geq0$, finite distortion is
equivalent to $|S_h|=0$. In that case
Corollary~\ref{cutoff-finite-distortion} gives quasiconformality.
The inverse change of variables is then valid, so $h^{-1}$
minimises $\I_p$ among quasiconformal maps with the prescribed
inverse trace. Corollary~\ref{qc-minimiser-regularity} gives the
stated regularity. If $|S_h|>0$, a nonzero two-dimensional matrix
with zero determinant has rank one, giving \textup{(ii)}.
The boundary examples and their barriers are those constructed in
Proposition~\ref{lower-range-boundary-obstruction}; requiring
distortion greater than $K_{\mathrm c}(p)$ also excludes the
smaller bound $K_{\mathrm v}(p)$.
\end{proof}

This dichotomy is conditional on attainment; below $p_{\mathrm c}$, weak lower semicontinuity can fail.

\subsection{The endpoint barrier and sublinear-growth counterexamples}\label{p-one-barrier}
The sublinear-growth theorem of \cite[Theorem 5.3]{AIMOExtremal} does not apply to our density, since in the normalised Hilbert--Schmidt convention
\begin{equation}\label{endpoint-distortion-growth}
 \Psi_p(S)=\tfrac12(S-1)^{p/2}(S+1)^{1-p/2},\qquad
 \lim_{S\to\infty}\Psi_p(S)/S=\tfrac12.
\end{equation}
Instead, weak localisation and a finite comparison exclude nonconformal positive-Jacobian states for $0<p\leq1$.

\begin{theorem}[Conformal or rank one at and below the endpoint]\label{endpoint-phase-alternative}
Let $0<p\leq1$, and let $h\in\mathcal M_\beta$ satisfy the minimality hypothesis of Lemma~\ref{weak-inner-localisation}. Then, almost everywhere,
\begin{equation}\label{endpoint-phases}
 h_{\bw}=0
 \qquad\hbox{or}\qquad
 J_h=0,\quad Dh\ne0.
\end{equation}
In particular $h$ is weakly conformal on $\{J_h>0\}$.
\end{theorem}
\begin{proof}
The localisation lemma applies for $0<p<2$ and the proof uses continuity and quadratic growth of $W_p$, without differentiating at a conformal matrix. At a positive-Jacobian Lebesgue point with a nonconformal derivative, normalise to $A=\operatorname{diag}(1,\tau)$, $0<\tau<1$. Fix any $0<\delta<\tau$, and set
\[
 B=\operatorname{diag}(1,\delta),\qquad C=I,
 \qquad \theta=\frac{1-\tau}{1-\delta}.
\]
These rank-one connected matrices have positive determinant and satisfy $\theta B+(1-\theta)C=A$. The stripe construction in the proof of Lemma~\ref{rankone-weierstrass}, now with $s_\theta=1$, is still admissible. Hence, 
\begin{equation}\label{endpoint-finite-comparison}
 \phi_p(\tau)\leq\theta\phi_p(\delta)
                      +(1-\theta)\phi_p(1)
       =\frac{1-\tau}{1-\delta}\phi_p(\delta).
\end{equation}
But
\[
 g_p(t):=\frac{\phi_p(t)}{1-t}
       =\tfrac14(1-t)^{p-1}(1+t)^{2-p},
 \qquad
 \frac{g_p'(t)}{g_p(t)}
       =\frac{1-p}{1-t}+\frac{2-p}{1+t}>0
       \quad(0<t<1).
\]
Consequently $\phi_p(\tau)>(1-\tau)\phi_p(\delta)/(1-\delta)$, contradicting \eqref{endpoint-finite-comparison}. Every comparison has strictly positive Jacobian; no passage to a singular competitor is needed for this contradiction. At $p=1$ the energy gap is particularly explicit:
\[
 \phi_1(\tau)-\theta\phi_1(\delta)
       =\frac{(1-\tau)(\tau-\delta)}4>0.
\]
The only remaining derivatives are conformal (including zero) and nonzero rank-one matrices, proving \eqref{endpoint-phases}.
\end{proof}

\begin{corollary}[Nonattainment in the homeomorphic inverse class]\label{endpoint-nonattainment}
Let $0<p\leq1$. If a minimiser as above has finite distortion, or is a homeomorphism with $h^{-1}\in W^{1,1}_{\loc}(\D)$, then it is a M\"obius automorphism. Therefore, for non-M\"obius boundary data admitting a finite-energy competitor, the inverse energy has no minimiser in the homeomorphic class with Sobolev inverse, provided the class is preserved by the stated changes of domain. The same conclusion holds for the corresponding forward problem whenever the inverse energy identity and Sobolev inverse calculus apply.
\end{corollary}
\begin{proof}
Finite distortion excludes a nonzero derivative on $J_h=0$ by definition. For the stated homeomorphic class the same exclusion follows from the planar bi-Sobolev theorem \cite[Theorem 3.4]{PratelliBiSobolev}, applied locally. Theorem~\ref{endpoint-phase-alternative} then gives $h_{\bw}=0$ almost everywhere. Weyl's lemma makes $h$ holomorphic. The degree-one boundary homeomorphism identifies this holomorphic map as a disk automorphism.
\end{proof}

This is a necessary barrier for $p>1$ for nonconformal minimisers in that class and this does not exclude minimisers in the enlarged monotone class. At the endpoint a classical Hopf equation cannot simply be differentiated through the conformal set. Our finite-comparison proof avoids that issue altogether.

The finite-energy hypothesis is not vacuous even below $1$. Choose a smooth radial cutoff $\chi$, zero for $r\leq1/2$ and one for $r\geq3/4$, and set, for sufficiently small nonzero real $\epsilon$,
\[
 h_\epsilon(z)=z\exp\!\left[\frac\epsilon2\chi(|z|)
                              (z^2-z^{-2})\right]\quad(z\ne0),
 \qquad h_\epsilon(0)=0.
\]
It is a small $C^1$ perturbation of the identity and an orientation-preserving smooth disk diffeomorphism. It is conformal near the circle, where its boundary value is
\[
 \beta_\epsilon(e^{i\theta})
        =e^{i(\theta+\epsilon\sin2\theta)}.
\]
This boundary map fixes four points and is not the identity, so it is not M\"obius. The Beltrami coefficient is supported in a compact subset of the disk, and both $h_\epsilon$ and its inverse have finite energy for every $p>0$. Thus Corollary~\ref{endpoint-nonattainment} gives actual nonattainment examples with smooth quasisymmetric data and smooth quasiconformal barriers.

In contrast to the sublinear-growth theorem, the infimum here does not collapse to zero for fixed non-M\"obius data. Let $H_\beta$ be the Euclidean harmonic extension of the inverse boundary map. Since $W_p\geq|h_{\bw}|^2$, $\eta\geq1$, and the integrated Jacobian is fixed by the trace, the Dirichlet principle gives
\begin{equation}\label{endpoint-positive-infimum}
 \E_p(h)\geq\int_\D|h_{\bw}|^2\area
       \geq\int_\D|(H_\beta)_{\bw}|^2\area>0.
\end{equation}
The last inequality holds precisely when $\beta$ is not the trace of a disk automorphism. Quasisymmetry implies that there is a $W^{1,2}$ extension, so the harmonic extension and these integrals are well defined.

\begin{proposition}[The frozen cell value for $0<p\leq1$]\label{endpoint-cell-value}
For a square $Q$ and any real $2\times2$ matrix $A$ with $\det A\geq0$,
\begin{equation}\label{endpoint-cell-formula}
 \inf_\psi\frac1{|Q|}\int_Q W_p(A D\psi)\,dx
       =\overline W(A):=\frac14
          \bigl(\|A\|_{\mathrm{op}}^2-\det A\bigr),
\end{equation}
where the infimum is over the smooth orientation-preserving diffeomorphisms of $Q$ equal to the identity near its boundary.
\end{proposition}
\begin{proof}
If $\lambda\geq\nu\geq0$ are the singular values, then
\[
 W_p(A)\geq W_1(A)=\frac{\lambda^2-\nu^2}4
                  \geq\frac{\lambda^2-\lambda\nu}4
                  =\overline W(A).
\]
The function $\overline W(A)+\det A/4=\|A\|_{\mathrm{op}}^2/4$ is convex. Since $|Q|^{-1}\int_Q A D\psi=A$ and $|Q|^{-1}\int_Q\det(A D\psi)=\det A$, Jensen's inequality gives the lower bound in \eqref{endpoint-cell-formula}. For $A=\operatorname{diag}(1,\tau)$, $0<\tau<1$, the preceding stripes give an upper bound $\theta\phi_p(\delta)$. Sending $\delta\downarrow0$ gives $(1-\tau)/4=\overline W(A)$. Rotation invariance, homogeneity and rotating the stripe direction handle every positive-determinant $A$ on the same square. At a conformal, rank-one or zero matrix, $W_p(A)=\overline W(A)$ and the identity test gives equality.
\end{proof}

In complex derivatives this common relaxed density is
\[
 \overline W(Dh)=\tfrac12|h_{\bar w}|(|h_w|+|h_{\bar w}|)
               =\tfrac14(K_h-1)J_h\quad(J_h>0).
\]
This frozen relaxation is independent of $p$ for $0<p\leq1$, with contact at conformal and rank-one matrices. A global recovery theorem would additionally have to handle the unbounded hyperbolic weight and the prescribed trace.

The previously obtained cutoff also reflects the endpoint barrier. As $p\downarrow1$,
\begin{equation}\label{endpoint-cutoff-asymptotic}
 r_{\mathrm v}(p)\sim(2p)^{-1/(p-1)}
                    \sim e^{-1}2^{-1/(p-1)}.
\end{equation}
Indeed $Q_p(r)\to2(1+r^2)/(1+r)^2>1$ at each fixed $0<r<1$, so $r_{\mathrm v}\to0$. Its defining equation gives
\[
 (p-1)\log r_{\mathrm v}
       =-\log(2p)+2\log(1+r_{\mathrm v})
         -\log\!\left(1+\frac{2-p}{p}r_{\mathrm v}^2\right).
\]
First this shows exponential decay, hence $r_{\mathrm v}/(p-1)\to0$; substitution back into the identity proves \eqref{endpoint-cutoff-asymptotic}. The allowable nonconformal branch therefore shrinks exponentially as the endpoint is approached.

\subsection{The remaining rank-one set and the scope of NoHairs}\label{rankone-audit}
Next we return to $1<p<2$ and put $S_h=\{J_h=0,\ Dh\ne0\}$. The same first-variation calculation as in Theorem~\ref{monotone-existence} gives a holomorphic $\Phi_p$ for any minimiser considered above in this range, also when $p<p_{\mathrm c}$. Our proof of stationarity uses local quadratic growth and admissible changes of domain, independently of the upper-range existence argument.  

At a rank-one matrix the frozen necessary condition itself gives no exclusion. Indeed $A D\psi$ still has rank one, and
\[
 W_p(A D\psi)=\tfrac14|A D\psi|_{\mathrm F}^2,
 \qquad \int_Q A D\psi=|Q|A,
\]
where $|\cdot|_{\mathrm F}$ is the real Frobenius norm. Jensen's inequality therefore proves \eqref{frozen-inner-inequality} automatically at every rank-one $A$. Additional infinitesimal localisation by changes of domain cannot remove this alternative on its own.

\begin{proposition}[What the Hopf equation says on the singular set]\label{rankone-hopf-geometry}
Suppose $h$ satisfies the weighted Hopf equation \eqref{monotone-hopf}. Almost everywhere on $S_h$,
\begin{equation}\label{rankone-energy-hopf}
 |\Phi_p|=\eta(h)|h_w|^2=\eta(h)W_p(Dh),
 \qquad \int_{S_h}|\Phi_p|\area\leq\E_p(h).
\end{equation}
On a natural-coordinate patch $\zeta=x+iy$ where $\Phi_p\,dw^2=d\zeta^2$, write $H=h\circ w(\zeta)$. Then, almost everywhere,
\begin{equation}\label{rankone-natural-form}
 H_y=i\frac{1-r}{1+r}H_x,
 \qquad \eta(H)|H_x|^2=\frac4{Q_p(r)}.
\end{equation}
In particular, on $S_h$ these identities become
\[
 H_y=0,\qquad \sqrt{\eta(H)}|H_x|=2.
\]
If the variational cutoff holds, the other branch instead satisfies
\[
 \sqrt{\eta(H)}|H_y|\geq
 2\frac{1-r_{\mathrm v}}{1+r_{\mathrm v}}>0.
\]
\end{proposition}
\begin{proof}
At $r=1$, the Hopf prefactor $\tfrac12[p+(2-p)r^2]r^{p-2}$ equals one. This proves \eqref{rankone-energy-hopf}. In the natural coordinate the Hopf value is $1$, so $H_{\bar\zeta}=rH_\zeta$ and $\eta(H)|H_\zeta|^2F_p(r)=1$. The identities follow from $H_x=(1+r)H_\zeta$ and $H_y=i(1-r)H_\zeta$. There $r>0$ almost everywhere because the Hopf value is nonzero. Finally $Q_p(r)\leq1$ and $(1-r)/(1+r)\geq(1-r_{\mathrm v})/(1+r_{\mathrm v})$ on the nonsingular branch.
\end{proof}

If $\Phi_p\equiv0$, then $h_{\bw}=0$ almost everywhere and $h$ is conformal. Otherwise the discrete zero set of $\Phi_p$ cannot contain $S_h$ up to a null set. The vertical kernel direction does not itself establish a nontrivial fibre; a Sobolev right inverse supplies the needed additional information.

\begin{corollary}[NoHairs with a Sobolev right inverse]\label{sobolev-right-inverse-cutoff}
Let $h$ satisfy Theorem~\ref{weak-variational-cutoff}. If $h$ is a homeomorphism and $h^{-1}\in W^{1,1}_{\loc}(\D)$, then $|S_h|=0$ and $h$ is $K_{\mathrm v}(p)$-quasiconformal. The same conclusion follows if there is a continuous $g\in W^{1,1}_{\loc}(\D,\D)$ with $h\circ g=\id_\D$.
\end{corollary}
\begin{proof}
A planar homeomorphism whose map and inverse are locally Sobolev has $Dh=0$ almost everywhere on $\{J_h=0\}$; see \cite[Theorem 3.4]{PratelliBiSobolev}, applied locally. Thus the first assertion follows from Corollary~\ref{cutoff-finite-distortion}. For the second, $h(\D)\subset\D$ and the homeomorphic boundary trace give $h^{-1}(\Sone)=\Sone$. Lemma~\ref{hair-homeomorphism} makes $h$ a homeomorphism with inverse $g$, reducing to the first assertion.
\end{proof}

The Sobolev inverse condition is not part of the definition of $\mathcal M_\beta$. Uniform forward $W^{1,1}$ bounds alone do not establish it, since a limit may have a singular distributional derivative.

The same corollary strengthens the boundary obstruction as some smooth quasisymmetric data with finite-energy barriers have no minimiser in the inverse homeomorphic class with Sobolev forward inverse. Indeed any such minimiser would have the universal bound $K_{\mathrm v}(p)$, whereas the boundary maps constructed in Proposition~\ref{lower-range-boundary-obstruction} can require a larger bound. This is a nonattainment conclusion for that specified class and not a nonexistence conclusion in the enlarged monotone class.

\begin{example}[A hair-free flat minimiser with nonzero rank-one gradients]\label{flat-rankone-minimiser}
The Sobolev inverse condition cannot be replaced by injectivity alone, even if the same subquadratic Hopf equation and the variational cutoff both hold. Here take the flat weight $\eta\equiv1$ and the square $Q=(0,1)^2$. Let
\[
 \tau_{\mathrm v}=\frac{1-r_{\mathrm v}}{1+r_{\mathrm v}},\qquad
 E\Subset(0,1)
\]
be a compact nowhere-dense set of positive measure, and define
\[
 g(y)=\tau_{\mathrm v}\int_0^y\boldsymbol1_{(0,1)\setminus E}(t)\,dt,
 \qquad h(x,y)=x+i g(y).
\]
Every interval meets the open complement of $E$, so $g$ is strictly increasing. Thus $h$ is a Lipschitz homeomorphism onto a rectangle, all fibres are singletons, and
\[
 Dh=\operatorname{diag}(1,\tau_{\mathrm v})\quad\hbox{or}\quad
 Dh=\operatorname{diag}(1,0)
\]
almost everywhere. The second alternative occurs on $(0,1)\times E$, which has positive area and empty interior. At both derivatives the weighted Ahlfors--Hopf value is $Q_p(r)/4=1/4$, so the differential is holomorphic.

\medskip

In fact $h$ minimises $\int_Q W_p(Du)$ among all $W^{1,2}$ maps with its trace and $J_u\geq0$. To verify this, put $m_p=\phi_p'(\tau_{\mathrm v})<0$. The common-tangent identity and its supporting inequality are
\begin{equation}\label{common-tangent-support}
 \phi_p(\tau_{\mathrm v})=\tfrac14+m_p\tau_{\mathrm v},
 \qquad \phi_p(t)\geq\tfrac14+m_pt\quad(0\leq t\leq1).
\end{equation}
For completeness, $\phi_p''$ is negative before
$\tau_{\mathrm c}=(1-r_{\mathrm c})/(1+r_{\mathrm c})$ and positive after it. Since $\tau_{\mathrm v}>\tau_{\mathrm c}$, $\phi_p'$ first decreases and then increases. The common-tangent identity implies $m_p<\phi_p'(0)$; otherwise integrating $\phi_p'-m_p$ from $0$ to $\tau_{\mathrm v}$ would give a negative value instead of zero. The function $\phi_p(t)-1/4-m_pt$ therefore rises and then falls to zero at $\tau_{\mathrm v}$, and increases thereafter. This proves the inequality.

For a matrix with singular values $\lambda\geq\nu\geq0$ and nonnegative determinant, \eqref{common-tangent-support} gives
\[
 W_p(A)\geq m_p\det A+\tfrac14\|A\|_{\mathrm{op}}^2
            \geq m_p\det A+\tfrac14|Ae_1|^2.
\]
The integral of the Jacobian is fixed by the $W^{1,2}$ boundary trace. Also $\int_Q u_x=(1,0)|Q|$ for every competitor. Hence
\[
 \int_Q W_p(Du)\geq m_p\int_Q J_h+\tfrac14\int_Q|u_x|^2
             \geq m_p\int_QJ_h+\tfrac14|Q|
             =\int_Q W_p(Dh).
\]
This proves minimality. The inverse is not Sobolev since $g$ maps the positive-measure set $E$ to a null set, so $g^{-1}$ is not absolutely continuous. This flat example establishes the independence of topological injectivity and the rank-one exclusion, even for genuine minimisers.  
\end{example}

\subsection{Proof of Theorem~\ref{thm-A}}
\begin{proof}[Proof of Theorem~\ref{thm-A}]
(i) is Corollary~\ref{sharp-polyconvexity-range}. (ii) is Theorem~\ref{monotone-existence} and Corollary~\ref{strong-monotone-minimising}; the barrier gives the finite-energy competitor, and quasisymmetry of $\beta$ follows from Proposition~\ref{finite-energy-boundary-moduli} when the barrier is the inverse of a finite-distortion extension. (iii): the cutoff is Theorem~\ref{weak-variational-cutoff}; the dichotomy, with the regularity of the finite-distortion alternative, is Corollary~\ref{attained-minimum-dichotomy}, which uses Corollaries~\ref{cutoff-finite-distortion} and \ref{qc-minimiser-regularity}; the boundary data forcing the second alternative are those of Proposition~\ref{lower-range-boundary-obstruction}, since a bound larger than $K_{\mathrm c}(p)$ is larger than $K_{\mathrm v}(p)$. (iv) is Corollary~\ref{endpoint-nonattainment} and the example following it.
\end{proof}

\section{Scalar equations for regular solutions}\label{scalar-equations}
We first treat the functional
\begin{equation}\label{subquadratic-functional}
 \I_p(f)=\int_\D\frac{|\mu_f|^p}{1-|\mu_f|^2}\eta\area,
 \qquad 1<p\leq2,\qquad \eta(z)=(1-|z|^2)^{-2}.
\end{equation}
In particular,
\[
 \int_\D(\K_f-1)\eta\area
 =2\int_\D\frac{|\mu_f|^2}{1-|\mu_f|^2}\eta\area
 \leq2\I_p(f).
\]
Thus the asserted integrability of $\K_f-1$ follows pointwise, including where $|\mu_f|$ is small. For $h=f^{-1}$, write $a=h_w$, $b=h_{\bw}$ and $r=|b/a|$. The inverse formula becomes
\begin{equation}\label{subquadratic-inverse}
 \E_p(h)=\int_\D\eta(h)|a|^2r^p\area=\I_p(f),
 \qquad \mathcal D(h)\leq\pi+2\E_p(h).
\end{equation}

For a regular stationary diffeomorphism, a scalar estimate for $|\mu_h|$ gives global quasiconformality without uniform convergence of first derivatives. The iteration also handles the critical exponent. Section~\ref{above-two} treats $p>2$.

\subsection{Vanishing energy in a divergent M\"obius normalisation}
\begin{lemma}\label{escaping-normalisations}
Let $h:\overline\D\to\overline\D$ be an orientation-preserving homeomorphism, a $C^1$ diffeomorphism in $\D$, with quasisymmetric boundary values and $\E_p(h)<\infty$. Suppose $w_j\to\partial\D$. Choose disk automorphisms $A_j,B_j$ so that
\[
 A_j(0)=w_j,\qquad h_j=B_j\circ h\circ A_j,
\]
and $h_j|_{\Sone}$ fixes $1,i,-i$. Then
\begin{equation}\label{escaping-convergence}
 h_j\longrightarrow\mathrm{id}\quad\hbox{uniformly on }\overline\D
 \quad\hbox{and strongly in }W^{1,2}_{\loc}(\D),
 \qquad r_j:=|\mu_{h_j}|\longrightarrow0\quad\hbox{in }L^2_{\loc}.
\end{equation}
If the differentials $\Phi_p(h_j)$ from \eqref{variation} are holomorphic, then
\begin{equation}\label{escaping-hopf}
 \Phi_p(h_j)\longrightarrow0\quad\hbox{locally uniformly}.
\end{equation}
\end{lemma}
\begin{proof}
M\"obius invariance and \eqref{subquadratic-inverse} give a common Dirichlet bound for all $h_j$. A fixed quasiconformal extension of $h|_{\Sone}$, subjected to the same normalisations, gives extensions with a common quasiconformal bound fixing three points. Proposition~\ref{boundary-modulus} gives equicontinuity on the closed disk. Their boundary maps, after extraction, converge to a quasisymmetric homeomorphism fixing those points.

For every $0<R<1$, invariance holds locally as well:
\begin{equation}\label{escaping-energy}
 \E_p(h_j;\D_R)=\E_p(h;A_j(\D_R))\longrightarrow0.
\end{equation}
Indeed $A_j(\D_R)$ is a hyperbolic ball of fixed radius whose centre escapes to infinity; it eventually avoids each compact subset of $\D$. The original inverse-energy density is integrable. Since $r_j^2\leq r_j^p$ and $\eta\geq1$, this gives $(h_j)_{\bw}\to0$ in $L^2_{\loc}$. Any uniform subsequential limit is therefore holomorphic, with a boundary homeomorphism of degree one. It is a disk automorphism, and the three fixed boundary points make it the identity. This proves convergence of the whole sequence uniformly on $\overline\D$.

For a compactly supported cutoff $\chi$, put $u_j=\chi(h_j-\mathrm{id})$, extended by zero to the plane. Then
\[
 \|(u_j)_w\|_{L^2(\mathbb C)}
 =\|(u_j)_{\bw}\|_{L^2(\mathbb C)}\longrightarrow0.
\]
This proves strong local convergence of the gradients. On a fixed compact set,
\[
 \int r_j^2\area
 \leq4\int|(h_j)_{\bw}|^2\area
       +4\int|(h_j)_w-1|^2\area\longrightarrow0;
\]
split the set into $|(h_j)_w|\geq1/2$ and its complement to obtain this estimate.

Put $G_j=\eta(h_j)|(h_j)_w|^2$. Uniform convergence to the identity and the Dirichlet bound give a uniform local $L^1$ bound for $G_j$. In the normalisation \eqref{variation},
\[
 |\Phi_p(h_j)|
 =G_j\frac{r_j^{p-1}\bigl(p+(2-p)r_j^2\bigr)}2
 \leq G_jr_j^{p-1}.
\]
H\"older's inequality and \eqref{escaping-energy} give, for each compact disk $C\Subset\D$,
\[
 \int_C|\Phi_p(h_j)|\area
 \leq\left(\int_CG_jr_j^p\area\right)^{(p-1)/p}
       \left(\int_CG_j\area\right)^{1/p}\longrightarrow0.
\]
The local mean-value estimate for holomorphic functions proves \eqref{escaping-hopf}.
\end{proof}

This lemma uses the boundary compactness mechanism from NoHairs, but does not need a separate no-hairs argument for these particular limits: vanishing local energy makes the limit conformal directly. Nor is stability of global minimality under varying boundary values required.

\subsection{The scalar equation from hyperbolic curvature}
Assume that $h$ is a $C^2$ orientation-preserving diffeomorphism and that
\begin{equation}\label{subquadratic-hopf}
 \Phi=\frac12\eta(h)
       \bigl(p+(2-p)r^2\bigr)r^{p-2}h_w\overline{h_{\bw}}
\end{equation}
is holomorphic, with value zero at $r=0$. It suffices to assume $C^2$ on $\{r>0\}$ and $C^1$ diffeomorphic regularity throughout; all subsequent scalar tests are supported above a positive level of $r$.

Define
\[
 F_p,\ H_p\quad\text{are the functions in }\eqref{flux-functions}.
\]
On a patch where $\Phi\ne0$, take a natural coordinate $\zeta=x+iy$ with $d\zeta^2=\Phi\,dw^2$. In this coordinate the differential equals $1$, hence $\mu_h=r>0$. With $U=h_\zeta$,
\[
 h_x=(1+r)U,\qquad h_y=i(1-r)U,\qquad
 \eta(h)|U|^2F_p(r)=1.
\]
Consequently the pullback of the curvature-$-4$ target metric is
\begin{equation}\label{orthogonal-pullback}
 \eta(h)|dh|^2=L(r)^2\,dx^2+M(r)^2\,dy^2,
 \qquad L(r)=\frac{1+r}{\sqrt{F_p(r)}},\quad
 M(r)=\frac{1-r}{\sqrt{F_p(r)}}.
\end{equation}
For an orthogonal metric $L^2dx^2+M^2dy^2$, its Gaussian curvature is
\[
 -\frac1{LM}\left\{\partial_x\left(\frac{M_x}{L}\right)
                       +\partial_y\left(\frac{L_y}{M}\right)\right\}.
\]
Equating it to $-4$ gives
\begin{equation}\label{scalar-natural}
 \partial_x\bigl(\alpha_p(r)r_x\bigr)
 +\partial_y\bigl(\beta_p(r)r_y\bigr)
 =-4\frac{1-r^2}{F_p(r)},
\end{equation}
where direct differentiation gives
\begin{equation}\label{scalar-coefficients}
 \boxed{\alpha_p(r)=-\frac{M'(r)}{L(r)}
   =\frac{H_p(-r)}{2r\,[p+(2-p)r^2]},\qquad
 \beta_p(r)=-\frac{L'(r)}{M(r)}
   =\frac{H_p(r)}{2r\,[p+(2-p)r^2]}.}
\end{equation}
These identities recover the earlier ellipticity polynomial by a separate geometric calculation.

For $C^2$ maps the curvature computation is a distributional identity and does not require third derivatives. The orthonormal coframe $L\,dx,M\,dy$ has connection form
\[
 \omega=\frac{L_y}{M}\,dx-\frac{M_x}{L}\,dy.
\]
This is the pulled-back target connection, up to the differential of the local angle of $U$. Its exterior derivative is therefore $-4LM\,dx\wedge dy$; the exterior derivative of that angle differential is zero in distributions. This gives \eqref{scalar-natural} in the weak form used below.

To return to the original coordinate, let $\theta=\frac12\arg\Phi$ and let $R_\theta$ denote planar rotation by $\theta$. Set
\[
 \mathsf A_p(w)=R_\theta^{T}
          \begin{pmatrix}\alpha_p(r)&0\\0&\beta_p(r)\end{pmatrix}
          R_\theta.
\]
Changing the square-root branch changes $R_\theta$ by a sign and leaves this tensor unchanged. Conformal change of variables in the divergence equation gives
\begin{equation}\label{scalar-original}
 \boxed{-\operatorname{div}(\mathsf A_p\nabla r)
       =4|\Phi|\frac{1-r^2}{F_p(r)}.}
\end{equation}
No derivative of $\Phi$ is needed for these estimates as its argument only determines the eigendirections of a measurable symmetric coefficient matrix. The equation is used on $\{r>k\}$ for $k>0$, which avoids every zero of $\Phi$.

\subsection{The allowed branch and its possible degeneracy}
Recall that for $1<p<p_{\mathrm c}$, $r_{\mathrm c}(p)\in(0,1)$ is the zero of $H_p$,
\[
 r_{\mathrm c}(p)=\frac{\sqrt{p(2-p)}-p(2-p)}{(p-1)(2-p)},
\]
and put $r_{\mathrm c}(p)=1$ for $p_{\mathrm c}\leq p\leq2$. For $p\geq p_{\mathrm c}$, $H_p(r)\geq0$ on $[0,1]$. For $p<p_{\mathrm c}$, it is nonnegative precisely on $[0,r_{\mathrm c}(p)]$. The rank-one second variation established above shows that any $C^1$ diffeomorphic local minimiser necessarily satisfies
\begin{equation}\label{stable-distortion-branch}
 r\leq r_{\mathrm c}(p)\quad (p<p_{\mathrm c}).
\end{equation}
This use of minimality is essential below the threshold; an arbitrary stationary solution need not satisfy it. The regular barrier family of Theorem~\ref{lambda-uniform-selection} obtains its uniform branch restriction by a separate finite-variation argument.

For every fixed $0<k<r_{\mathrm c}(p)$, on $k\leq r\leq r_{\mathrm c}(p)$ the coefficients satisfy
\begin{equation}\label{scalar-degenerate-bounds}
 c_p\bigl(r_{\mathrm c}(p)-r\bigr)I\ \leq\ \mathsf A_p
 \ \leq\ \Lambda_{p,k}I.
\end{equation}
For $p<2$, one can take $c_p=p(2-p)/2$. Indeed $H_p(r_{\mathrm c})\geq0$,
$-H_p'(r)\geq2p(2-p)$, the denominator in \eqref{scalar-coefficients} is at most $4$, and $\alpha_p\geq\beta_p$. For $p=2$, $\alpha_2=\beta_2=1/(2r)$ and one may take $c_2=1/2$.

Above $p_{\mathrm c}$ the lower bound can be strengthened to a positive constant on $[k,1]$. At $p=p_{\mathrm c}$, and at $r=r_{\mathrm c}(p)$ below that exponent, one eigenvalue vanishes linearly. The following explicit iteration handles this degeneration; it is not an invocation of a uniformly elliptic estimate at those endpoints.

\subsection{A local estimate allowing linear degeneracy at the upper value}
\begin{lemma}\label{ceiling-iteration}
Let $0\leq v\leq a$ belong to $W^{1,2}(B_1)$, where $a>0$ is fixed. Suppose a measurable symmetric matrix $M$ satisfies
\[
 c(a-v)I\leq M\leq\Lambda I,
 \qquad -\operatorname{div}(M\nabla v)\leq\varepsilon
\]
weakly, with $c>0$, $\Lambda<\infty$ and $0\leq\varepsilon\leq1$. Then
\begin{equation}\label{ceiling-bound}
 \|v\|_{L^\infty(B_{1/2})}
 \leq C(a,c,\Lambda)\bigl(\|v\|_{L^2(B_1)}+\varepsilon\bigr).
\end{equation}
\end{lemma}
\begin{proof}
The proof modifies the classical power iteration \cite{Moser1960}; the modification is given explicitly because the lower ellipticity bound may vanish. Put $u=v+\varepsilon+\delta$, $A=a+\varepsilon+\delta$, where $0<\delta<1$. Then $M\geq c(A-u)I$ and $\varepsilon\leq u$. Test the subsolution inequality with $\chi^2u^{q-1}$, $q\geq2$. Cauchy--Schwarz in the nonnegative quadratic form $M$ and Young's inequality give
\begin{align}\label{ceiling-caccioppoli}
 \int\chi^2u^{q-2}(A-u)|\nabla u|^2
 &\leq\frac{4\Lambda}{c(q-1)^2}
                  \int u^q|\nabla\chi|^2
       +\frac{2}{c(q-1)}\int\chi^2u^q.
\end{align}
Define
\[
 G_q(t)=\int_0^t s^{q/2-1}\sqrt{A-s}\,ds,\qquad0\leq t\leq A.
\]
There is an absolute constant $c_0>0$ such that
\begin{equation}\label{ceiling-transform}
 c_0\sqrt A\,q^{-3/2}t^{q/2}
 \leq G_q(t)\leq\frac{2\sqrt A}{q}t^{q/2}.
\end{equation}
For the lower bound substitute $s=ty$ and use
$\int_0^1 y^{q/2-1}(1-y)^{1/2}dy=B(q/2,3/2)\asymp q^{-3/2}$.

The left side of \eqref{ceiling-caccioppoli} is $\int\chi^2|\nabla G_q(u)|^2$. Apply the planar Sobolev inequality to $\chi G_q(u)$ and use both bounds in \eqref{ceiling-transform}. For $1/2\leq s<t\leq1$, with an appropriate cutoff,
\[
 \|u\|_{L^{2q}(B_s)}
 \leq\left(\frac{Cq^2}{(t-s)^2}\right)^{1/q}
       \|u\|_{L^q(B_t)}.
\]
Here $C$ depends on $a,c,\Lambda$, but not on $q,\varepsilon,\delta$: $a\leq A\leq a+2$. Iterate with $q=2,4,8,\ldots$ and radii decreasing from $1$ to $1/2$. The product of these constants is finite. Finally let $\delta\downarrow0$. This proves \eqref{ceiling-bound}.
\end{proof}

\begin{corollary}[Local control of large distortion]\label{scalar-distortion-bound}
Suppose $h$ satisfies \eqref{subquadratic-hopf}, with the regularity stated there and $r\leq r_{\mathrm c}(p)$. For every $0<k<r_{\mathrm c}(p)$ and $B_{2\rho}(w_0)\Subset\D$,
\begin{equation}\label{scalar-local-bound}
 \sup_{B_\rho(w_0)}(r-k)_+
 \leq C_{p,k}\left[
       \rho^{-1}\|r\|_{L^2(B_{2\rho}(w_0))}
       +\rho^2\|\Phi\|_{L^\infty(B_{2\rho}(w_0))}\right].
\end{equation}
\end{corollary}
\begin{proof}
On $\{r>k\}$, the right side of \eqref{scalar-original} is nonnegative and bounded above by $C_{p,k}|\Phi|$. Put $v=(r-k)_+$ and $a=r_{\mathrm c}(p)-k$. Extend $\mathsf A_p$ measurably to $\{r\leq k\}$ as a fixed positive scalar matrix. Then \eqref{scalar-degenerate-bounds} becomes $c_p(a-v)I\leq\mathsf A_p\leq\Lambda I$ everywhere. Positive-part truncation gives
\[
 -\operatorname{div}(\mathsf A_p\nabla v)
 \leq C_{p,k}\|\Phi\|_\infty
\]
weakly on the ball. To justify the sign, test \eqref{scalar-original} with an increasing smooth approximation to $\mathbf1_{\{r>k\}}$ times a nonnegative test function; the extra term is a nonnegative quadratic form and is subtracted from the right side. No control of the eigendirections or their derivatives is required. Rescaling and Lemma~\ref{ceiling-iteration} prove \eqref{scalar-local-bound}. If the rescaled source exceeds $1$, the assertion follows directly from $0\leq v\leq a$ after increasing the constant.
\end{proof}

\subsection{Consequence for smooth minimisers}
\begin{proposition}[Conditional asymptotic conformality for $1<p\leq2$]\label{subquadratic-asymptotic}
Let $h:\overline\D\to\overline\D$ be an orientation-preserving homeomorphism with quasisymmetric boundary values and $\E_p(h)<\infty$. Assume that $h$ is a $C^1$ diffeomorphism in $\D$, is $C^2$ on $\{|\mu_h|>0\}$, and has holomorphic differential \eqref{subquadratic-hopf}. If $p<p_{\mathrm c}$, assume additionally $|\mu_h|\leq r_{\mathrm c}(p)$; this condition is necessary for a diffeomorphic local minimiser. Then
\begin{equation}\label{subquadratic-asymptotic-conclusion}
 \lim_{|w|\to1}|\mu_h(w)|=0.
\end{equation}
In particular $h$ and $f=h^{-1}$ are globally quasiconformal and asymptotically conformal. Thus every smooth diffeomorphic minimiser of \eqref{subquadratic-functional} with quasisymmetric boundary values has these properties, provided the established inverse/variation calculus applies.
\end{proposition}
\begin{proof}
Take any sequence $w_j\to\partial\D$ and the normalisations in Lemma~\ref{escaping-normalisations}. They preserve the energy, the stationarity equation and the branch bound. That lemma gives $r_j\to0$ in $L^2_{\loc}$ and $\Phi_p(h_j)\to0$ locally uniformly. For each fixed $0<k<r_{\mathrm c}(p)$, apply \eqref{scalar-local-bound} on two fixed concentric disks about $0$. It yields
\[
 \limsup_{j\to\infty}r_j(0)\leq k.
\]
Let $k\downarrow0$. Since $r_j(0)=|\mu_h(w_j)|$, this proves \eqref{subquadratic-asymptotic-conclusion}. On each compact subset of $\D$, continuity and the positive Jacobian give a strict bound $|\mu_h|<1$. Combine this with the boundary limit to obtain $\|\mu_h\|_\infty<1$. Finally,
$|\mu_f(z)|=|\mu_h(f(z))|$ and properness of $f$ transfer both conclusions to $f$.
\end{proof}
\subsection{A closed scalar equation on the forward side}\label{forward-scalar-equation}
The same threshold is visible in a scalar equation for a natural-coordinate component of the forward map. This reduction retains the hyperbolic weight and identifies the principal eigenvalues exactly. Here we use the unregularised differential $\Phi_p$ of \eqref{monotone-hopf}, rather than the half-normalised $q_\lambda$ in Appendix~\ref{other-formulations}.

\begin{proposition}[Forward scalar reduction]\label{forward-scalar-reduction}
Let $1<p\leq2$, and let $h=f^{-1}$ be a $C^2$ orientation-preserving diffeomorphic solution of the Hopf equation. On a simply connected patch where $\Phi_p\ne0$, choose a natural coordinate $\zeta$ with $\zeta'(w)^2=\Phi_p(w)$. Write
\[
 g=\zeta\circ f=u+iv,\qquad r=|\mu_f|,
 \qquad F_p\text{ as in }\eqref{flux-functions}.
\]
After restricting the patch so that these maps are one-to-one, define the strictly increasing function
\begin{equation}\label{forward-scalar-parameter}
 \Theta_p(r)=\frac{\sqrt{F_p(r)}}{1-r},\qquad 0<r<1,
 \qquad
 A_p\bigl(\Theta_p(r)\bigr)=\frac{1-r}{1+r}.
\end{equation}
Then
\begin{equation}\label{forward-closed-scalar}
 t=\frac{|\nabla v|}{\sqrt{\eta(z)}}=\Theta_p(r),
 \qquad
 \boxed{\operatorname{div}\!\left(
 A_p\!\left(\frac{|\nabla v|}{\sqrt{\eta(z)}}\right)\nabla v
 \right)=0.}
\end{equation}
The eigenvalues of its principal matrix, perpendicular and parallel to $\nabla v$, are $A_p(t)$ and $A_p(t)+tA_p'(t)$, respectively, and
\begin{equation}\label{forward-scalar-eigenvalues}
 \boxed{\frac{A_p(t)+tA_p'(t)}{A_p(t)}
 =\left(\frac{1-r}{1+r}\right)^2\frac{H_p(r)}{H_p(-r)}.}
\end{equation}
In particular the scalar equation is elliptic exactly where $H_p(r)>0$.
\end{proposition}
\begin{proof}
The inverse of $g$ has Hopf coefficient $1$. The inverse derivative formula therefore gives
\[
 g_{\bar z}=-r\,\overline{g_z},\qquad
 |g_z|^2=\frac{\eta(z)F_p(r)}{(1-r^2)^2}.
\]
Writing $k=(1+r)/(1-r)$, the real and imaginary parts of the first identity are
\[
 v_y=k u_x,\qquad v_x=-k u_y.
\]
Consequently $|\nabla v|=(1+r)|g_z|$ and
$\operatorname{div}(k^{-1}\nabla v)=0$. These are \eqref{forward-closed-scalar}. Since $F_p'>0$, the parameter $\Theta_p$ increases from $0$ to $\infty$, so this is a closed equation in $v$ and the prescribed weight.

For the principal matrix, differentiating the flux with respect to $\nabla v$ gives
\[
 A_p(t)I+tA_p'(t)\,e\otimes e,
 \qquad e=\frac{\nabla v}{|\nabla v|}.
\]
Derivatives of $\eta(z)$ contribute lower-order terms. With
$C_p=p+(2-p)r^2$ and $D_p=p(p-1)+(2-p)(p+1)r^2$, use
\[
 \frac{\Theta_p'}{\Theta_p}
 =\frac{D_p}{2rC_p}+\frac1{1-r},\qquad
 \frac{d}{dr}\log\frac{1-r}{1+r}=-\frac2{1-r^2}.
\]
The factorisations
\[
 (1+r)D_p-2rC_p=(1-r)H_p(r),\qquad
 (1-r)D_p+2rC_p=(1+r)H_p(-r)
\]
prove \eqref{forward-scalar-eigenvalues}. Here $H_p(-r)>0$ for $1<p\leq2$ and $0<r<1$.
\end{proof}

The radial flux is particularly transparent:
\[
 tA_p(t)=\frac{\sqrt{F_p(r)}}{1+r}.
\]
Its derivative with respect to $t$ is the second principal eigenvalue. Positivity of $A_p$ alone is therefore insufficient for nonlinear ellipticity. At $p=3/2$, $r=1/2$, the eigenvalues are $1/3$ and $-1/621$. In coordinates with $\nabla v$ parallel to the $y$-axis, the frozen principal part is proportional to $v_{xx}-v_{yy}/207$.

A prescribed smooth nonvanishing multiplier of $f$ cannot remove such characteristic directions: multiplication changes the principal symbol by a nonsingular factor, and its derivatives enter at lower order. The same holds for invertible point transformations of the dependent variables. A derivative-dependent substitution would require a separate equivalence argument. Freezing $A_p$ along a known locally quasiconformal solution does give a linear elliptic conductivity equation, but its coefficients contain the unknown gradient; the nonlinear regularity test remains \eqref{forward-scalar-eigenvalues}.

\paragraph{The quadratic minimal-graph equation.}
For $p=2$, $F_2(r)=r$ and
\[
 A_2(t)=(1+4t^2)^{-1/2}.
\]
Equation~\eqref{forward-closed-scalar} is then the minimal-graph equation for height $2v$ over the metric $\eta(z)|dz|^2$. More generally it is the Euler equation for the local scalar integral
\[
 \int\eta(z)\,\mathcal V_p\!\left(\frac{|\nabla v|}{\sqrt{\eta(z)}}\right)\area,
 \qquad \mathcal V_p'(t)=tA_p(t),\quad\mathcal V_p(0)=0.
\]
At $p=2$, $\mathcal V_2(t)=(\sqrt{1+4t^2}-1)/4$, one quarter of the graph-area excess density. This recovers the harmonic-map/minimal-graph correspondence \cite{WolfMinimal}. This scalar variational representation is only local.

\begin{remark}[The large-gradient structure of the forward equation]\label{forward-large-gradient-structure}
For each fixed $p_{\mathrm c}<p\leq2$, set $\ell_p=2(p-1)^2-1>0$. The two principal eigenvalues have the asymptotics
\begin{equation}\label{forward-large-gradient-eigenvalues}
 A_p(t)\sim\frac1{2t},\qquad
 A_p(t)+tA_p'(t)\sim\frac{\ell_p}{8t^3}
 \quad(t\to\infty).
\end{equation}
These are the same large-gradient orders as for the minimal-graph equation, with a positive radial coefficient precisely above $p_{\mathrm c}$. At the critical exponent the leading radial term vanishes, and instead
\begin{equation}\label{forward-critical-large-gradient}
 A_{p_{\mathrm c}}(t)+tA_{p_{\mathrm c}}'(t)
 \sim\frac1{8\sqrt2\,t^4}.
\end{equation}
To verify the constants, put $\delta=1-r$. Then $t\delta\to1$, $A_p(t)\sim\delta/2$, $H_p(-1)=2$, and $H_p(1)=2\ell_p$. At $p=p_{\mathrm c}$, one has $H_p(1-\delta)\sim\sqrt2\,\delta$. Substitution into \eqref{forward-scalar-eigenvalues} proves both assertions.

Proposition~\ref{lambda-uniform-interior-gradient} gives the corresponding height estimate uniformly for $\lambda_0\leq\lambda\leq1$.

\end{remark}
\subsection{Uniform scalar structure for the $\lambda$ approximation}\label{lambda-scalar-structure}
We record the scalar coefficients for the regular approximants constructed in Theorem~\ref{selected-rational-regularity}.

For $0<\lambda\leq1$, put $C_{p,\lambda}(r)=p\lambda+(2-p)r^2$ and use the $\lambda$-Hopf equations.
\[
 F_{p,\lambda}(r)
 =\frac12r^{p-1}\frac{(1-r^2)^2C_{p,\lambda}(r)}{(\lambda-r^2)^2},
 \qquad 0<r<\sqrt\lambda.
\]
At $\lambda=1$ this is $F_p$. Define
\begin{equation}\label{lambda-scalar-parameter}
 \Theta_{p,\lambda}(r)=\frac{\sqrt{F_{p,\lambda}(r)}}{1-r},
 \qquad
 A_{p,\lambda}\bigl(\Theta_{p,\lambda}(r)\bigr)
       =\frac{1-r}{1+r}.
\end{equation}
Each $\Theta_{p,\lambda}$ increases from zero to infinity. On a regular patch without Hopf zeros, take $\zeta_\lambda'^2=\Phi_{p,\lambda}$ and write $\zeta_\lambda\circ f_\lambda=u_\lambda+i v_\lambda$. The proof of Proposition~\ref{forward-scalar-reduction} gives
\begin{equation}\label{lambda-forward-scalar-equation}
 \operatorname{div}\!\left(A_{p,\lambda}(t_\lambda)
                              \nabla v_\lambda\right)=0,
 \qquad t_\lambda=\frac{|\nabla v_\lambda|}{\sqrt\eta}
                  =\Theta_{p,\lambda}(r_\lambda).
\end{equation}

For every $T<\infty$ and fixed $p_{\mathrm c}<p\leq2$,
\begin{equation}\label{lambda-scalar-operator-limit}
 A_{p,\lambda}\longrightarrow A_p
 \quad\hbox{in }C^1([0,T])\qquad(\lambda\uparrow1).
\end{equation}
Away from $t=0$, this follows by smoothly inverting the strictly increasing profile $\Theta_{p,\lambda}$. At zero, use
\[
 \Theta_{p,\lambda}^{-1}(t)
 =(2\lambda/p)^{1/(p-1)}t^{2/(p-1)}
   \bigl(1+O_p(t^{2/(p-1)})\bigr),
\]
with the corresponding differentiated expansion uniform for $\lambda$ near one. Thus the second-order operators converge locally uniformly in the point, gradient and Hessian variables. The bounds below give a positive ellipticity constant independent of $\lambda$ on every such bounded gradient range. The limiting equation is \eqref{forward-closed-scalar}; at $p=2$ it is explicitly
\begin{equation}\label{quadratic-forward-scalar-limit}
 \operatorname{div}\left(
 \frac{\nabla v}{\sqrt{1+4|\nabla v|^2/\eta}}\right)=0.
\end{equation}
This scalar equation for a Hopf-coordinate component of the forward map is consistent with hyperbolic harmonicity of its inverse.

\begin{proposition}[Scalar ellipticity bounds independent of the barrier parameter]\label{lambda-uniform-scalar-structure}
Let $p_{\mathrm c}<p\leq2$, $\ell_p=2(p-1)^2-1$, and $0<\lambda\leq1$. For every $t\geq0$, set
\[
 B_{p,\lambda}(t)=A_{p,\lambda}(t)+tA_{p,\lambda}'(t),
 \qquad \vartheta_{p,\lambda}(t)=\frac{B_{p,\lambda}(t)}{A_{p,\lambda}(t)}.
\]
Then
\begin{align}\label{lambda-uniform-scalar-bounds}
 \frac1{\sqrt{1+4t^2}}&\leq A_{p,\lambda}(t)\leq1,\notag\\
 \frac{\ell_p}{1+4t^2}&\leq\vartheta_{p,\lambda}(t)\leq1,\\
 \frac{\ell_p}{(1+4t^2)^{3/2}}&\leq B_{p,\lambda}(t)\leq A_{p,\lambda}(t).
 \notag
\end{align}
All constants are independent of $\lambda$. At $p=p_{\mathrm c}$, the first line still holds, and with $c_*=p_{\mathrm c}(p_{\mathrm c}-1)/2$ the other lower bounds can be replaced by
\begin{equation}\label{lambda-critical-scalar-bounds}
 \vartheta_{p_{\mathrm c},\lambda}(t)
       \geq\frac{c_*}{(1+4t^2)^{3/2}},\qquad
 B_{p_{\mathrm c},\lambda}(t)
       \geq\frac{c_*}{(1+4t^2)^2}.
\end{equation}
\end{proposition}
\begin{proof}
For fixed $r$ in the allowed branch,
\[
 \partial_\lambda\log F_{p,\lambda}(r)
 =\frac p{C_{p,\lambda}(r)}-\frac2{\lambda-r^2}<0.
\]
Hence $F_{p,\lambda}\geq F_p$. Moreover,
\begin{align*}
 \frac{F_{p,\lambda}'}{F_{p,\lambda}}-\frac{F_p'}{F_p}
 &=2(2-p)r\left(\frac1{C_{p,\lambda}}-\frac1{C_{p,1}}\right)
   +\frac{4r(1-\lambda)}{(\lambda-r^2)(1-r^2)}\geq0.
\end{align*}
In particular $F_{p,\lambda}$ and $\Theta_{p,\lambda}$ are strictly increasing. Direct differentiation of \eqref{lambda-scalar-parameter} gives
\begin{equation}\label{lambda-scalar-ratio-comparison}
 \vartheta_{p,\lambda}(\Theta_{p,\lambda}(r))
 =1-\frac4{(1-r^2)F_{p,\lambda}'(r)/F_{p,\lambda}(r)+2(1+r)}
 \geq\left(\frac{1-r}{1+r}\right)^2\frac{H_p(r)}{H_p(-r)}.
\end{equation}
Thus the regularisation improves the normalised radial ellipticity at a given $r$.

For $p_{\mathrm c}\leq p\leq2$, $F_p(r)\geq r$. For example,
\[
 \frac d{dr}\left(\frac{F_p(r)}r\right)
 =-\frac{p(2-p)}2r^{p-3}(1-r^2)\leq0,
 \qquad F_p(1)=1.
\]
Consequently, at $t=\Theta_{p,\lambda}(r)$ and $A=(1-r)/(1+r)$,
\[
 1+4t^2\geq1+\frac{4r}{(1-r)^2}=A^{-2}.
\]
Also $H_p(r)\geq H_p(1)=2\ell_p$ and $H_p(-r)\leq H_p(-1)=2$. Equation~\eqref{lambda-scalar-ratio-comparison} therefore gives
$\vartheta_{p,\lambda}\geq\ell_p A^2$. These inequalities prove all the lower bounds in \eqref{lambda-uniform-scalar-bounds}. The upper bounds follow because $A$ decreases with $t$ and our radial ratio is at most one.

At the critical exponent,
\[
 H_{p_{\mathrm c}}(r)
 =(1-r)\left[p_{\mathrm c}(p_{\mathrm c}-1)
       +(p_{\mathrm c}-1)(2-p_{\mathrm c})r\right].
\]
Using $1-r=A(1+r)\geq A$ in the same comparison gives
$\vartheta_{p_{\mathrm c},\lambda}\geq c_*A^3$ and
$B_{p_{\mathrm c},\lambda}\geq c_*A^4$, which proves
\eqref{lambda-critical-scalar-bounds}. The values at $t=0$ follow by continuity.
\end{proof}

The weight terms also have a uniform form. Put $e=\nabla v_\lambda/|\nabla v_\lambda|$, $L=\log\eta$ and $\vartheta=\vartheta_{p,\lambda}(t_\lambda)$. Dividing \eqref{lambda-forward-scalar-equation} by $A_{p,\lambda}$ gives
\begin{equation}\label{lambda-normalised-scalar-equation}
 \bigl[I-(1-\vartheta)e\otimes e\bigr]:D^2v_\lambda
 +\frac{1-\vartheta}{2}\nabla L\cdot\nabla v_\lambda=0.
\end{equation}
The principal eigenvalues are $1$ and $\vartheta$, and the drift coefficient is bounded by $|\nabla L|/2$, independently of $\lambda$, on a fixed compact source patch. At zero gradient the expression has its continuous interpretation. Thus the coefficients retain uniform minimal-graph-type lower bounds above $p_{\mathrm c}$, even though the allowed distortion approaches one.

The uniform structural bounds are now proved. Proposition~\ref{lambda-uniform-interior-gradient} obtains $t_\lambda\leq T$ on smaller common patches from bounded scalar height oscillations. This makes the scalar equations uniformly elliptic there, independently of $\lambda$, and gives us the explicit quasiconformal conversion below. We also obtain the local gradient bounds used in our subsequent limiting argument.

\begin{proposition}[An exact separation of the scalar quadratic regularisation]\label{lambda-scalar-viscosity}
Fix $p_{\mathrm c}<p\leq2$ and $0<\lambda_0\leq\lambda\leq1$. Put
\[
 s=\sqrt\lambda,\qquad x=r/s,\qquad
 \varepsilon_\lambda=\frac{1-s}{1+s},\qquad
 \tau=\Theta_p(x),\qquad
 \mathfrak a_{p,\lambda}(t)=tA_{p,\lambda}(t).
\]
Then the exact identities
\begin{align}\label{lambda-scalar-flux-splitting}
 t&=s^{(p-3)/2}\frac{1+sx}{1+x}\,\tau,\notag\\
 \mathfrak a_{p,\lambda}(t)
   &=\varepsilon_\lambda t+\widetilde{\mathfrak a}_{p,\lambda}(t),
 &\widetilde{\mathfrak a}_{p,\lambda}(t)
   &=\frac{2s^{(p-1)/2}}{1+s}\,\mathfrak a_{p,1}(\tau)
     =\frac{2s^{(p-1)/2}}{1+s}\frac{\sqrt{F_p(x)}}{1+x}
\end{align}
hold. The residual flux is increasing and bounded, with
\[
 \widetilde{\mathfrak a}_{p,\lambda}(0)=0,\qquad
 \lim_{t\to\infty}\widetilde{\mathfrak a}_{p,\lambda}(t)
       =\frac{s^{(p-1)/2}}{1+s}.
\]
Let
\[
 \widetilde\varphi_{p,\lambda}(t)
      =\int_0^t\widetilde{\mathfrak a}_{p,\lambda}(q)\,dq,
 \qquad
 \widetilde{\mathcal F}_{p,\lambda}(Q)
      =\widetilde\varphi_{p,\lambda}(|Q|),\quad Q\in\IR^2.
\]
There are $c,C>0$, depending only on $p,\lambda_0$, such that, as quadratic forms,
\begin{equation}\label{lambda-residual-area-comparison}
 cD^2\sqrt{1+|Q|^2}
 \ \leq D^2\widetilde{\mathcal F}_{p,\lambda}(Q)
 \ \leq CD^2\sqrt{1+|Q|^2}.
\end{equation}
In particular the residual scalar energy has linear growth and uniform
minimal-graph Hessian bounds. The full scalar energy is exactly
\begin{equation}\label{lambda-scalar-energy-splitting}
 \int_\Omega\left\{
 \frac{\varepsilon_\lambda}{2}|\nabla_\eta v|^2+
 \widetilde\varphi_{p,\lambda}(|\nabla_\eta v|)
 \right\}\,dA_\eta,
 \qquad \varepsilon_\lambda\sim(1-\lambda)/4.
\end{equation}
\end{proposition}
\begin{proof}
The parameter formula gives
\[
 t=\frac{(1+r)\sqrt{r^{p-1}C_{p,\lambda}(r)/2}}{s^2-r^2}.
\]
Subtracting $\varepsilon_\lambda t$ from $t(1-r)/(1+r)$ cancels the factor
$s-r$ and gives
\[
 t\left(\frac{1-r}{1+r}-\frac{1-s}{1+s}\right)
 =\frac{2\sqrt{r^{p-1}C_{p,\lambda}(r)/2}}{(1+s)(s+r)}.
\]
Substitution of $r=sx$ proves \eqref{lambda-scalar-flux-splitting}.
To check the bounds, write $t=d_s(x)\tau$, with
$d_s(x)=s^{(p-3)/2}(1+sx)/(1+x)$. Since $F_p'>0$,
$0\leq\tau/\tau'\leq1-x$. Therefore
\[
 s^{(p-1)/2}\leq
 \frac{dt}{d\tau}=d_s+\frac{\tau}{\tau'}d_s'
 \leq s^{(p-3)/2},
 \qquad
 s^{(p-1)/2}\leq\frac t\tau\leq s^{(p-3)/2}.
\]
The values at zero are interpreted by limits. The unregularised tangential
and radial eigenvalues satisfy
\[
 A_p(\tau)\asymp_p(1+\tau^2)^{-1/2},\qquad
 B_p(\tau)\asymp_p(1+\tau^2)^{-3/2}.
\]
Indeed the lower bounds were proved above, while
$B_p\leq A_p^3$ follows from $H_p(x)\leq H_p(-x)$;
also $(1+4\tau^2)A_p(\tau)^2\leq5$, since $F_p(x)\leq1$.
Differentiating the residual flux in
\eqref{lambda-scalar-flux-splitting} proves the same comparisons uniformly
for its tangential and radial eigenvalues. These are precisely
\eqref{lambda-residual-area-comparison}. Integrating the flux proves
\eqref{lambda-scalar-energy-splitting}; its Euler equation is
\eqref{lambda-forward-scalar-equation}.
\end{proof}

This identity separates the positive quadratic term from the
linear-growth part. Lemma~\ref{lambda-uniform-gradient-tail} uses it to
verify the tail condition for an adapted maximum-principle test.
Proposition~\ref{lambda-uniform-interior-gradient} then gives a uniform
interior gradient estimate on the fixed hyperbolic metric. Common scalar
domains and height bounds remain inputs when this estimate is applied to
the mappings. At $p=p_{\mathrm c}$ the flux splitting still holds,
but the residual radial eigenvalue has order $t^{-4}$, so the positive
lower comparison in \eqref{lambda-residual-area-comparison} fails.
\subsection{Local gradient estimates and quasiconformal bounds}\label{scalar-gradient-qc}
The relation between the scalar gradient and dilatation is particularly direct: the scalar conductivity is the reciprocal of the operator-norm distortion. Throughout this subsection gradients with a subscript $\eta$ are taken in the fixed metric $\eta|dz|^2$ of curvature $-4$.

\begin{proposition}[Converting a scalar gradient bound into dilatation]\label{scalar-gradient-to-qc}
On a regular natural-coordinate patch, let $1<p\leq2$ and $0<\lambda\leq1$. Put $t_\lambda=|\nabla_\eta v_\lambda|$ and $K_\lambda=(1+r_\lambda)/(1-r_\lambda)$. Then
\begin{equation}\label{scalar-exact-qc-conversion}
 K_\lambda=\frac1{A_{p,\lambda}(t_\lambda)}
 \leq\frac{1+\Theta_p^{-1}(t_\lambda)}{1-\Theta_p^{-1}(t_\lambda)}
 \leq\sqrt{1+4t_\lambda^2}.
\end{equation}
In particular $t_\lambda\leq T$ gives a quasiconformal bound independent of $\lambda$. For the unregularised equation the exact relation is
\begin{equation}\label{scalar-p-dependent-qc}
 4t^2=\bigl[K^2+2(p-1)K+1\bigr]
            \left(\frac{K-1}{K+1}\right)^{p-1}.
\end{equation}
At $p=2$ this reduces to $K=\sqrt{1+4t^2}$.
\end{proposition}
\begin{proof}
The first identity is the definition of $A_{p,\lambda}$. Since $F_{p,\lambda}\geq F_p$, one has $r_\lambda\leq\Theta_p^{-1}(t_\lambda)$. Also $F_p(r)\geq r$ for $1<p\leq2$, so
\[
 t_\lambda^2\geq\frac{r_\lambda}{(1-r_\lambda)^2}
                 =\frac{K_\lambda^2-1}{4}.
\]
Substitution of $r=(K-1)/(K+1)$ into $t^2=F_p(r)/(1-r)^2$ proves \eqref{scalar-p-dependent-qc}.
\end{proof}

The height parameter on a forward patch $\Omega$ is
\[
 M=\operatorname{osc}_{\Omega}\operatorname{Im}(\zeta\circ f),
 \qquad \zeta(w)=\int^w\sqrt{\Phi_p(\xi)}\,d\xi.
\]
It is the vertical width of $\zeta(f(\Omega))$, unchanged by the sign or additive constant of the natural coordinate. At $p=2$, $s=2\operatorname{Im}(\zeta\circ f)$ defines a local minimal graph in $\mathbb H^2\times\mathbb R$ and
$K_f=\sqrt{1+|\nabla_\eta s|^2}$ \cite{WolfMinimal}.

\subsubsection{A gradient test adapted to the positive quadratic term}
Write $\mathfrak a(t)=tA_{p,\lambda}(t)$ and
$\vartheta(t)=t\mathfrak a'(t)/\mathfrak a(t)$. The gradient test uses

\begin{equation}\label{lambda-gradient-tail-condition}
 \mathcal P(t):=\frac{t\vartheta'(t)}{\vartheta(t)}
                    +3-\vartheta(t)
     =4-2\frac{t\mathfrak a'(t)}{\mathfrak a(t)}
                    +\frac{t\mathfrak a''(t)}{\mathfrak a'(t)}.
\end{equation}

\begin{lemma}[A uniform tail condition for the $\lambda$ scalar flux]
\label{lambda-uniform-gradient-tail}
Fix $p_{\mathrm c}<p\leq2$ and $0<\lambda_0<1$. There is
$T=T(p,\lambda_0)\geq1$ such that
$\mathcal P_{p,\lambda}(t)\geq0$ for all
$\lambda_0\leq\lambda\leq1$ and $t\geq T$.
\end{lemma}
\begin{proof}
Use the splitting in Proposition~\ref{lambda-scalar-viscosity} and write
$\mathfrak a(t)=\varepsilon t+b(t)$, where
$b=\widetilde{\mathfrak a}_{p,\lambda}$ and
$\varepsilon=\varepsilon_\lambda\geq0$. Put $s=\sqrt\lambda$ and
\[
 \alpha_s=s^{(p-3)/2}\frac{1+s}{2},\qquad
 \beta_s=s^{(p-3)/2}\frac{1-s}{2},\qquad
 \gamma_s=\frac{2s^{(p-1)/2}}{1+s}.
\]
The exact change of parameter in \eqref{lambda-scalar-flux-splitting}
can also be written as 
\[
 t=\alpha_s\tau+\beta_s\mathfrak a_{p,1}(\tau),
 \qquad b(t)=\gamma_s\mathfrak a_{p,1}(\tau).
\]
The unregularised profile has the differentiated expansion
\[
 \mathfrak a_{p,1}(\tau)
       =\frac12-\frac{\ell_p}{16\tau^2}+O_p(\tau^{-3}).
\]
It follows that, uniformly for $s\in[\sqrt{\lambda_0},1]$,
\begin{equation}\label{lambda-residual-flux-tail}
 \begin{split}
 b(t)&=L_s-\frac{d_s}{t^2}+O_{p,\lambda_0}(t^{-3}),\\
 b'(t)&=\frac{2d_s}{t^3}+O_{p,\lambda_0}(t^{-4}),\qquad
 b''(t)=-\frac{6d_s}{t^4}+O_{p,\lambda_0}(t^{-5}),\\
 L_s&=\frac{\gamma_s}{2},\qquad
 d_s=\frac{\gamma_s\alpha_s^2\ell_p}{16}>0 .
 \end{split}
\end{equation}
The remainders may be differentiated to these orders as the profiles are
analytic in $1-r$ near $r=1$, and the inverse parameter changes are
analytic in $1/\tau$ and $1/t$ near zero, uniformly on this compact
parameter interval. Both $L_s$ and $d_s$ are bounded away from zero.
Consequently
\[
 \mathcal P_b:=4-\frac{2tb'}b+\frac{tb''}{b'}\longrightarrow1,
 \qquad 4b+t^2b''\longrightarrow4L_s>0
\]
uniformly. Choose $T$ so both expressions are nonnegative for $t\geq T$.
Finally, the exact identity
\begin{equation}\label{lambda-tail-viscosity-identity}
 (\varepsilon t+b)(\varepsilon+b')\mathcal P_{p,\lambda}
 =2\varepsilon^2t+\varepsilon(4b+t^2b'')+bb'\mathcal P_b
\end{equation}
proves the assertion, since $b,b'>0$. No division by $\varepsilon$ is
used.
\end{proof}

\begin{proposition}[Uniform interior gradients for the $\lambda$ scalar family]
\label{lambda-uniform-interior-gradient}
Fix $p_{\mathrm c}<p\leq2$ and $0<\lambda_0<1$. Let
$\lambda_0\leq\lambda\leq1$, and let $v_\lambda$ be a classical solution of
\eqref{lambda-forward-scalar-equation} on a neighbourhood of the closed
hyperbolic ball $B_R^\eta(z_0)$. Put
$M=\operatorname{osc}_{B_R^\eta(z_0)}v_\lambda$ and
$t_\lambda=|\nabla_\eta v_\lambda|$. Then
\begin{equation}\label{lambda-uniform-height-gradient-bound}
 \boxed{\quad
 \log\bigl(1+4t_\lambda(z_0)^2\bigr)
       \leq C_{p,\lambda_0}(1+M)^{9/4}(1+R^{-2}).
 \quad}
\end{equation}
The constant is independent of $\lambda$ and of any a priori gradient
or forward Dirichlet bound. If $v_\lambda$ is the Hopf-coordinate
component of $f^\lambda$, then
\begin{equation}\label{lambda-uniform-height-qc-bound}
 \log K_{f^\lambda}(z_0)
       \leq \tfrac12 C_{p,\lambda_0}(1+M)^{9/4}(1+R^{-2}).
\end{equation}
The power $9/4$ and the constants are not best possible.
\end{proposition}
\begin{proof}
All the differential operators in this proof are for the fixed metric
$\eta|dz|^2$. We suppress the subscripts $p,\lambda$, and define
\begin{equation}\label{lambda-adapted-gradient-factor}
 S(t)=1+\int_0^t s\vartheta(s)\,ds,\qquad
 j(t)=\frac{S'(t)}{S(t)}=\frac{t\vartheta(t)}{S(t)} .
\end{equation}
Proposition~\ref{lambda-uniform-scalar-structure} gives
\begin{equation}\label{lambda-adapted-factor-growth}
 1+\frac{\ell_p}{8}\log(1+4t^2)
          \leq S(t)\leq1+\frac{t^2}{2}.
\end{equation}
In particular $S(t)$ tends to infinity uniformly in the parameter.

Subtract a constant from $v=v_\lambda$ so that $0\leq v\leq M$, and set
\[
 f(v)=(M+1-v)^{-1/4},\qquad
 q(v)=\frac{f'(v)}{f(v)}=\frac1{4(M+1-v)},\qquad
 q_0=\frac1{4(M+1)}.
\]
Thus $q\geq q_0$, $q'=4q^2$, and
$(M+1)^{-1/4}\leq f\leq1$. Let $\rho$ denote distance from $z_0$,
and put
\[
 \chi=1-\rho^2/R^2,\qquad g=\chi^2,\qquad
 G=g f(v) S(|\nabla v|).
\]
This function vanishes on the boundary so we may work at an interior maximum
with $t=|\nabla v|>T$, since otherwise $G\leq1+T^2/2$.
Choose an orthonormal frame with $E_1=\nabla v/t$, and put
$v_{ij}=\nabla^2v(E_i,E_j)$,
$q_i=g_i/g$. The scalar equation is
$\vartheta v_{11}+v_{22}=0$. The first derivative test gives
\[
 jv_{11}=-(qt+q_1),\qquad jv_{12}=-q_2.
\]

We now have the following second derivative calculation, with its curvature term
retained. Write
$\mathcal L=\vartheta\nabla^2(\,\cdot\,)(E_1,E_1)
                    +\nabla^2(\,\cdot\,)(E_2,E_2)$.
Differentiation of the scalar equation and the Bochner identity give
\[
 \mathcal L t=-\vartheta'v_{11}^2
       +\frac{2-\vartheta}{t}v_{12}^2
       +\frac1t v_{22}^2
       +t\operatorname{Ric}(E_1,E_1).
\]
This is the calculation underlying \cite[Lemma 7]{RipollTomi}.
Apply $\mathcal L$ to $\log G$, discard the nonnegative term
$(j/t)v_{22}^2$, and substitute the first derivative identities.
At the maximum,
\begin{align}\label{lambda-adapted-maximum-inequality}
 0\geq{}&
 U(qt+q_1)^2+Vq_2^2+\vartheta q't^2
       +\frac{\vartheta t^2}{S}\operatorname{Ric}(E_1,E_1)\notag\\
 &+\frac{\vartheta g_{11}+g_{22}}g
       -\vartheta q_1^2-q_2^2,
\end{align}
where the two cancellations are exact:
\begin{equation}\label{lambda-adapted-cancellations}
 \begin{split}
 U&=-\frac{\vartheta'}j+\vartheta\frac{j'}{j^2}
        =\frac{S}{t^2}-\vartheta,\\
 V&=\frac{j'}{j^2}+\frac{2-\vartheta}{tj}
        =\frac{S}{\vartheta t^2}
              \left(\frac{t\vartheta'}{\vartheta}
                         +3-\vartheta\right)-1\geq-1.
 \end{split}
\end{equation}
The last inequality is Lemma~\ref{lambda-uniform-gradient-tail}.
It is this choice of $S$, rather than $\log t$, which accommodates
the positive quadratic term.

Put $D=\sup|\nabla\chi|$ and $H=\sup|\nabla^2\chi|$ on the ball.
If $qt<2|q_1|$, then $|q_1|\leq2D/\chi$ gives
\[
 G\leq\chi^2S(t)\leq1+\frac{8D^2}{q_0^2}.
\]
Otherwise $(qt+q_1)^2\geq q^2t^2/4$. Using also
$(qt+q_1)^2\leq2q^2t^2+2q_1^2$,
$q'=4q^2$, $V\geq-1$, and $\operatorname{Ric}=-4$, inequality
\eqref{lambda-adapted-maximum-inequality} implies
\[
 0\geq \frac14 S q^2
       +\vartheta t^2\left(2q^2-\frac4S\right)
       -\frac{12D^2}{\chi^2}-\frac{4H}{\chi}.
\]
Indeed $q_i=2\chi_i/\chi$, and
$(\vartheta g_{11}+g_{22})/g\geq-4H/\chi$.
If $S<4/q_0^2$, then $G<4/q_0^2$. In the remaining case the
$\vartheta t^2$ term is nonnegative, so
\[
 G\leq\chi^2S\leq \frac{48D^2+16H}{q_0^2}.
\]
We have therefore bounded the maximum of $G$ by
\[
 \max\left\{1+\frac{T^2}{2},\
              1+\frac{8D^2}{q_0^2},\
              \frac4{q_0^2},\
              \frac{48D^2+16H}{q_0^2}\right\}.
\]
At $z_0$, $g=1$ and $f\geq(M+1)^{-1/4}$.
In curvature $-4$, $D\leq2/R$ and
$H\leq C(1+R^{-2})$, by the Hessian formula for squared distance.
Consequently
\[
 S(t(z_0))\leq
 C_{p,\lambda_0}(1+M)^{9/4}(1+R^{-2}).
\]
Use \eqref{lambda-adapted-factor-growth} to obtain
\eqref{lambda-uniform-height-gradient-bound}, and
Proposition~\ref{scalar-gradient-to-qc} to obtain
\eqref{lambda-uniform-height-qc-bound}.
\end{proof}

\begin{remark}[Consequences of the uniform estimate]\label{lambda-gradient-scope}
On common patches, bounded height oscillations give uniform interior gradients and a common uniformly elliptic range. For the maps,
\[
 |Df^\lambda(z)|_{\rm op}
 =\frac{\sqrt{\eta(z)}\,t_\lambda(z)}
 {\sqrt{|\Phi_{p,\lambda}(f^\lambda(z))|}}.
\]
Thus a positive lower bound for the Hopf coefficient also bounds $Df^\lambda$. The quasiconformal estimate itself does not require this lower bound; Section~\ref{common-hopf-patches} handles the exceptional image values. At $p=p_{\mathrm c}$, $\vartheta_{p,1}(t)\sim1/(4\sqrt2\,t^3)$ and $1+\int_0^\infty t\vartheta_{p,1}(t)\,dt<\infty$, so the unbounded gradient factor used here is unavailable.
\end{remark}

\section{The \texorpdfstring{$\lambda$}{lambda} approximation and the strict upper range}\label{lambda-programme}
We construct a regular quasiconformal minimiser in the strict upper range from a finite-energy quasiconformal barrier. Fixed-parameter existence, polynomial truncation and two parameter selections produce regular approximants with convergent Hopf differentials. The scalar estimates then give a homeomorphic limit and identify its equation. Section~\ref{monotone-qc-comparison} proves absolute monotone minimality.

\subsection{The barrier and the uniform energy bound}\label{lambda-barrier}
Write
\begin{equation}\label{lambda-density}
 R_\lambda(s)=\frac{s^{p/2}}{\lambda-s},\qquad
 e_\lambda(r)=\frac{r^p}{\lambda-r^2},\qquad
 0\leq s=r^2<\lambda<1,
\end{equation}
with $p$ fixed, so that $\I_{p,1}=\I_p$. Set the density to $+\infty$ when $s\geq\lambda$. Finite energy then gives
\[
 \|\mu_{f_\lambda}\|_\infty\leq\sqrt\lambda<1,
 \qquad K_{\mathrm{op}}(f_\lambda)
 \leq\frac{1+\sqrt\lambda}{1-\sqrt\lambda}.
\]
Fix a finite-energy quasiconformal competitor $f_*$, put $k_*=\|\mu_{f_*}\|_\infty$, and choose $k_*^2<\lambda_0<1$. Such a competitor exists precisely for boundary data in $T^p$, by Proposition~\ref{finite-distortion-boundary-class}. Then
\begin{equation}\label{lambda-smooth-barrier}
 \I_{p,\lambda_0}(f_*)\leq
 \frac{1-k_*^2}{\lambda_0-k_*^2}\I_p(f_*)<\infty,
\end{equation}
and every minimiser with $\lambda\geq\lambda_0$ satisfies
\[
 \I_p(f_\lambda)\leq\I_{p,\lambda}(f_\lambda)
 \leq\I_{p,\lambda_0}(f_*)=:M.
\]

For $h_\lambda=f_\lambda^{-1}$, the inverse density is particularly simple:
\begin{equation}\label{barrier-excess}
 J_h R_\lambda(|\mu_h|^2)
 =|h_w|^2r^p\left(1+\frac{1-\lambda}{\lambda-r^2}\right),
 \qquad r=|\mu_h|.
\end{equation}
The inverse Dirichlet bound is therefore uniform in $\lambda$, and Proposition~\ref{compact-containment} supplies compactness and excludes boundary collapse.

There is also a useful identification at the level of infima. Let $m_{\rm QC}$ be the infimum of $\I_p$ over quasiconformal competitors with these boundary values and finite energy. Comparison with each fixed competitor, followed by dominated convergence, gives
\[
 m(\lambda)\downarrow m_{\rm QC},\qquad
 \I_p(f_\lambda)\longrightarrow m_{\rm QC}.
\]
Here $m(\lambda)=\I_{p,\lambda}(f_\lambda)$. Consequently the integral of the extra term in \eqref{barrier-excess} tends to zero along the whole family. This shows that the regularised minimisers form a minimising family in the quasiconformal class. In the strict upper range, Theorem~\ref{absolute-monotone-comparison} identifies this infimum with the finite-distortion homeomorphic and absolute monotone minima.
\subsection{Fixed-parameter existence in the enlarged inverse class}\label{lambda-existence}
The $\lambda$-regularised barrier is a convex increasing modification of the original energy density. This gives us direct use of the enlarged-class strategy of \cite[Section 5]{MYLp} before requiring regularity of a forward pseudo-inverse.

\begin{proposition}[Quasiconformal attainment for the $\lambda$-regularised barrier]\label{rational-lambda-existence}
Let $p_{\mathrm c}\leq p\leq2$ and $0<\lambda<1$. Suppose the quasisymmetric trace $\beta$ has a quasiconformal extension $h_*$ with finite $\lambda$-inverse energy. Define
\[
 W_{p,\lambda}(A)=
 \begin{cases}
 (\det A)e_\lambda(r),&\det A>0,\quad r=|b|/|a|<\sqrt\lambda,\\
 0,&A=0,\\
 +\infty,&\text{otherwise},
 \end{cases}
 \qquad A=(a,b).
\]
Then $W_{p,\lambda}$ is polyconvex, and
\[
 \E_{p,\lambda}(h)=\int_\D\eta(h)W_{p,\lambda}(Dh)\area
\]
attains its minimum on $\mathcal M_\beta$. Every finite-energy member of this class, in particular every minimiser, is a homeomorphism and satisfies
\[
 K_h=K_{h^{-1}}\leq\frac{1+\sqrt\lambda}{1-\sqrt\lambda}
 \quad\hbox{almost everywhere},
\]
where equality of the two dilatations is understood at corresponding points. Its inverse is a minimiser of the corresponding forward problem with the prescribed trace.
\end{proposition}
\begin{proof}
Write $e(r)=r^p/(1-r^2)$ and $u_\lambda=e(\sqrt\lambda)$. On $0\leq u<u_\lambda$, define
\[
 \chi_\lambda(u)=e_\lambda(e^{-1}(u)).
\]
Extend it by $u/\lambda$ for $u<0$, and by $+\infty$ for $u\geq u_\lambda$. Put $C_\lambda(r)=p\lambda+(2-p)r^2$ and $C_1(r)=p+(2-p)r^2$. Then
\[
 \chi_\lambda'(e(r))
 =\frac{e_\lambda'(r)}{e'(r)}
 =\frac{C_\lambda(r)}{C_1(r)}
       \left(\frac{1-r^2}{\lambda-r^2}\right)^2.
\]
This derivative has limit $1/\lambda$ at zero, and
\[
 \frac{d}{dr}\log\frac{e_\lambda'}{e'}
 =2(2-p)r\left(\frac1{C_\lambda}-\frac1{C_1}\right)
   +\frac{4r(1-\lambda)}{(\lambda-r^2)(1-r^2)}>0
 \quad(0<r<\sqrt\lambda).
\]
Thus $\chi_\lambda$ is convex and increasing. With $G_p$ from Lemma~\ref{jacobian-convexification} and $c_p=(p-1)/2$, the function of independent minors
\[
 \mathcal P_{p,\lambda}(A,d)
 =d\,\chi_\lambda\!\left(\frac{G_p(A)}d-c_p\right),
 \qquad d>0,
\]
is convex: This is an increasing convex perspective composed with the convex $G_p$. At $d=\det A>0$, one has $G_p(A)/d-c_p=W_p(A)/d=e(r)$, giving $W_{p,\lambda}$. The closed extension at $d=0$ is infinite for $A\ne0$ and zero at $A=0$. This proves polyconvexity and excludes every nonzero rank-one matrix at finite energy.

For completeness, weighted lower semicontinuity follows by the same local argument as Lemma~\ref{weighted-lsc}. The supporting-line representation
\[
 \chi_\lambda(u)=\sup_{t\geq1/\lambda}\{tu-\chi_\lambda^*(t)\}
\]
gives, for $J\geq0$,
\[
 W_{p,\lambda}(A)=\sup_{t\geq1/\lambda}
 \left\{tG_p(A)-[c_pt+\chi_\lambda^*(t)]J\right\}.
\]
First restrict $t$ to a bounded interval and use bounded continuous choices $t=t(w)$. The convex-gradient term is lower semicontinuous, and the weighted Jacobian term converges by \eqref{weighted-jacobian-convergence}. Measurable maximisers can be approximated by continuous choices for the finite measure with density $G_p(Dh)+J_h$ on a compact set. Letting the interval increase recovers the density. The baseline $t=1/\lambda$ gives $W_p/\lambda\geq0$ on physical gradients, so exhaustion gives global lower semicontinuity. Since $W_{p,\lambda}\geq W_p$, Proposition~\ref{monotone-compactness} gives us compactness and our barrier implies a finite minimum.

A finite-energy map satisfies $|h_{\bar w}|\leq\sqrt\lambda\,|h_w|$ almost everywhere and lies in $W^{1,2}$. It is nonconstant, hence quasiregular and therefore open and discrete. Its fibres are also connected, so every interior fibre is a singleton. The definition of $\mathcal M_\beta$ excludes interior preimages of the circle, and the boundary trace is injective. Thus $h$ is a homeomorphism of the closed disk and is quasiconformal in the disk. The usual inverse and area formulas for quasiconformal maps give the forward energy identity and minimality.
\end{proof}

This proposition settles fixed-parameter attainment and topology in the upper range. Near $r=\sqrt\lambda$, however, the energy has a simple pole and its first-variation stress a double pole. Quasiconformality alone therefore does not justify arbitrary two-sided variations through the constraint. The truncation and selection of Section~\ref{rational-truncation-selection} establish the equation for selected minimisers; Lemma~\ref{rational-finite-stress-regularity}, proved first, then obtains local pole separation as a consequence of the weak equation.
\subsection{Regularity at a fixed parameter}\label{lambda-regularity-transfer}
The corollary following Theorem~1.1 of \cite{MYLp} proves that a locally quasiconformal minimiser of $\int\K^m$, $m>1$, is a smooth diffeomorphism. The proof uses a strict contraction for the forward potential: the potential satisfies $\mathcal F_{\bz}=\mathcal A_m(|\mathcal F_z|)$ with $|\mathcal A_m'|\leq k_m<1$. Quasiconformality gives us the required integrability, but the contraction is a property of the density and must be verified for \eqref{lambda-density}. For a locally constant weight the forward inner-variation potential has the schematic relation
\[
 |\mathcal F_z|=T_\lambda(r),\qquad
 \mathcal F_{\bz}=e_\lambda(r),\qquad
 T_\lambda(r)=\frac{1-r^2}{2}e_\lambda'(r),
\]
and for a varying weight a Cauchy correction must be retained. At $\lambda=1$ one has $e_1/T_1=2r/(p+(2-p)r^2)<1$, which controls the size of the coefficient of a \emph{given} first-order equation; differentiating the nonlinear constitutive relation requires instead $e_1'/T_1'<1$, and
\begin{equation}\label{aux-derivative}
 \frac{T_1'}{e_1'}-1=
 \frac{(1-r)^2H_p(r)}{2r\,[p+(2-p)r^2]},
\end{equation}
which is \eqref{forward-stress-contraction}. Thus the threshold reappears, and even above it $e_1'/T_1'\to1$ as $r\uparrow1$: there is no contraction uniform over all distortions. At a fixed parameter $\lambda<1$ the pole restores a global contraction.

\begin{lemma}[Strict potential contraction for each upper-range barrier]\label{rational-fixed-potential-contraction}
Let $p_{\mathrm c}\leq p\leq2$ and $0<\lambda<1$. Put $T_\lambda=(1-r^2)e_\lambda'/2$. The increasing inverse relation
\[
 \mathcal B_\lambda=e_\lambda\circ T_\lambda^{-1}
\]
has a global Lipschitz constant $k_{p,\lambda}<1$. More precisely, writing
\[
 \Delta_p(r)=\frac{(1-r)^2H_p(r)}{2rC_1(r)},
\]
one has
\begin{equation}\label{rational-potential-gap}
 \frac{T_\lambda'}{e_\lambda'}-1
 =\Delta_p(r)
  +\frac{p(2-p)(1-\lambda)r(1-r^2)}
         {C_\lambda(r)C_1(r)}
  +\frac{2r(1-\lambda)}{\lambda-r^2}
 =:D_{p,\lambda}(r)>0.
\end{equation}
Thus $k_{p,\lambda}=(1+d_{p,\lambda})^{-1}$, where
$d_{p,\lambda}=\inf_{0<r<\sqrt\lambda}D_{p,\lambda}(r)>0$.
For every fixed $p$ in this range, $k_{p,\lambda}\to1$ as $\lambda\uparrow1$.
\end{lemma}
\begin{proof}
Logarithmic differentiation of $e_\lambda'$ gives \eqref{rational-potential-gap}. All three terms are nonnegative and the last is positive. The first tends to infinity as $r\downarrow0$, and the last tends to infinity as $r\uparrow\sqrt\lambda$. Continuity consequently gives a positive minimum. Also $T_\lambda$ increases from zero to infinity and
\[
 \mathcal B_\lambda'(T_\lambda(r))
 =\frac{e_\lambda'}{T_\lambda'}
 =\frac1{1+D_{p,\lambda}(r)}.
\]
This proves the contraction, including its continuous interpretation at zero. For fixed $r<1$, $D_{p,\lambda}(r)\to\Delta_p(r)$, whereas $\Delta_p(r)\to0$ as $r\uparrow1$. Taking these limits in this order proves $d_{p,\lambda}\to0$.
\end{proof}

The integral identity below is a hypothesis of the lemma.
For the selected minimisers of
Theorem~\ref{selected-rational-regularity}, it is supplied by
\eqref{selected-forward-stationarity}, obtained by passing
the target variations of the truncated inverse minimisers
to the limit before establishing separation from the pole.
Thus the lemma is applied after weak stationarity has been proved;
local pole separation is a conclusion, not a prior admissibility
assumption.

\begin{lemma}[Regularity from finite $\lambda$ stress]
\label{rational-finite-stress-regularity}
Let $p_{\mathrm c}<p<2$, $0<\lambda<1$, and let $\eta$ be smooth
and positive. Suppose $f$ is a quasiconformal homeomorphism with
$r=|\mu_f|<\sqrt\lambda$ almost everywhere. Assume that
\[
 U=T_\lambda(r)\frac{\overline{\mu_f}}r\in L^1_{\loc},
 \qquad U=0\quad\hbox{when }r=0,
\]
and that the weak identity
\[
 \int\eta U\varphi_{\bz}\area
       =\int e_\lambda(r)(\eta\varphi)_z\area
 \qquad\bigl(\varphi\in C_c^\infty(\D,\IC)\bigr)
\]
holds. Then $U$ is locally bounded, $r$ is locally bounded strictly below
$\sqrt\lambda$, and $f$ is a $C^{2,\alpha}_{\loc}$ diffeomorphism
for $0<\alpha<\min\{1,(2-p)/(p-1)\}$. It is smooth where
$\mu_f\ne0$.
\end{lemma}
\begin{proof}
The first gain of integrability comes from the pole itself:
\begin{equation}\label{rational-square-energy-stress}
 \frac{e_\lambda(r)^2}{T_\lambda(r)}
 =\frac{2r^{p+1}}{(1-r^2)[p\lambda+(2-p)r^2]}
 \leq\frac{2\lambda^{(p-1)/2}}{p(1-\lambda)}.
\end{equation}
The zero values are interpreted continuously. Since $U\in L^1_{\loc}$ and $|U|=T_\lambda(r)$,
estimate \eqref{rational-square-energy-stress} implies
$e_\lambda(r)\in L^2_{\loc}$.
It is useful to retain the weight in the differentiated stress:
\begin{equation}\label{rational-weighted-stress-distribution}
 (\eta U)_{\bz}
   =(\eta e_\lambda(r))_z-\eta_z e_\lambda(r).
\end{equation}
Choose a cutoff equal to one on a smaller disk. A Beurling transform
of the cut-off function $\eta e_\lambda(r)$, minus a Cauchy
transform of the cut-off function $\eta_z e_\lambda(r)$, solves
the same $\bar\partial$ equation there and belongs locally to
$L^2$. Its difference from $\eta U$ is a holomorphic distribution,
and hence a smooth function. Therefore $U\in L^2_{\loc}$.
No strong $L^1$ estimate for a singular integral is used.

Put $L=\log\eta$ and take a local correction $G$ satisfying
$G_{\bz}=L_{\bz}U$, with $G\in W^{1,2}_{\loc}$. The compatible
derivatives
\[
 \mathcal F_z=U+G,\qquad
 \mathcal F_{\bz}=e_\lambda(r)
      =\mathcal B_\lambda(|\mathcal F_z-G|)
\]
define a $W^{1,2}_{\loc}$ potential. By
Lemma~\ref{rational-fixed-potential-contraction}, the radial
constitutive map has a global Lipschitz constant $k<1$.
Localised difference quotients, as in
Proposition~\ref{forward-regularity-proposition}, give
$\mathcal F\in W^{2,2}_{\loc}$ and $U\in W^{1,2}_{\loc}$.

The Sobolev chain rule now gives
\eqref{forward-stress-beltrami}, with $\mathcal B$ replaced by
$\mathcal B_\lambda$, and coefficient norm at most $k/(2-k)<1$.
Since $U\in W^{1,2}_{\loc}$ in two dimensions, its lower-order
term belongs to every finite local $L^q$. Local higher integrability
therefore gives $U\in W^{1,s}_{\loc}$ for some $s>2$.
In particular $U$ is H\"older continuous and locally bounded.
Because $T_\lambda(r)\to\infty$ at the pole, this already proves
the strict local gap from $\sqrt\lambda$.  

Schauder bootstrapping and reconstruction proceed exactly as in
\eqref{forward-stress-reconstruction}. To include the conformal
set, put $\gamma=1/(p-1)>1$. Near zero,
\[
 T_\lambda^{-1}(t)
   =(2\lambda/p)^\gamma t^\gamma B(t^{2\gamma}),
 \qquad B(0)=1,
\]
where $B$ is real analytic near zero. The coefficient functions in
the stress equation are locally Lipschitz functions of $U$, including
at zero, since $\mathcal B_\lambda'(t)=O(t^\gamma)$.
Thus $U\in C^{1,\alpha}_{\loc}$ for every $\alpha<1$, and both
\[
 \mu_f=\frac{T_\lambda^{-1}(|U|)}{|U|}\overline U,
 \qquad r=T_\lambda^{-1}(|U|)
\]
belong to $C^{1,\alpha}_{\loc}$ for the exponents in the statement.
The linear Beltrami equation gives $f\in C^{2,\alpha}_{\loc}$;
the nonvanishing-Jacobian theorem for a homeomorphic solution with
H\"older coefficient, as used in
Proposition~\ref{forward-regularity-proposition}, gives $J_f>0$.
Smooth bootstrap applies where $U\ne0$.
\end{proof}

\subsection{Truncation and selection of regular $\lambda$ minimisers}
\label{rational-truncation-selection}
The pole in the $\lambda$-density prevents a direct appeal to the
inner-variation equation for a weak minimiser. We remove this obstacle
by an increasing approximation with polynomial growth in distortion.
The selection is made in the parameter before taking the truncation
limit. This process produces at least one regular minimiser for almost every
admissible parameter; this is enough for us to obtain a sequence tending to one.

\begin{theorem}[Regular $\lambda$-minimisers at almost every parameter]
\label{selected-rational-regularity}
Let $p_{\mathrm c}<p<2$. Fix a quasisymmetric inverse trace $\beta$
and a quasiconformal competitor $h_*$ with
$\E_{p,\lambda_0}(h_*)<\infty$ for some $0<\lambda_0<1$.
There is a full-measure set $\Lambda\subset(\lambda_0,1)$ such that,
for every $\lambda\in\Lambda$, the inverse energy has a
minimiser $h_\lambda$ on $\mathcal M_\beta$ with the following
properties. Its Hopf differential is holomorphic, its inverse
$f_\lambda$ satisfies the distributional forward variation equation,
and both maps are $C^{2,\alpha}_{\loc}$ diffeomorphisms for
\[
 0<\alpha<\min\left\{1,\frac{2-p}{p-1}\right\}.
\]
They are smooth away from their conformal sets. For each compact
$V\Subset\D$ there is a strict local gap
$\sup_V|\mu_{f_\lambda}|<\sqrt\lambda$. 
\end{theorem}
\begin{proof}
\emph{The finite truncations.}
Put $\delta=1-\lambda$ and, for $N\geq1$, define
\begin{equation}\label{rational-geometric-truncation}
 e_{\lambda,N}(r)
 =\sum_{n=0}^N\delta^n\frac{r^p}{(1-r^2)^{n+1}},
 \qquad 0\leq r<1.
\end{equation}
As $N\to\infty$, these densities increase to $e_\lambda(r)$
for $r^2<\lambda$, and to $+\infty$ for $r^2\geq\lambda$.
For $q_n(r)=r^p/(1-r^2)^{n+1}$ and $e=q_0$, the quotient
\begin{equation}\label{truncation-convexity-quotient}
 \frac{q_n'(r)}{e'(r)}
 =(1-r^2)^{-n}
   \frac{p+(2n+2-p)r^2}{p+(2-p)r^2}
 =(1-r^2)^{-n}\left(1+\frac{2nr^2}{p+(2-p)r^2}\right)
\end{equation}
is nondecreasing. Thus $q_n\circ e^{-1}$ is increasing and convex
on $[0,\infty)$; extend it linearly, with slope one, to negative
arguments. The perspective construction in
Proposition~\ref{rational-lambda-existence}, using $G_p-c_pJ$,
proves polyconvexity of each inverse summand $Jq_n(r)$.
For $n\geq1$ its closed extension is infinite at nonzero rank-one
matrices and zero at the zero matrix.

Let $\E_{\lambda,N}$ denote the corresponding weighted inverse
functional and $m_N(\lambda)$ its minimum on $\mathcal M_\beta$.
The direct method used above applies and $\E_{\lambda,N}\geq\E_p$
gives compactness, and the competitor gives
\begin{equation}\label{truncated-minimum-bound}
 0\leq m_N(\lambda)\leq\E_{p,\lambda_0}(h_*)=:M
 \qquad(\lambda\geq\lambda_0).
\end{equation}
Write $h_N$ for any such minimiser at the parameter under consideration.
These intermediate maps will not usually be homeomorphic.

At a fixed $N$, the density has polynomial growth in $\K$.
Small smooth changes of domain preserve finite energy. On their
compact support the differentiated integrand is bounded by
$C_{N,\lambda}\eta(h_N)[J_{h_N}e_{\lambda,N}(r_N)+J_{h_N}]$,
where $r_N=|(h_N)_{\bw}|/|(h_N)_w|$ and $r_N=0$ when $Dh_N=0$.
The bound near $r_N=0$ follows from $p>1$; near $r_N=1$ it follows
from the finite polynomial growth. Compact containment controls the
weight. Both signs of the variation are therefore admissible and
differentiation is justified. The Hopf expression
\begin{equation}\label{truncated-hopf-expression}
 \Phi_N=\eta(h_N)J_{h_N}T_{\lambda,N}(r_N)
                  \frac{\overline{\mu_{h_N}}}{r_N},
 \qquad T_{\lambda,N}=\frac{1-r^2}{2}e_{\lambda,N}'
\end{equation}
is holomorphic. Products at $r_N=0$ and $Dh_N=0$ are zero.
The same growth estimate justifies compactly supported smooth
changes of target.

\emph{Selection with bounded stress.}
Set $D_{\lambda,N}=-\partial_\lambda e_{\lambda,N}\geq0$.
Differentiating the finite geometric sum gives the exact identity
\begin{equation}\label{truncated-stress-identity}
 T_{\lambda,N}(r)
 =\frac{p+(2-p)r^2}{2r}\,e_{\lambda,N}(r)
       +r(1-\lambda)D_{\lambda,N}(r).
\end{equation}
On every compact parameter interval $[a,b]\subset(\lambda_0,1)$,
\begin{equation}\label{truncated-stress-majorant}
 T_{\lambda,N}(r)
 \leq C_{p,a}\bigl(1+e_{\lambda,N}(r)\bigr)
           +(1-\lambda)D_{\lambda,N}(r),
\end{equation}
uniformly in $N$. Indeed, for $r\geq\sqrt a/2$ the first
coefficient in \eqref{truncated-stress-identity} is bounded. For
$r\leq\sqrt a/2$, the ratio $(1-\lambda)/(1-r^2)$ is bounded
strictly below one, so the first term is bounded by $Cr^{p-1}$.

The functions $m_N$ are decreasing in $\lambda$. At a differentiability
point, comparison with the same minimiser at $\lambda+t$ gives
\begin{equation}\label{truncated-envelope-estimate}
 \int_\D\eta(h_N)J_{h_N}D_{\lambda,N}(r_N)\area
 \leq -m_N'(\lambda).
\end{equation}
The differentiation under this particular integral is legitimate as
for fixed $N$ and $\lambda<1$, the parameter derivative is bounded
by $N e_{\lambda,N}/(1-\lambda)$.
By \eqref{truncated-minimum-bound},
$\int_a^b(-m_N')\,d\lambda\leq M$. Fatou's lemma implies
\[
 \int_a^b\liminf_{N\to\infty}(-m_N'(\lambda))\,d\lambda\leq M.
\]
After discarding the null set where some $m_N$ is not differentiable,
and exhausting the parameter interval, we obtain a full-measure
set $\Lambda$ such that each $\lambda\in\Lambda$ admits a subsequence
$N_j\to\infty$ with $\sup_j[-m_{N_j}'(\lambda)]<\infty$.

\medskip
Now fix such a parameter and subsequence.

For $V\Subset\D$, the sets $h_N(V)$ lie in a common compact subset
of $\D$, by Proposition~\ref{monotone-compactness}. Since
$\int_\D J_{h_N}\area=\pi$, equations
\eqref{truncated-stress-majorant}--\eqref{truncated-envelope-estimate}
give
\begin{equation}\label{selected-truncated-hopf-bound}
 \sup_j\int_V|\Phi_{N_j}|\area<\infty.
\end{equation}
Holomorphic compactness yields a locally uniformly convergent
subsequence, with holomorphic limit $\Phi_\lambda$.

\emph{Energy convergence and strong convergence.}
Let $m(\lambda)$ be the $\lambda$-regularised minimum from
Proposition~\ref{rational-lambda-existence}. The numbers
$m_N(\lambda)$ increase and are at most $m(\lambda)$.
Compactness gives $h_N\to h$ uniformly on the closed disk and
weakly in $W^{1,2}$. For each fixed $L$, lower semicontinuity gives
\[
 \E_{\lambda,L}(h)\leq\liminf_{N\to\infty}
                         \E_{\lambda,L}(h_N)
                  \leq\lim_{N\to\infty}m_N(\lambda).
\]
Letting $L\to\infty$ and comparing with the $\lambda$ minimum shows
\begin{equation}\label{truncated-energy-identification}
 m_N(\lambda)\uparrow m(\lambda)=\E_{p,\lambda}(h).
\end{equation}
In particular $h$ is quasiconformal and homeomorphic.

Write $\E_{\lambda,N}=\E_p+\mathcal R_N$, where $\mathcal R_N$
is the sum of the terms with $n\geq1$. These are nonnegative
polyconvex functionals increasing to
$\mathcal R=\E_{p,\lambda}-\E_p$ on the finite-energy class.
Applying lower semicontinuity separately to $\E_p$ and to each
fixed partial sum of $\mathcal R_N$, and using
\eqref{truncated-energy-identification}, gives
\begin{equation}\label{truncated-base-energy-convergence}
 \E_p(h_N)\longrightarrow\E_p(h).
\end{equation}
The weighted Jacobian convergence and the convexity-defect argument
in Corollary~\ref{strong-monotone-minimising} now imply
\begin{equation}\label{truncated-strong-convergence}
 Dh_N\longrightarrow Dh\quad\hbox{in }L^2_{\loc}.
\end{equation}
This is the Radon--Riesz mechanism of \cite{MYRadonRiesz}; the
explicit convexification verifies it here without imposing a
different power-growth hypothesis on the density.

Pass to a subsequence with almost-everywhere derivative convergence.
Since $h$ is quasiconformal, $J_h>0$ almost everywhere, and
$r_h^2<\lambda$ almost everywhere by its finite energy.
The profiles $e_{\lambda,N}$ and their first derivatives converge
uniformly on every closed subinterval below the pole. Consequently
\eqref{truncated-hopf-expression} converges almost everywhere to
the $\lambda$ Hopf expression of $h$, identifying it with
$\Phi_\lambda$. Moreover, the nonnegative densities
\begin{equation}\label{truncated-density-Lone}
 \eta(h_N)J_{h_N}e_{\lambda,N}(r_N)
 \longrightarrow \eta(h)J_he_\lambda(r_h)
 \quad\hbox{in }L^1(\D).
\end{equation}
Indeed they converge almost everywhere, and their integrals
converge by \eqref{truncated-energy-identification} if we once again apply
Scheff\'e's lemma.

\emph{Passing the forward variation without differentiating a weak inverse.}
For $\varphi\in C_c^\infty(\D,\IC)$, the target variation
$h_N\mapsto(\id+t\varphi)\circ h_N$ gives
\begin{equation}\label{truncated-target-variation}
 \int_\D\Phi_N\frac{\overline{(h_N)_w}}{(h_N)_w}
                 \varphi_{\bz}(h_N)\area
 +\int_\D J_{h_N}e_{\lambda,N}(r_N)
                 (\eta\varphi)_z(h_N)\area=0.
\end{equation}
Set the phase to zero when $(h_N)_w=0$. To check the sign, if
$a=(h_N)_w$ and $b=(h_N)_{\bw}$, then target composition gives
$\dot a=\varphi_z a+\varphi_{\bz}\bar b$ and
$\dot b=\varphi_z b+\varphi_{\bz}\bar a$.
The distortion contribution is
$2T_{\lambda,N}\operatorname{Re}(\varphi_{\bz}\bar b/a)/r_N$.

The supports in \eqref{truncated-target-variation} lie in one
compact source set for all sufficiently large $N$: this follows
from uniform convergence to the proper homeomorphism $h$.
On this compact set, the first integral passes to the limit
by dominated convergence and $\Phi_N$ and $h_N$ converge locally
uniformly, while
\[
 \frac{\overline{(h_N)_w}}{(h_N)_w}
 \longrightarrow \frac{\overline{h_w}}{h_w}
 \quad\hbox{almost everywhere},
 \qquad
 \left|\frac{\overline{(h_N)_w}}{(h_N)_w}\right|\leq1,
\]
with the quotient defined as zero when $(h_N)_w=0$.
The second passes by \eqref{truncated-density-Lone}.
Write $f=h^{-1}$, $r=|\mu_f|$, and
\[
 U=T_\lambda(r)\frac{\overline{\mu_f}}r,\qquad U=0\text{ at }r=0.
\]
The quasiconformal change of variables and inverse derivative
formula give
\[
 (\eta U)(h)J_h=-\Phi_\lambda\frac{\overline{h_w}}{h_w},
 \qquad
 \int_A\eta|U|\area=\int_{h^{-1}(A)}|\Phi_\lambda|\area
 \quad(A\Subset\D).
\]
Thus $U\in L^1_{\loc}$ and the limiting identity is
\begin{equation}\label{selected-forward-stationarity}
 \int_\D\eta U\varphi_{\bz}\area
 =\int_\D e_\lambda(r)(\eta\varphi)_z\area.
\end{equation}
Lemma~\ref{rational-finite-stress-regularity} proves the
claimed regularity and local pole separation, including at conformal
points. Applying the inverse-function theorem gives the assertions
for $h$ as well.
\end{proof}
\subsection{Selecting the parameters as $\lambda\uparrow1$}\label{parameter-selection-section}
The moving pole suggests measuring its stress contribution directly. The following observation requires minimisers but does \emph{not} require a differentiable family of minimising maps or convexity of the energy.

\begin{proposition}[Parameter selection]\label{parameter-selection}
Assume that for each $\lambda\in(\lambda_0,1)$ the functional
\[
 \I_{p,\lambda}(f)=\int_\D\eta(z)
 \frac{|\mu_f|^p}{\lambda-|\mu_f|^2}\area
\]
has a homeomorphic minimiser $f_\lambda$ in the prescribed boundary class, and that a common finite-energy competitor as above exists. Put $m(\lambda)=\I_{p,\lambda}(f_\lambda)$. There is a sequence $\lambda_j\uparrow1$ such that
\begin{equation}\label{stress-small}
 (1-\lambda_j)D_{\lambda_j}\longrightarrow0,\qquad
 D_\lambda=\int_\D\eta(z)
 \frac{|\mu_{f_\lambda}|^p}{(\lambda-|\mu_{f_\lambda}|^2)^2}\area.
\end{equation}
\end{proposition}
\begin{proof}
The admissible classes increase and the integrands decrease with $\lambda$, so $m$ is nonnegative, bounded and decreasing. For $t>0$, comparison with the fixed minimiser $f_\lambda$ gives
\[
 \frac{m(\lambda)-m(\lambda+t)}t
 \geq\int_\D\eta
 \frac{|\mu_{f_\lambda}|^p}
 {(\lambda-|\mu_{f_\lambda}|^2)(\lambda+t-|\mu_{f_\lambda}|^2)}\area.
\]
At almost every differentiability point, monotone convergence as $t\downarrow0$ yields $D_\lambda\leq-m'(\lambda)$. Since $-m'$ is integrable on $(\lambda_0,1)$, a sequence of these points satisfies $(1-\lambda_j)(-m'(\lambda_j))\to0$; otherwise its integral would dominate a positive multiple of $\int d\lambda/(1-\lambda)$. This proves \eqref{stress-small}.
\end{proof}

Here is its precise connection to the Hopf differential. Use the normalisation of Theorem~\ref{hopf}
\[
 \Phi_{p,\lambda}[h]=\eta(h)(1-r^2)^2R_\lambda'(r^2)
                  h_w\overline{h_{\bw}}.
\]
Set $B=\eta(h)|h_w|^2$, $d=\lambda-r^2$, $\delta=1-\lambda$, and $c_r=p+(2-p)r^2$. Then
\begin{align*}
 |\Phi_{p,\lambda}[h]|&=\tfrac12 Br^{p-1}
           (1+\delta/d)^2(c_r-p\delta),\\
 |\Phi_p[h]|&=\tfrac12 Br^{p-1}c_r,\qquad
 \mathfrak d_\lambda(h)=Br^p\frac{1-r^2}{(\lambda-r^2)^2}.
\end{align*}
The two differentials evaluated on the \emph{same map} have the same phase. Since $c_r\leq2$ and $\Psi_\lambda'\geq\Psi_1'$, we obtain, for $r>0$,
\begin{equation}\label{stress-difference}
 |\Phi_{p,\lambda}[h]-\Phi_p[h]|
 \leq\frac2r\,\delta\,\mathfrak d_\lambda(h).
\end{equation}
The change of variables gives $\int\mathfrak d_\lambda(h_\lambda)\area=D_\lambda$. Thus \eqref{stress-small} eliminates the barrier contribution to the Hopf differential on $r\geq r_0>0$, in $L^1$.

On $r<r_0<1$, the coefficient difference is bounded by $C_{p,r_0}\delta Br^{p-1}$ for $\lambda$ sufficiently close to one. The common quasiconformal competitor has fixed quasisymmetric boundary values, and Proposition~\ref{compact-containment} ensures that $h_{\lambda_j}$ send each fixed compact set into a common compact subset of $\D$. Thus $\eta(h_{\lambda_j})$ is uniformly bounded there. The inverse $W^{1,2}$ bound controls $\int B$ on those sets, and gives
\begin{equation}\label{hopf-defect-zero}
 \Phi_{p,\lambda_j}[h_{\lambda_j}]-\Phi_p[h_{\lambda_j}]
       \longrightarrow0\quad\hbox{in }L^1_{\loc}.
\end{equation}
Also $|\Phi_p[h]|\leq B$. If the regularised Hopf differentials are holomorphic, they consequently have a locally uniformly convergent subsequence to a holomorphic differential. This provides a concrete possible form of the desired equation convergence without a uniform quasiconformal constant.

For a general weak barrier family the remaining gap is identification: weak $W^{1,2}$ convergence of $h_{\lambda_j}$ does not by itself give $\Phi_p[h_{\lambda_j}]\to \Phi_p[h]$. The parameter-selection lemma controls the extra barrier stress, but one still needs enough derivative convergence to identify its limit. For the regular family of Theorem~\ref{selected-rational-regularity} this is supplied by the scalar estimates of Section~\ref{scalar-equations}, as follows.
\subsection{A forward Dirichlet bound suffices for the homeomorphic limit}
A local forward $W^{1,2}$ bound gives the modulus of continuity needed for the topological limit; compare \cite[Theorem 6.2]{MYLp}.

\begin{proposition}[Homeomorphic limit from local forward Dirichlet bounds]\label{forward-Dirichlet-homeomorphic-limit}
Let $f_j,h_j$ be inverse quasiconformal homeomorphisms of the closed disk, with $h_j|_{\Sone}=\beta$. Suppose $h_j\to h$ uniformly on the closed disk, where $h$ is a continuous monotone surjection with trace $\beta$ and $h^{-1}(\Sone)=\Sone$. The inverse compactness and our NoHairs argument shows these hypotheses hold for the bounded-energy $\lambda$ family. If
\begin{equation}\label{forward-local-Dirichlet-target}
 \sup_j\int_U\|Df_j\|^2\area<\infty
 \qquad\hbox{for every }U\Subset\D,
\end{equation}
then $h$ is a homeomorphism of the closed disk,
$f_j\to f=h^{-1}$ uniformly there, and $f\in W^{1,2}_{\loc}(\D)$.
No uniform bound for the quasiconformal constants is required here.

A sufficient inverse-side condition for \eqref{forward-local-Dirichlet-target} is
\begin{equation}\label{inverse-distortion-Lone-target}
 \sup_j\int_V\K_{h_j}\area<\infty
 \qquad\hbox{for every }V\Subset\D.
\end{equation}
\end{proposition}
\begin{proof}
For $B_r(z_0)\subset B_R(z_0)\Subset\D$, the Courant--Lebesgue circle argument for a Sobolev homeomorphism gives
\begin{equation}\label{forward-logarithmic-modulus}
 \bigl(\operatorname{diam}f_j(B_r(z_0))\bigr)^2
 \leq\frac{C}{\log(R/r)}
       \int_{B_R(z_0)}\|Df_j\|^2\area ,
 \qquad 0<r<R,
\end{equation}
with an absolute constant $C$. Indeed, for almost every $s\in(r,R)$ the image of $\partial B_s$ is a rectifiable Jordan curve enclosing $f_j(B_r)$. Its length controls the latter diameter; Cauchy--Schwarz bounds its squared length by $2\pi s$ times the tangential Dirichlet integral on $\partial B_s$. Integrating $ds/s$ proves \eqref{forward-logarithmic-modulus}, with the harmless factor for our normalised Hilbert--Schmidt norm.

Condition~\eqref{forward-local-Dirichlet-target} therefore gives local equicontinuity. Since the maps are disk-valued, a subsequence converges locally uniformly and weakly in $W^{1,2}_{\loc}$ to a continuous $f:\D\to\overline\D$. Uniform convergence of $h_j$ and continuity of $h$ give
\[
 h(f(z))=\lim_jh_j(f_j(z))=z\qquad(z\in\D).
\]
Because $h(\Sone)=\Sone$, this identity also implies $f(\D)\subset\D$. Lemma~\ref{hair-homeomorphism} now excludes interior hairs, and the assumed NoHairs conclusion excludes the remaining boundary hairs. Thus $h$ is a homeomorphism and $f=h^{-1}$ throughout the disk.

The inverse extends continuously to the closed disk. If $\omega_f$ denotes its modulus of continuity there, then
\[
 \|f_j-f\|_{L^\infty(\overline\D)}
 \leq\omega_f\!\left(\|h_j-h\|_{L^\infty(\overline\D)}\right).
\]
This follows by setting $w=f_j(z)$ and comparing $f(h(w))$ with $f(z)$. This establishes convergence of the whole sequence on the closed disk once $h_j\to h$ has been fixed. In particular, no global forward Dirichlet bound or separate boundary equicontinuity estimate for $f_j$ is needed.

Finally, the quasiconformal inverse formula and change of variables give
\begin{equation}\label{forward-energy-inverse-distortion}
 \int_U\|Df_j(z)\|^2\,dA(z)
 =\int_{f_j(U)}\K_{h_j}(w)\,dA(w).
\end{equation}
For fixed $U\Subset\D$, the sets $f_j(U)$ lie in a common relatively compact subset of $\D$, even before homeomorphism of the limit is known. Otherwise there would be $z_j\in\overline U$ and $w_j=f_j(z_j)\to w\in\Sone$. After a subsequence $z_j\to z\in\overline U$, and uniform convergence would give $h(w)=z$, contradicting $h(\Sone)=\Sone$. Thus \eqref{inverse-distortion-Lone-target} implies \eqref{forward-local-Dirichlet-target}. For the whole disk, \eqref{forward-energy-inverse-distortion} reduces to $\mathcal D(f_j)=\int_\D\K_{h_j}\area$.
\end{proof}

Theorem~\ref{regular-rational-no-interior-hairs} below obtains
\eqref{forward-local-Dirichlet-target} for the regular $\lambda$ family
with convergent Hopf data in the strict upper range. It first obtains
local quasiconformal bounds from the scalar equations and then uses
\eqref{fibre-patch-forward-dirichlet}. Thus this compactness proposition
provides the continuous inverse needed in the subsequent regularity
argument. For other approximation families the forward Dirichlet
bound, or its sufficient inverse form
\eqref{inverse-distortion-Lone-target}, remains a separate analytic
input. The compactness proposition itself is independent of the
ellipticity threshold and uses no variational equation.

\subsection{Common Hopf patches and removing compact interior fibres of monotone mappings} 
\label{common-hopf-patches}
In this subsection we assume $p_{\mathrm c}<p\leq2$, the critical exponent will be
considered later. Consider regular $\lambda$ solutions
$h_j=f_j^{-1}$ with $\lambda_j\uparrow1$ and one fixed quasisymmetric
inverse boundary trace $\beta$, and suppose
\begin{equation}\label{fibre-patch-convergence}
 h_j\longrightarrow h\quad\hbox{uniformly on }\overline\D,\qquad
 \Phi_j:=\Phi_{p,\lambda_j}\longrightarrow\Phi
       \quad\hbox{locally uniformly in }\D .
\end{equation}
The limit $h$ is a continuous monotone surjection. We assume the
NoHairs conclusion $h^{-1}(\Sone)=\Sone$ and a common inverse
$W^{1,2}$ bound, as supplied by Proposition~\ref{compact-containment}
for the bounded-energy family. The holomorphic limit $\Phi$ need
not yet be identified with the Hopf expression of $h$.

Here \emph{regular} means that the maps are diffeomorphisms on $|\mu|<\sqrt{\lambda_j}$ in the disk,
their $\lambda$ Hopf differentials are holomorphic, and the classical
scalar equation is valid on every patch where that differential is
nonzero. It suffices that the maps be $C^1$ throughout and have the
regularity required for the scalar calculations off the Hopf zeros.

\begin{lemma}[Trapping the inverse images of a whole fibre]
\label{whole-fibre-trapping}
Let $a\in\D$ and $E=h^{-1}(a)$. If $V\Subset\D$ is any open
neighbourhood of $E$, there are $\delta>0$ and $j_0$ such that
\[
 f_j(\overline{B_\delta(a)})\subset V\qquad(j\geq j_0),
\]
where this ball is Euclidean. There are also bounded simply connected
smooth Jordan domains
$E\subset U_0\Subset U\Subset V$.
\end{lemma}
\begin{proof}
The fibre $E$ is nonseparating: extend the disk map across the circle by a homeomorphism of the exterior, obtaining a monotone sphere map. The preimage of the connected complement of $\{a\}$ is connected, so the complement of $E$ is connected.
The compact set $\overline\D\setminus V$ contains no preimage of $a$,
so its image under $h$ has positive distance from $a$. Uniform
convergence proves the first assertion after decreasing $\delta$.
To construct the domains, fix one sufficiently large index $k$.
The trapping assertion puts
$f_k(\overline{B_\delta(a)})$ in $V$, while uniform convergence makes
$h_k(E)\subset B_{\delta/4}(a)$. Thus one may take, for example,
$U_0=f_k(B_{\delta/2}(a))$ and $U=f_k(B_{3\delta/4}(a))$.
If only $C^1$ regularity is assumed, smooth the two Jordan boundaries
inside their disjoint collars. The inclusions persist. The index $k$
is now fixed; these domains do not vary with the subsequent sequence.
\end{proof}

\begin{proposition}[Common patches and bounded heights away from Hopf zeros]
\label{zero-free-fibre-patches}
Suppose $\Phi\not\equiv0$ and that the compact fibre $E=h^{-1}(a)$
contains no zero of $\Phi$. Then a fixed neighbourhood of $a$
supports single-valued scalar functions
$v_j=\operatorname{Im}(\zeta_j\circ f_j)$ with uniformly bounded
height oscillations. On a smaller neighbourhood, $t_j$, $K_{f_j}$
and $|Df_j|_{\mathrm{op}}$ are uniformly bounded for all large $j$.
In particular $E$ is a singleton.
\end{proposition}
\begin{proof}
Choose $V\Subset\{\Phi\ne0\}$ containing $E$, and choose
$E\subset U_0\Subset U\Subset V$ by
Lemma~\ref{whole-fibre-trapping}. For large $j$, $\Phi_j$ has no
zeros on $\overline U$. Choose compatible holomorphic square roots
on the simply connected domain $U$, and normalise their primitives
by
\[
 \zeta_j'=\sqrt{\Phi_j},\qquad \zeta_j(w_*)=0
\]
at one fixed $w_*\in U$. The primitives converge locally uniformly
on $U$. More explicitly, there is a finite length $L$ such that
every point of $\overline{U_0}$ can be joined to $w_*$ by a path in
$U$ of length at most $L$. If
$P=\sup_{j\geq j_0}\sup_U|\Phi_j|<\infty$, then
\begin{equation}\label{whole-fibre-height-bound}
 \sup_{U_0}|\zeta_j|\leq L\sqrt P,\qquad
 \operatorname{osc}v_j\leq2L\sqrt P
\end{equation}
wherever $f_j$ takes values in $U_0$.
The finite path-length bound follows by covering $\overline{U_0}$
with finitely many disks compactly contained in $U$ and choosing
one path to each disk.

Now apply our trapping lemma again, with $U_0$ in place of $V$.
That lemma yields a fixed hyperbolic ball $B_{2R}^\eta(a)$ whose closure
is mapped into $U_0$ by $f_j$ for every sufficiently large index $j$.
Although $\zeta_j$ need not be globally one-to-one on $U$, it is
locally conformal there. The scalar reduction is a local
calculation, so the single-valued $v_j$ satisfies
\eqref{lambda-forward-scalar-equation} on this entire fixed ball.
Equation~\eqref{whole-fibre-height-bound} and
Proposition~\ref{lambda-uniform-interior-gradient} give uniform
bounds for $t_j$ on $B_R^\eta(a)$.
The quasiconformal bound follows from
\eqref{scalar-exact-qc-conversion}. Since $|\Phi_j|$ is bounded
away from zero on $\overline{U_0}$, the derivative formula in
Remark~\ref{lambda-gradient-scope} also bounds $Df_j$ there.

If $w_1,w_2\in E$, then $h_j(w_i)\to a$. A fixed smaller Euclidean
ball has a common Lipschitz constant for $f_j$, and hence
\[
 |w_1-w_2|
 =|f_j(h_j(w_1))-f_j(h_j(w_2))|
 \leq C|h_j(w_1)-h_j(w_2)|\longrightarrow0.
\]
Thus $w_1=w_2$.  
\end{proof}

For a nontrivial $\Phi$, the only fibres not yet covered contain a
Hopf zero. Their images form the set
\[
 \Sigma=h(\{w\in\D:\Phi(w)=0\}).
\]
This set is locally finite: if $C\Subset\D$, then $h^{-1}(C)$ is
a compact subset of the disk and contains only finitely many zeros
of $\Phi$. The following estimate allows us to pass across each
point of $\Sigma$ without extending a square-root branch through
the corresponding fibre.

\begin{lemma}[A uniform bound for the curvature term]
\label{lambda-curvature-ratio-bound}
With $p_{\mathrm c}<p\leq2$ and $\lambda\in[\lambda_0,1]$, there is
$D_0=D_0(p,\lambda_0)<\infty$ such that
\begin{equation}\label{lambda-curvature-ratio}
 \frac{t^2\vartheta_{p,\lambda}(t)}{S_{p,\lambda}(t)}
 \leq D_0\qquad(t\geq0).
\end{equation}
\end{lemma}
\begin{proof}
Write the flux as $\mathfrak a=\varepsilon t+b(t)$, as in
Proposition~\ref{lambda-scalar-viscosity}.
The uniform expansions \eqref{lambda-residual-flux-tail} give
constants $T_1,b_*,B,C_1>0$ such that
\[
 b_*\leq b(t)\leq B,\qquad 0<b'(t)\leq C_1t^{-3}
       \quad(t\geq T_1),
\]
independently of the parameter; also $b(t)\leq B$ for all $t$.
For $t\geq2T_1$,
\[
 t^2\vartheta(t)
   =\frac{t^3(\varepsilon+b'(t))}{\varepsilon t+b(t)}
   \leq\frac{\varepsilon t^3+C_1}{\varepsilon t+b_*}.
\]
On the other hand, positivity of $b'$ gives
\[
 S(t)\geq1+\int_{t/2}^t
             \frac{\varepsilon q^2}{\varepsilon q+b(q)}\,dq
 \geq1+\frac{7\varepsilon t^3}{24(\varepsilon t+B)}.
\]
The denominators $\varepsilon t+b_*$ and $\varepsilon t+B$ are
uniformly comparable, proving the bound on this range, including
$\varepsilon=0$. On $0\leq t\leq2T_1$, use
$\vartheta\leq1$ and $S\geq1$.
\end{proof}

\begin{proposition}[Propagation of a boundary bound without a Hopf branch]
\label{lambda-branch-free-propagation}
Fix $p_{\mathrm c}<p\leq2$ and $0<\lambda_0<1$. There are
$R_0,c,T>0$, depending only on $p,\lambda_0$, with the following
property. For any regular $\lambda$ map with $\lambda\in[\lambda_0,1]$
and any hyperbolic ball
$\overline{B_R^\eta(a)}\Subset\D$, $0<R\leq R_0$, set globally
\[
 r_\lambda=|\mu_{f^\lambda}|,\qquad
 t_\lambda=\Theta_{p,\lambda}(r_\lambda),\qquad
 S_\lambda(t)=1+\int_0^t s\vartheta_{p,\lambda}(s)\,ds .
\]
If $t_\lambda\leq L$ on $\partial B_R^\eta(a)$, then
\begin{equation}\label{lambda-boundary-propagation}
 \sup_{B_R^\eta(a)}S_\lambda(t_\lambda)
       \leq e^{cR^2}
                  \left(1+\tfrac12\max\{T^2,L^2\}\right).
\end{equation}
In particular the interior bounds for $t_\lambda$ and
$K_{f^\lambda}$ are independent of $\lambda$. No single-valued
Hopf coordinate on the whole ball is required.
\end{proposition}
\begin{proof}
The quantity $t_\lambda$ is continuous, and this is zero wherever
$\Phi_{p,\lambda}(f^\lambda)=0$. At any point with $t_\lambda>0$
there is a local natural coordinate and a local scalar function
$v$ with $|\nabla_\eta v|=t_\lambda$. Changes of square-root branch
replace $v$ by $\pm v+$ a constant and leave $t_\lambda$ unchanged.

Let $T$ be the uniform threshold in
Lemma~\ref{lambda-uniform-gradient-tail}. Put $\rho=d_\eta(a,\cdot)$,
$\psi=c\rho^2$, and maximise the globally defined function
\[
 G=e^\psi S_\lambda(t_\lambda)
       \quad\hbox{on }\overline{B_R^\eta(a)}.
\]
Suppose this now has an interior maximum where $t_\lambda>T$.
Choose a local scalar coordinate just at this maximum, and use
\eqref{lambda-adapted-maximum-inequality} with the height factor
equal to one and $g=e^\psi$. The terms from differentiating $g$
combine into $\mathcal L\psi$, giving
\[
 0\geq U\psi_1^2+V\psi_2^2
       -4\frac{\vartheta t^2}{S}+\mathcal L\psi
 \geq-|\nabla\psi|^2-4D_0+\mathcal L\psi .
\]
Here $U=S/t^2-\vartheta\geq-1$, $V\geq-1$, and
Lemma~\ref{lambda-curvature-ratio-bound} controls the curvature term.
The hyperbolic Hessian comparison gives
$\nabla^2\rho^2\geq2\,\eta|dz|^2$, so
\[
 \mathcal L\psi\geq2c(1+\vartheta)\geq2c,\qquad
 |\nabla\psi|^2\leq4c^2R^2.
\]
Take $c=4D_0+1$ and $R_0=(4c)^{-1/2}$. The right-hand side is then
at least $2c-c-4D_0=1$, a contradiction.
Thus the maximum is on the boundary or at a point with
$t_\lambda\leq T$. The upper bound in
\eqref{lambda-adapted-factor-growth} gives
\eqref{lambda-boundary-propagation}. Its lower bound converts
this into a bound for $t_\lambda$; then
\eqref{scalar-exact-qc-conversion} bounds the dilatation.
Since \(t_\lambda=0\) at every Hopf zero, any interior maximum point of \(G\) satisfying \(t_\lambda>T\) lies away from the Hopf zeros. A local natural coordinate is therefore available at that point.
\end{proof}

\begin{theorem}[Homeomorphic limits of regular $\lambda$ solutions]
\label{regular-rational-no-interior-hairs}
Under \eqref{fibre-patch-convergence} and the regularity, inverse
Sobolev and NoHairs hypotheses stated above, the limit $h$ is a
homeomorphism of the closed disk and $f_j\to h^{-1}$ uniformly there.
If $\Phi\not\equiv0$, the family $f_j$ has a common quasiconformal
bound on each compact subset of $\D$, and in particular has
uniform local forward $W^{1,2}$ bounds. If $\Phi\equiv0$, the
limit $h$ is a disk automorphism.
\end{theorem}
\begin{proof}
First suppose $\Phi\not\equiv0$.
Proposition~\ref{zero-free-fibre-patches} gives uniform bounds
for $t_j$ near every point outside the locally finite set $\Sigma$.
For $a\in\Sigma$, choose a hyperbolic ball with radius $R\leq R_0$,
compactly contained in the disk, whose closed ball meets $\Sigma$
only at $a$. A finite cover of its boundary by the preceding
patches gives a common bound for $t_j$ on that boundary.
Proposition~\ref{lambda-branch-free-propagation} extends the bound
through the ball, including the images of every approximating
Hopf zero inside it. Covering a compact set by finitely many such
balls and regular patches proves a uniform local bound for $K_{f_j}$.

Since $f_j$ is injective and disk-valued, the area formula now gives,
for every $C\Subset\D$,
\begin{equation}\label{fibre-patch-forward-dirichlet}
 \int_C\|Df_j\|^2\area
   =\int_C\K_{f_j}J_{f_j}\area
   \leq K_C\int_CJ_{f_j}\area\leq \pi K_C .
\end{equation}
Here $\K_{f_j}\leq K_{f_j}\leq K_C$, and $\|\cdot\|$ is the
normalised Hilbert--Schmidt norm used throughout.
Proposition~\ref{forward-Dirichlet-homeomorphic-limit} applies.
Now this gives a continuous right inverse, excludes every interior hair,
and, together with NoHairs, proves the asserted homeomorphism and
uniform convergence of the inverses.

If $\Phi\equiv0$, use the $\lambda$ Hopf relation and
$\omega_{p,\lambda}\geq\omega_p\geq1$ to obtain
\[
 |(h_j)_{\bar w}|^2
      \leq\frac{|\Phi_j|}{\eta(h_j)}\leq|\Phi_j|
      \longrightarrow0
       \quad\hbox{locally uniformly}.
\]
The inverse Sobolev bound identifies the weak limit and $h$,
and hence $h_{\bar w}=0$ distributionally. Weyl's lemma makes
$h$ holomorphic. Its boundary homeomorphism has degree one, so
$h$ is a disk automorphism. Uniform convergence of the inverses
then follows as in Proposition~\ref{forward-Dirichlet-homeomorphic-limit}.
\end{proof}

This argument uses compactness of the whole fibre and the branch-independent estimate, without identifying the Hopf expression of the limit. We make that identification next.

\subsection{The scalar limit and strong convergence of forward derivatives}\label{scalar-limiting-argument}
\begin{proposition}[Regular scalar limits and strong gradient convergence]\label{scalar-uniform-limit}
Let $p_{\mathrm c}<p\leq2$, $\lambda_j\uparrow1$, and let $v_j$ be classical solutions of \eqref{lambda-forward-scalar-equation} on a fixed open set $\Omega\Subset\D$. If $v_j\to v$ locally uniformly, then $v\in C^{2,\alpha}_{\loc}(\Omega)$ for some $\alpha>0$, solves the unregularised scalar equation, and satisfies \eqref{lambda-uniform-height-gradient-bound} at $\lambda=1$. Moreover,
\begin{equation}\label{scalar-strong-Lone-limit}
 \begin{split}
 \nabla v_j&\longrightarrow\nabla v,\\
 A_{p,\lambda_j}(|\nabla v_j|/\sqrt\eta)\nabla v_j
 &\longrightarrow A_p(|\nabla v|/\sqrt\eta)\nabla v
 \end{split}
 \qquad\hbox{in }L^2_{\loc}(\Omega).
\end{equation}
Bounded height oscillations on common patches suffice, after additive normalisation, to extract such a subsequence.
\end{proposition}
\begin{proof}
Fix $U\Subset V\Subset\Omega$. Local uniform convergence bounds the height oscillations, so Proposition~\ref{lambda-uniform-interior-gradient}, applied on slightly larger balls, gives $|\nabla v_j|\leq T$ on $V$. Hence $v\in W^{1,\infty}(V)$ and $\nabla v_j\rightharpoonup\nabla v$ in $L^2(V)$.

Set $\mathbf a_j(z,P)=A_{p,\lambda_j}(|P|/\sqrt{\eta(z)})P$. By \eqref{lambda-scalar-operator-limit} and Proposition~\ref{lambda-uniform-scalar-structure}, these fluxes converge uniformly on bounded gradient ranges, have a common Lipschitz bound there, and satisfy
\[
 [\mathbf a_j(z,P)-\mathbf a_j(z,Q)]\cdot(P-Q)
 \geq c|P-Q|^2\qquad(|P|,|Q|\leq T),
\]
with $c>0$ independent of $j$. Choose $0\leq\chi\in C_c^\infty(V)$ equal to one on $U$. Testing the equation for $v_j$ with $\chi(v_j-v)$ gives
\[
 \int\chi\,\mathbf a_j(z,\nabla v_j)\cdot(\nabla v_j-\nabla v)\,dA
 =-\int(v_j-v)\mathbf a_j(z,\nabla v_j)\cdot\nabla\chi\,dA
 \longrightarrow0.
\]
Weak gradient convergence and uniform convergence of the bounded comparison flux give
\[
 \int\chi\,\mathbf a_j(z,\nabla v)\cdot(\nabla v_j-\nabla v)\,dA
 \longrightarrow0.
\]
Subtracting and using strong monotonicity proves $\nabla v_j\to\nabla v$ in $L^2(U)$. The common Lipschitz bound proves flux convergence and permits passage to the weak equation.

On this bounded gradient range the limiting equation is uniformly elliptic. Difference quotients give a linear uniformly elliptic equation with bounded inhomogeneous terms for the first derivatives; local H\"older estimates and Schauder regularity then give $v\in C^{2,\alpha}_{\loc}$. The coefficient regularity at zero gradient follows from the expansion following \eqref{lambda-scalar-operator-limit}. The height estimate applies with $\lambda=1$. Finally, bounded height oscillations give local gradient bounds and hence compactness after additive normalisation.
\end{proof}

\begin{corollary}[Quasiconformality of a homeomorphic barrier limit away from Hopf zeros]\label{homeomorphic-barrier-scalar-limit}
Let $p_{\mathrm c}<p\leq2$. Suppose $h_j=f_j^{-1}$ are regular disk diffeomorphisms for the $\lambda$ problem, $\lambda_j\uparrow1$, their forward energies are uniformly bounded, and
\[
 h_j\longrightarrow h\quad\hbox{uniformly on }\overline\D,
 \qquad \Phi_{p,\lambda_j}\longrightarrow\Phi
                      \quad\hbox{locally uniformly in }\D.
\]
Assume that $h$ is a homeomorphism. Then on every sufficiently small natural-coordinate patch where $\Phi\ne0$, the inverse $f=h^{-1}$ is locally quasiconformal, with the bound \eqref{lambda-uniform-height-qc-bound}, and $f_j\to f$ strongly in $W^{1,1}_{\loc}$ on its corresponding forward patch. Both limit maps are smooth local diffeomorphisms there and satisfy the unregularised Hopf relation with differential $\Phi$.
\end{corollary}
\begin{proof}
Uniform convergence to a homeomorphism gives $f_j\to f$ uniformly on the closed disk. Choose a small Euclidean disk $U\Subset\{\Phi\ne0\}$ on which a natural coordinate $\zeta$ is one-to-one. For all large $j$, choose compatible square roots and primitives $\zeta_j'^2=\Phi_{p,\lambda_j}$, normalised at the same point. They converge with all derivatives to $\zeta$ on smaller disks. A forward patch $\Omega\Subset h(U)$ therefore has $f_j(\overline\Omega)\Subset U$ for large $j$, and
\[
 u_j+iv_j=\zeta_j\circ f_j\longrightarrow\zeta\circ f=u+iv
 \quad\hbox{locally uniformly on }\Omega.
\]
Proposition~\ref{scalar-uniform-limit} applies. The conjugacy relations

\[
 (u_j)_x=A_{p,\lambda_j}(t_j)(v_j)_y,\qquad
 (u_j)_y=-A_{p,\lambda_j}(t_j)(v_j)_x
\]
pass strongly in $L^1$ to their unregularised versions. Thus $u+iv$ is $C^{2,\alpha}$ locally and has
\[
 |\mu_{u+iv}|=\frac{1-A_p(|\nabla_\eta v|)}{1+A_p(|\nabla_\eta v|)}
            =\Theta_p^{-1}(|\nabla_\eta v|)<1
\]
where its derivative is nonzero, and its full derivative vanishes wherever $\nabla v=0$. Our local gradient estimate provides a uniform bound strictly below one on compact subsets. Since $\zeta$ is conformal, this proves local quasiconformality of $f$. The strong convergence of both scalar components and the locally bounded derivatives of $\zeta_j^{-1}$ give the asserted strong $W^{1,1}$ convergence of $f_j$.

There is no remaining branch point on this patch. To check this directly, set $g=u+iv$. Its coefficient is
\[
 \mu_g=-\Theta_p^{-1}(|\nabla_\eta v|)\,
                 \frac{\overline{g_z}}{g_z}
\]
where $g_z\ne0$, and zero otherwise. The small-gradient expansion following \eqref{lambda-scalar-operator-limit}, with exponent $2/(p-1)\geq2$, shows that this coefficient is locally H\"older even at $Dg=0$: multiplication by the vanishing power removes the angular quotient's singularity. The local linear Beltrami theorem for H\"older coefficients makes a homeomorphic solution a $C^{1,\alpha}$ diffeomorphism; see \cite[Theorem 1.1]{ACFJK}. Thus $Dg$ cannot vanish. The relation $t=\Theta_p(r)$ and the conjugacy identities give the unregularised Hopf equation after inversion. Since $\Phi\ne0$, the distortion is nonzero there, and elliptic bootstrapping gives smoothness.
\end{proof}

The height in this corollary is controlled by the Hopf chart: if $U=B_\rho(w_0)$ and $\sup_U|\Phi|\leq P$, then $\operatorname{osc}_\Omega v\leq2\rho\sqrt P$. A hyperbolic radius $R$ with $B_R^\eta(z_0)\Subset h(U)$ gives the other input to \eqref{lambda-uniform-height-qc-bound}. Note that these are genuine geometric quantities and neither is an assumed dilatation bound.

\begin{corollary}[Local quasiconformality of the full homeomorphic limit]\label{homeomorphic-barrier-full-qc}
Under the hypotheses of Corollary~\ref{homeomorphic-barrier-scalar-limit}, $h$ and $f$ are locally quasiconformal on the whole disk, and $f_j\to f$ strongly in $W^{1,1}_{\loc}(\D)$. For $p_{\mathrm c}<p<2$ they are $C^{2,\alpha}_{\loc}$ diffeomorphisms for the exponents in \eqref{forward-holder-exponent}, and are smooth off the conformal set. If the boundary values are quasisymmetric, they are also globally quasiconformal and asymptotically conformal.
\end{corollary}
\begin{proof}
The energy bound supplies a uniform inverse $W^{1,2}$ bound. If $\Phi\not\equiv0$, the common-fibre and branch-independent estimates in the proof of Theorem~\ref{regular-rational-no-interior-hairs} give a common quasiconformal bound for $f_j$ on every compact subset. Their local hypotheses hold here because $h$ is a homeomorphism: its fibres are singletons and $h^{-1}(\Sone)=\Sone$. Local quasiconformal compactness makes $f$ and $h$ locally quasiconformal throughout the disk. If $\Phi\equiv0$, the estimate $|(h_j)_{\bar w}|^2\leq|\Phi_{p,\lambda_j}|/\eta(h_j)$ makes $h$ conformal directly.

For nontrivial $\Phi$, the strong forward convergence also extends across the images of its isolated zeros. For a small forward disk $B$ around any one of them,
\[
 \int_B\|Df_j\|\area
 \leq\left(\int_B\K_{f_j}\area\right)^{1/2}|f_j(B)|^{1/2}.
\]
The first factor is uniformly bounded, and uniform convergence to the continuous $f$ makes the second uniformly small as $B$ shrinks, for all sufficiently large $j$. Cover the finitely many exceptional points on a compact set by such disks and use the already proved strong convergence on their complement. In the conformal case, $|(h_j)_{\bar w}|\to0$ locally uniformly, while the local Cauchy decomposition and uniform convergence imply $Dh_j\to Dh$ strongly in every finite local $L^q$. For a compact forward set $C$, change variables by the smooth $h_j$ to obtain, in any fixed matrix norm,
\[
 \int_C|Df_j-Df|\area
 =\int_{f_j(C)}
    |\operatorname{adj}(Dh_j)-J_{h_j}Df(h_j)|\area\longrightarrow0.
\]
Indeed the domains lie in a common compact subset, the adjugates converge strongly in $L^2$, the Jacobians in $L^1$, and $Df(h_j)\to Df(h)$ uniformly there. The limiting integrand is zero since $f=h^{-1}$ is conformal.

For $p<2$, on every diffeomorphic patch the classical Hopf relation implies the forward stationarity identity \eqref{forward-stress-distribution}. Near a remaining isolated point, local quasiconformality makes the fields $U$ and $e(r)$ in that identity bounded. Cutoffs deleting a disk of radius $\epsilon$, with gradient $L^1$ norm $O(\epsilon)$, extend the identity distributionally across the point. The proof of Proposition~\ref{forward-regularity-proposition} is local and now applies. Thus the asserted diffeomorphic regularity holds also at these points. This step does not assert real analyticity at conformal points. At $p=2$, the harmonic equation holds off the isolated points; finite local Dirichlet energy makes those points removable, and the harmonic homeomorphism is a real analytic local diffeomorphism by the curved Lewy theorem \cite{Martin}.

Finally, lower semicontinuity of the original inverse energy in the upper range gives $\E_p(h)<\infty$, since $\E_p(h_j)\leq\E_{p,\lambda_j}(h_j)$ are uniformly bounded. With the stated quasisymmetric trace, Proposition~\ref{subquadratic-asymptotic} proves the global conclusions.
\end{proof}

The selected regular family now has a homeomorphic, regular limit; the following corollary identifies its energy.

\subsection{The attained quasiconformal minimum}\label{attained-qc-minimum}
\begin{corollary}[An attained quasiconformal minimum in the strict upper range]
\label{selected-rational-family}
Under the hypotheses of Theorem~\ref{selected-rational-regularity},
there are regular absolute $\lambda$-regularised minimisers with
$\lambda_j\uparrow1$ whose inverses $h_j$ converge, after extraction,
uniformly on $\overline\D$ and strongly in $W^{1,2}_{\loc}$ to a
globally quasiconformal diffeomorphism $h$. Its inverse minimises
$\I_p$ over the finite-energy quasiconformal class with trace
$\beta^{-1}$. The limit is asymptotically conformal, is
$C^{2,\alpha}_{\loc}$ for the exponents above, and is smooth off
its conformal set. In particular,
\begin{equation}\label{selected-family-energy-convergence}
 \E_p(h_j)\longrightarrow\E_p(h)=m_{\rm QC}.
\end{equation}
Theorem~\ref{absolute-monotone-comparison} below identifies this
map with the unique absolute monotone minimiser; the construction
and convergence in this proof do not rely on that comparison.
\end{corollary}
\begin{proof}
The full-measure set of regular parameters can be intersected with
the differentiability set of the minimum $m$.
The proof of Proposition~\ref{parameter-selection} then selects
$\lambda_j\in\Lambda$ with $(1-\lambda_j)(-m'(\lambda_j))\to0$.
It follows that the extra barrier stress vanishes as in
\eqref{hopf-defect-zero}, and the holomorphic Hopf differentials
have a locally uniformly convergent subsequence. The energy bound
is supplied by the common barrier. Theorem~\ref{regular-rational-no-interior-hairs}
and Corollaries~\ref{homeomorphic-barrier-scalar-limit} and
\ref{homeomorphic-barrier-full-qc} give the stated topology,
regularity and quasiconformality of the limit.

Each $h_j$ is quasiconformal, and comparison with every fixed
finite-energy quasiconformal competitor gives
$m(\lambda)\downarrow m_{\rm QC}$. Hence
$m_{\rm QC}\leq\E_p(h_j)\leq m(\lambda_j)$.
Lower semicontinuity gives $\E_p(h)\leq m_{\rm QC}$; the reverse
inequality follows because the limit itself belongs to that
quasiconformal class. This proves \eqref{selected-family-energy-convergence}.
The convexity-defect argument of
Corollary~\ref{strong-monotone-minimising} proves strong local
derivative convergence. After extraction, Scheff\'e's lemma also
gives convergence of the full weighted energy densities in $L^1$.
\end{proof}
\subsection{The lower range: a uniform cutoff for regular barrier minimisers}\label{lambda-variational-selection}
Below the threshold the pole of the $\lambda$ density gives a quasiconformal bound for each fixed parameter, but continuity of the distortion, together with minimality and finite hyperbolic energy, gives a bound independent of that parameter. No differentiable dependence on $\lambda$ is required. For a derivative normalised to $\operatorname{diag}(1,t)$, write
\begin{equation}\label{lambda-stretch-density}
 \phi_{p,\lambda}(t)
 =\phi_p(t)\frac{4t}{4t-(1-\lambda)(1+t)^2},
 \qquad
 \tau_\lambda:=\frac{1-\sqrt\lambda}{1+\sqrt\lambda}<t\leq1,
\end{equation}
with $\phi_p$ as in Lemma~\ref{rankone-weierstrass}. This is the inverse density $J R_\lambda(r^2)$ with $r=(1-t)/(1+t)$, and it converges to $\phi_p$ on every compact subinterval of $(0,1]$, together with its derivatives away from $1$.

\begin{theorem}[A uniform cutoff for regular barrier minimisers]\label{lambda-uniform-selection}
Fix $1<p<p_{\mathrm c}$ and $R$ with
\[
 r_{\mathrm v}(p)<R<r_{\mathrm c}(p).
\]
There is $\lambda_R<1$, depending only on $p,R$, such that every finite-energy minimiser $f_\lambda$ with $\lambda_R<\lambda<1$ satisfies  
\begin{equation}\label{lambda-uniform-cutoff}
 |\mu_{h_\lambda}|<R\quad\hbox{throughout }\D,
 \qquad h_\lambda=f_\lambda^{-1},
\end{equation}
provided $h_\lambda$ is a $C^1$ orientation-preserving diffeomorphism and compactly supported admissible changes of domain are allowed. In particular, for any such family as $\lambda\uparrow1$,
\begin{equation}\label{lambda-limsup-cutoff}
 \limsup_{\lambda\uparrow1}\|\mu_{h_\lambda}\|_\infty
 \leq r_{\mathrm v}(p). 
\end{equation}
On this family the principal coefficients in \eqref{natural-second} have an ellipticity margin independent of $\lambda$, after discarding an initial parameter interval.
\end{theorem}
\begin{proof}
Put $t=(1-R)/(1+R)$. Since $Q_p(R)>1$, the endpoint Weierstrass inequality is strictly violated:
\[
 \phi_p(0)<\phi_p(t)-t\phi_p'(t).
\]
Choose a fixed $0<d<t$ sufficiently small and then a fixed $\theta>0$ sufficiently small, with $s=(t-\theta d)/(1-\theta)<1$, so that
\[
 \theta\phi_p(d)+(1-\theta)\phi_p(s)<\phi_p(t).
\]
All three stretches are positive. Convergence in \eqref{lambda-stretch-density} preserves this strict inequality for $\lambda$ sufficiently close to one, and the corresponding distortions remain strictly below $\sqrt\lambda$.

If $|\mu_{h_\lambda}(w_0)|=R$, use the finite stripe comparison from Lemma~\ref{rankone-weierstrass} near $w_0$. After fixing the cutoff and a sufficiently fine smooth stripe pattern, every test derivative remains strictly inside the pole bound $|\mu|^2<\lambda$. Continuity of $Dh_\lambda$ and $J_{h_\lambda}(w_0)>0$ preserve this margin on a sufficiently small neighbourhood. Thus these are admissible barrier competitors, and freezing the positive weight gives the strict energy decrease above. Therefore such a point cannot occur. This argument uses continuity of the derivative to justify the hard constraint and is not an unproved extension of weak localisation to an infinite-valued density.

Next,  finite energy gives us a point with $|\mu_{h_\lambda}|<R$. In fact, if the continuous distortion were everywhere at least $R$, its forward counterpart would be at least $R$ almost everywhere, and
\[
 \I_{p,\lambda}(f_\lambda)
 \geq\frac{R^p}{\lambda}\int_\D\eta\area=\infty.
\]
Continuity and connectedness now prevent the distortion from crossing the forbidden value $R$. This proves \eqref{lambda-uniform-cutoff}. Taking $R\downarrow r_{\mathrm v}(p)$ proves \eqref{lambda-limsup-cutoff}.

Finally $a_\lambda\geq a_1$ and \eqref{ell-norm} give, on $0\leq r\leq R$,
\[
 \kappa_\lambda(r)\leq
 \max\left\{R,\ \max_{0\leq s\leq R}\kappa_1(s)\right\}<1.
\]
The maximum is strict because $R<r_{\mathrm c}$. At $r=0$ the coefficient norm is interpreted by its zero limit. This is a bound for the principal coefficients; no second derivative at a conformal point has been assumed.
\end{proof}

\begin{remark}[Why the initial regularity hypothesis matters]\label{lambda-nonconvex-remark}
For $1<p<p_{\mathrm c}$ and $\lambda<1$ sufficiently close to $1$, the stretch profile $\phi_{p,\lambda}$ is nonconvex: a negative second derivative of $\phi_p$ persists, whereas the profile is convex near the pole and the conformal endpoint. Its lower convex envelope therefore has an affine segment with interior contacts $\tau_\lambda<t_-<t_+<1$ and supporting line $\phi_{p,\lambda}(t)\geq B+mt$, $B>0$, $m<0$. At $p=3/2$, $\lambda=0.999$, the contact distortions are approximately $0.334418$ and $0.909382$, both larger than $r_{\mathrm v}(3/2)=0.295598\ldots$. On a flat square, $h(x,y)=x+ig(y)$ with $g'=t_-$ below a horizontal line and $g'=t_+$ above it is bi-Lipschitz but not $C^1$. Jensen's inequality applied to
\[
 W_{p,\lambda}(A)\geq B|Ae_1|^2+m\det A
\]
proves minimality for its trace; its Hopf differential is the constant $B$. This flat example shows why the pole bound alone does not supply the derivative continuity used in Theorem~\ref{lambda-uniform-selection}.

\end{remark}

\section{Global comparison and uniqueness}\label{monotone-qc-comparison}
The regular quasiconformal minimiser that we constructed above is also the
unique absolute monotone minimiser. We prove this by a global
Reich--Strebel comparison using the inverse of the regular map.
The argument has two separate inputs: convexity in the principal
stretch ratio, and a weighted approximation of the holomorphic
differential. Neither the convexity of $\Psi_p$ itself nor global
integrability of the Hopf differential is required.
We write
\[
 m_{\rm mon}=\min_{H\in\mathcal M_\beta}\E_p(H),\qquad
 m_{\rm QC}=\inf_{\substack{H\in\mathcal M_\beta\\
                            H\ \mathrm{quasiconformal}}}\E_p(H).
\]

\subsection{Convexity in the principal stretch ratio}
\begin{lemma}[Convexity in the principal stretch ratio]
\label{principal-stretch-convexity}
For $p_{\mathrm c}\leq p\leq2$, put
\begin{equation}\label{principal-stretch-profile}
 \Psi_p(S)=\tfrac12(S-1)^{p/2}(S+1)^{1-p/2},\qquad
 P_p(t)=\Psi_p\!\left(\frac{t+t^{-1}}2\right)
       =\frac{|t-1|^p(t+1)^{2-p}}{4t},\quad t>0.
\end{equation}
Then $P_p$ is $C^1$ and strictly convex. For $t\ne1$,
\begin{equation}\label{principal-stretch-second-derivative}
 P_p''(t)=
 \frac{|t-1|^{p-2}(t+1)^{-p}}{2t^3}
 \bigl[\ell_p t^2+2(p-1)t+1\bigr]>0,
 \qquad \ell_p=2(p-1)^2-1.
\end{equation}
Moreover $P_p'(t)\uparrow1/4$ as $t\to\infty$. With
\[
 F_p(r)=\tfrac12r^{p-1}[p+(2-p)r^2],\qquad
 k=\frac{1+r}{1-r},
\]
one has
\begin{equation}\label{comparison-rankone-margin}
 P_p'(k)=\frac{F_p(r)}{(1+r)^2},\qquad
 \mathfrak n_p(r):=(1+r)^2-4F_p(r)>0\quad(0\leq r<1).
\end{equation}
These assertions include $p=p_{\mathrm c}$. Below that exponent,
$P_p''(t)<0$ for all sufficiently large $t$.
\end{lemma}
\begin{proof}
Differentiate the explicit expression in
\eqref{principal-stretch-profile} separately on the two sides of
$1$. Its first derivative tends to zero at $1$, and the
second derivative is positive because $\ell_p\geq0$. Thus there
is no loss of convexity at the joining point. Expansion at infinity
gives the derivative limit. Since $k$ is finite when $r<1$, strict
increase of $P_p'$ gives \eqref{comparison-rankone-margin}.
When $\ell_p<0$, the quadratic factor in
\eqref{principal-stretch-second-derivative} is eventually negative.
\end{proof}

This is the convexity needed for the comparison. Although
$\Psi_p''<0$ for $p<2$, the profile $P_p$ is strictly convex
throughout the upper range. This threshold agrees with the
polyconvexity and scalar ellipticity thresholds we have already obtained.

\subsection{A weighted holomorphic norm and the regular reference}
\begin{lemma}[A weighted holomorphic norm]
\label{weighted-hopf-approximation}
For $1<q<\infty$, let $\mathcal Q^q(\D)$ be the holomorphic quadratic differentials with finite norm
\begin{equation}\label{weighted-hopf-Bq}
 \|\phi\|_{\mathcal Q^q}^q
 =\int_\D\left(\frac{|\phi|}{\eta}\right)^q dA_\eta
 =\int_\D|\phi(w)|^q(1-|w|^2)^{2q-2}\area.
\end{equation}
This norm is invariant under pullback of quadratic differentials
by disk automorphisms. If it is finite, then
\begin{equation}\label{weighted-hopf-bloch-embedding}
 \sup_\D\frac{|\phi|}{\eta}
 \leq C_q\|\phi\|_{\mathcal Q^q}.
\end{equation}
The bounded holomorphic differentials
$\phi_\rho(w)=\rho^2\phi(\rho w)$ satisfy
\begin{equation}\label{weighted-hopf-dilation}
 \|\phi_\rho-\phi\|_{\mathcal Q^q}\longrightarrow0
 \qquad(\rho\uparrow1).
\end{equation}
\end{lemma}
\begin{proof}
M\"obius invariance follows from
$\eta(A)|A'|^2=\eta$ and $A^*\phi=(\phi\circ A)(A')^2$.
The submean inequality for $|\phi|^q$ on a disk of radius
$(1-|w|)/2$, where the weight is comparable to its central value,
proves \eqref{weighted-hopf-bloch-embedding}. For dilation, first
work on a fixed compact disk, where convergence is uniform.
The circular means of $|\phi|^q$ are increasing with the radius.
Consequently the weighted integral of $|\phi_\rho|^q$ over any
outer annulus is at most the corresponding integral of $|\phi|^q$.
The inequality $|a-b|^q\leq2^{q-1}(|a|^q+|b|^q)$ controls
the tail uniformly in $\rho$. We now let the compact disks exhaust $\D$ to obtain the result.
\end{proof}

The following quantitative consequence of the forward regularity
argument provides the required weighted norm. It uses stationarity
as well as quasiconformality.

\begin{lemma}[Hyperbolic derivative bounds for a quasiconformal reference]
\label{qc-stationary-hyperbolic-bilip}
Let $p_{\mathrm c}\leq p\leq2$. Suppose $f:\D\to\D$ is a
quasiconformal homeomorphism, stationary under compactly supported
inner variations of $\I_p$, with $|\mu_f|\leq\kappa<1$.
Then $f$ and $h=f^{-1}$ are diffeomorphisms and are bi-Lipschitz
for the hyperbolic metric, with constants depending only on
$p,\kappa$. In particular, for some $C=C(p,\kappa)$,
\begin{equation}\label{reference-hyperbolic-derivatives}
 C^{-1}\eta(w)\leq
 \eta(h(w))J_h(w)\leq\eta(h(w))|h_w(w)|^2
 \leq C\eta(w).
\end{equation}
\end{lemma}
\begin{proof}
For $p<2$, Proposition~\ref{forward-regularity-proposition} applies
since $H_p>0$ on $[0,\kappa]$. This proof also reveals uniform
local H\"older bounds for $\mu_f$ in a normalised family:
$U=T(|\mu_f|)\overline{\mu_f}/|\mu_f|$ is bounded by
$T(\kappa)$ and so the contraction margin in
\eqref{forward-stress-contraction} depends only on $p,\kappa$. Then
the Cauchy correction, difference-quotient estimate and local
higher-integrability estimate have uniform constants on
fixed concentric disks. Reconstruction by
\eqref{forward-stress-reconstruction} preserves the uniform
H\"older modulus. At $p=2$ the same argument applies to the
coefficient $\mu_f$. Namely its reconstruction is smooth at zero, and
the stress equation has the same strict contraction on
$[0,\kappa]$. We do not need differentiability of $|\mu_f|$ at its zeros in that case.

\medskip

Next, for each $a\in\D$, choose disk automorphisms $A,B$ with
$A(0)=a$ and $B(f(a))=0$, and set $f_a=B\circ f\circ A$.
M\"obius invariance preserves stationarity and the same
quasiconformal bound. On fixed smaller disks the preceding bounds
and linear Beltrami Schauder estimates give a uniform
$C^{1,\alpha}$ bound for $f_a$, since $|f_a|\leq1$.
The derivatives at zero also have a uniform positive lower bound.
Otherwise a sequence with $Df_a(0)$ tending to a singular matrix
would have a $C^1$ limit near zero. Normalised quasiconformal
compactness makes that limit a nonconstant quasiconformal disk
homeomorphism. The limiting Beltrami coefficient is H\"older and
the nonvanishing-Jacobian theorem
\cite[Theorem 1.1]{ACFJK}, applied to its linear Beltrami field,
excludes the singular derivative. The quasiconformal bound relates
the two singular values, so both are bounded above and below.

The derivative of $f_a$ at zero is exactly the hyperbolic
derivative of $f$ at $a$, up to rotations. This gives uniform
upper and lower bounds at every point. Integrating along
hyperbolic geodesics, and doing the same for the inverse, proves
the bi-Lipschitz assertion and
\eqref{reference-hyperbolic-derivatives}.
\end{proof}

\begin{corollary}[Excess energy and the weighted Hopf norm]
\label{reference-weighted-hopf-bound}
Let $h=f^{-1}$ be as in the preceding lemma, with
$\E_p(h)<\infty$, and suppose its inverse inner variation gives
the holomorphic differential $\Phi=\Phi_p$ in
\eqref{subquadratic-hopf}. For $q=p/(p-1)$,
\begin{equation}\label{reference-Bq-energy-bound}
 C(p,\kappa)^{-1}\E_p(h)
 \leq\|\Phi\|_{\mathcal Q^q}^{q}
 \leq C(p,\kappa)\E_p(h).
\end{equation}
In particular this applies to a finite-energy quasiconformal
minimiser. It does not assert $\Phi\in L^1(\D)$.
\end{corollary}
\begin{proof}
Put $r=|\mu_h|$ and $A_h=\eta(h)|h_w|^2$. The normalisation
of the Hopf differential used here is
\begin{equation}\label{comparison-hopf-normalisation}
 \Phi=2\eta(h)\Psi_p'(\K_h)h_w\overline{h_{\bar w}},
 \qquad |\Phi|=A_hF_p(r),
\end{equation}
with its continuous value zero at $r=0$. Put
$c_p(r)=\tfrac12[p+(2-p)r^2]$, so that $p/2\leq c_p(r)\leq1$.
The identity $(p-1)q=p$ gives the exact moving-weight comparison
\[
 \left(\frac p2\right)^q\E_p(h)
 \leq\int_\D|\Phi|^q A_h^{1-q}\area
 =\int_\D c_p(r)^q A_h r^p\area
 \leq\E_p(h).
\]
Since $A_h\asymp\eta$ by
\eqref{reference-hyperbolic-derivatives}, replacing $A_h^{1-q}$
by $\eta^{1-q}$ proves the two-sided estimate.
Both stationarity identities are legitimate for a quasiconformal
minimiser, since compact smooth variations remain in its admissible
class and stay bounded away from $|\mu|=1$.
\end{proof}

\subsection{The weighted Reich--Strebel inequality for monotone maps}
\begin{lemma}[The Jacobian identity for a proper monotone map]
\label{monotone-weighted-degree-identity}
Suppose $\chi:\overline\D\to\overline\D$ is a continuous
monotone surjection, is the identity on $\partial\D$, maps
$\D$ into $\D$, belongs to $W^{1,2}_{\loc}$, and has
$J_\chi\geq0$ almost everywhere. Then, for every nonnegative
Borel function $g$,
\begin{equation}\label{monotone-weighted-degree-formula}
 \int_\D g(\chi)J_\chi\area=\int_\D g\area.
\end{equation}
\end{lemma}
\begin{proof}
For a smooth function $g$ compactly supported in $\D$, its
preimage is compact in $\D$. We enclose this preimage in a smooth
relatively compact domain whose boundary image avoids
$\operatorname{supp}g$. Uniform and strong local $W^{1,2}$
mollification passes the ordinary degree formula to $\chi$:
Jacobians converge in $L^1$, and the degree is stable away from
the boundary image. That degree equals one on the support,
because $\chi$ fixes the outer circle and has no preimage there
outside our chosen domain. Equality of the resulting positive
measures proves the Borel version by exhaustion. Thus the
formula does not assume an inverse map or a Lusin property for
$\chi$.
\end{proof}

\begin{lemma}[Composition with the inverse of the regular reference]
\label{monotone-reference-composition}
Let $1<p\leq2$. Suppose $h$ is a quasiconformal diffeomorphism
with finite $\E_p(h)$, $|\mu_h|\leq\kappa<1$, and the bounds
\eqref{reference-hyperbolic-derivatives}. For every
$H\in\mathcal M_\beta$ with the same trace as $h$, the map
$\chi=h^{-1}\circ H$ belongs to $W^{1,2}(\D)$, is monotone,
fixes the boundary, and satisfies
\begin{equation}\label{monotone-composition-energy}
 \E_p(\chi)\leq C\bigl(\E_p(H)+\E_p(h)\bigr).
\end{equation}
Here $C$ depends on $p,\kappa$ and the constants in
\eqref{reference-hyperbolic-derivatives}.
\end{lemma}
\begin{proof}
Initially $\chi$ is locally $W^{1,2}$ by the ordinary chain
rule with the smooth local inverse of $h$. It is proper and has
nonnegative Jacobian, so Lemma~\ref{monotone-weighted-degree-identity}
applies. On $J_\chi>0$, put
$r=|\mu_h|\circ\chi$, $s=|\mu_H|$, $t=|\mu_\chi|$, and
$e_p(u)=u^p/(1-u^2)$. The composition formula gives
\[
 t\leq\frac{r+s}{1+rs},\qquad
 e_p(t)\leq C_\kappa\frac{r^p+s^p}{1-s^2}
          \leq C_{p,\kappa}\bigl(r^p+e_p(s)\bigr),
\]
as $s^2\leq s^p$. Multiplication by
$\eta(\chi)J_\chi$ and
$J_H=(J_h\circ\chi)J_\chi$ gives the pointwise estimate
\[
 \eta(\chi)W_p(D\chi)
 \leq C\{\eta(\chi)J_\chi r^p+\eta(H)W_p(DH)\}.
\]
Both sides are continuous functions of the derivative on the
closed cone $J_\chi\geq0$, so the inequality also holds at
rank-one derivatives. We integrate this and use
\eqref{monotone-weighted-degree-formula} to obtain
\[
 \int_\D\eta(\chi)J_\chi(|\mu_h|^p\circ\chi)\area
 =\int_\D\eta|\mu_h|^p\area\leq C\E_p(h).
\]
This proves \eqref{monotone-composition-energy}. Finally,
$\int J_\chi=\pi$ and $W_p(D\chi)\geq|\chi_{\bar w}|^2$
give
\[
 \int_\D\|D\chi\|^2\area
 =\pi+2\int_\D|\chi_{\bar w}|^2\area
 \leq\pi+2\E_p(\chi).
\]
The locally Sobolev map is therefore globally $W^{1,2}$.
\end{proof}

\begin{lemma}[Reich--Strebel comparison with finite excess]
\label{excess-reich-strebel-monotone}
Let $1<p\leq2$, $q=p/(p-1)$, and let $\chi$ be a proper
monotone $W^{1,2}$ disk map with identity boundary trace,
$J_\chi\geq0$, and $\E_p(\chi)<\infty$. For every holomorphic
$\phi\in\mathcal Q^q$, the following integral is absolutely
convergent and nonnegative:
\begin{equation}\label{monotone-RS-excess}
 \mathcal R_\phi(\chi)
 :=2\int_\D\left\{
 |\phi(\chi)||\chi_{\bar w}|^2+
 \operatorname{Re}\bigl(\phi(\chi)\chi_w\chi_{\bar w}\bigr)
 \right\}\area\ \geq0.
\end{equation}
For two such differentials,
\begin{equation}\label{monotone-RS-continuity}
 |\mathcal R_\phi(\chi)-\mathcal R_\psi(\chi)|
 \leq C_p\|\phi-\psi\|_{\mathcal Q^q}
       \bigl(\E_p(\chi)^{1/p}+\E_p(\chi)\bigr).
\end{equation}
\end{lemma}
\begin{proof}
Write $a=\chi_w$, $b=\chi_{\bar w}$, and
$s=|b|/|a|$, setting $s=0$ where $D\chi=0$.
The difference of the integrands is bounded by
$4|(\phi-\psi)(\chi)||a|^2s$. On $s\leq1/2$,
$|a|^2\leq4J_\chi/3$. H\"older's inequality for the measure
$\eta(\chi)J_\chi\,dA$ and
\eqref{monotone-weighted-degree-formula} therefore bound this
part by
\[
 C\|\phi-\psi\|_{\mathcal Q^q}
 \left(\int_\D\eta(\chi)J_\chi s^p\area\right)^{1/p}
 \leq C\|\phi-\psi\|_{\mathcal Q^q}\E_p(\chi)^{1/p}.
\]
On $s>1/2$, including $s=1$, use
$s\leq2^{p-1}s^p$ and
\eqref{weighted-hopf-bloch-embedding}. This gives the second
term in \eqref{monotone-RS-continuity}. Taking $\psi=0$
also proves absolute integrability.

First suppose $\phi$ is holomorphic on a neighbourhood of
$\overline\D$. By \cite[Theorem 1.3 and Remark 1.4]{IOMonotoneApprox},
choose diffeomorphisms $\chi_n$ with identity boundary values,
converging uniformly and strongly in $W^{1,2}$ to $\chi$.
For $\xi_n=\chi_n^{-1}$ and $\nu_n=\mu_{\xi_n}$, the
classical Reich--Strebel inequality in its finite-distortion
form \cite[Section 3]{MY} gives
\[
 2\int_\D\frac{|\phi||\nu_n|^2-
                   \operatorname{Re}(\phi\nu_n)}
                  {1-|\nu_n|^2}\area\geq0.
\]
The inverse derivative formula
$\nu_n\circ\chi_n=-(\chi_n)_{\bar w}/
\overline{(\chi_n)_w}$ changes this expression exactly into
$\mathcal R_\phi(\chi_n)$. These integrals are finite since
$\phi$ is bounded and $\chi_n\in W^{1,2}$. The resulting
formula is quadratic in $D\chi_n$ with bounded continuous
coefficients depending on $\chi_n$, so strong Sobolev and
uniform convergence yield $\mathcal R_\phi(\chi)\geq0$.

For general $\phi\in\mathcal Q^q$, apply this conclusion to
$\phi_\rho=\rho^2\phi(\rho\,\cdot)$ and then use
\eqref{weighted-hopf-dilation} and
\eqref{monotone-RS-continuity}. The order is essential: first
approximate the map for a fixed bounded differential, and only
then remove the differential's regularisation. No convergence of
$\E_p(\chi_n)$ is used.
\end{proof}

For an invertible comparison with the inverse change-of-variables
properties, the preceding formula is equivalently
\begin{equation}\label{homeomorphic-RS-excess}
 \mathcal R_\phi(\chi)
 =2\int_\D\frac{|\phi||\nu|^2-
                     \operatorname{Re}(\phi\nu)}{1-|\nu|^2}\area,
 \qquad \nu=\mu_{\chi^{-1}}.
\end{equation}
Equation~\eqref{monotone-RS-excess} is the form that remains
meaningful without an inverse of the weak map. The generalised
path principle is also explicitly discussed after Theorem~8 in
\cite[Section 3]{MYDiffeomorphic}. Our two-stage approximation
provides its excess-energy version when $\phi$ is not integrable.

\subsection{The comparison theorem}
\begin{theorem}[Strict comparison with every monotone competitor]
\label{absolute-monotone-comparison}
Let $p_{\mathrm c}\leq p\leq2$. Suppose $h$ is a finite-energy
quasiconformal disk diffeomorphism with trace $\beta$, with
holomorphic Hopf differential \eqref{comparison-hopf-normalisation},
and with the hyperbolic derivative bounds
\eqref{reference-hyperbolic-derivatives}. Then, for every
$H\in\mathcal M_\beta$ and $\chi=h^{-1}\circ H$,
\begin{equation}\label{absolute-excess-comparison}
 \E_p(H)-\E_p(h)\geq\mathcal R_\Phi(\chi)\geq0.
\end{equation}
Equality of energies holds if and only if $H=h$. In particular,
the theorem applies to the quasiconformal minimiser furnished by
Corollary~\ref{selected-rational-family}. This result does not use the Hopf
equation, the finite-distortion property or the existence of an inverse for $H$.
\end{theorem}
\begin{proof}
The composition and weighted Hopf estimates above justify all
the integrals in \eqref{absolute-excess-comparison}. It remains 
to prove the pointwise inequality and its equality statement.
At a point $w$, put $x=\chi(w)$, $a=\chi_w$, $b=\chi_{\bar w}$,
and $r=|\mu_h(x)|$. First suppose $J_\chi>0$ and $r>0$.
Let
\[
 k=\frac{1+r}{1-r},\quad q_0=\frac{\Phi(x)}{|\Phi(x)|},\quad
 T=\frac{|a|^2+|b|^2+2\operatorname{Re}(q_0ab)}{J_\chi},\quad
 U=\frac{|a|^2+|b|^2-2\operatorname{Re}(q_0ab)}{J_\chi}.
\]
Then $T,U>0$ and
\[
 TU=1+\frac{4[\operatorname{Im}(q_0ab)]^2}{J_\chi^2}\geq1,
 \qquad \K_H(w)=\tfrac12(kT+U/k).
\]
Since $\Psi_p$ is strictly increasing and $P_p$ is convex,
\begin{align*}
 \Psi_p(\K_H)
 &\geq P_p(kT)\\
 &\geq P_p(k)+kP_p'(k)(T-1).
\end{align*}
The normalisation of $\Phi$ gives
$\eta(h(x))J_h(x)kP_p'(k)=|\Phi(x)|$.
Now multiply this by $\eta(h(x))J_h(x)J_\chi(w)$ to obtain
\begin{equation}\label{pointwise-monotone-excess-comparison}
 \eta(H)W_p(DH)-
       [\eta(h)W_p(Dh)](\chi)J_\chi
 \geq2\{|\Phi(\chi)||b|^2+
              \operatorname{Re}(\Phi(\chi)ab)\}.
\end{equation}
At $r=0$, the reference is conformal, the right side and the
reference energy vanish, and the inequality is immediate.
Continuity in the derivative extends it to $J_\chi=0$.
Integration, the degree identity for the nonnegative integrable
reference energy density, and Lemma~\ref{excess-reich-strebel-monotone}
together prove \eqref{absolute-excess-comparison}.

\medskip

If the total energies agree, the nonnegative pointwise gap in
\eqref{pointwise-monotone-excess-comparison} must vanish almost
everywhere. Where $J_\chi>0$ and $r>0$, strict convexity first
forces $T=1$, and strict increase of $\Psi_p$ then forces $U=1$.
Thus $b=0$. The same conclusion follows directly when $r=0$.
At a nonzero rank-one derivative, write $|a|=|b|=t>0$ and
$c=\operatorname{Re}(q_0ab)/t^2\in[-1,1]$.
After factoring out $\eta(h(x))|h_w(x)|^2t^2$, the gap equals
\[
 1+r^2-2F_p(r)+2\bigl(r-F_p(r)\bigr)c
 \geq(1+r)^2-4F_p(r)=\mathfrak n_p(r)>0.
\]
We used $F_p(r)\geq r$ for $p\leq2$ and
\eqref{comparison-rankone-margin}; the case $r=0$ has gap
$\eta(h(x))|h_w(x)|^2t^2>0$ directly. A zero Jacobian can
therefore occur in the equality case only with $D\chi=0$.
It follows that $\chi_{\bar w}=0$ almost everywhere on $\D$.
Weyl's lemma and the identity boundary trace give
$\chi=\id$, hence $H=h$ as desired.
\end{proof}

\begin{corollary}[Equality of infima and uniqueness in the strict upper range]
\label{upper-absolute-attainment}
Let $p_{\mathrm c}<p<2$, and suppose the inverse trace $\beta$
has a finite-energy quasiconformal extension. Then $\E_p$ has
one and only one minimiser in $\mathcal M_\beta$. This is the
asymptotically conformal quasiconformal diffeomorphism constructed
by $\lambda$ approximation again. It is $C^{2,\alpha}_{\loc}$ for
$0<\alpha<\min\{1,(2-p)/(p-1)\}$, and real analytic off its
conformal set. If $m_{\rm FD}$ denotes the infimum over the
finite-distortion homeomorphisms in $\mathcal M_\beta$, then
\begin{equation}\label{monotone-qc-infimum-target}
 m_{\rm mon}=m_{\rm FD}=m_{\rm QC}=\E_p(h).
\end{equation}
The selected $\lambda$ sequence is an absolute monotone minimising
sequence for the unregularised energy. Every absolute monotone
minimising sequence has subsequences converging uniformly on the
closed disk and strongly in $W^{1,2}_{\loc}$ to this same $h$.
\end{corollary}
\begin{proof}
If the barrier's dilatation ratio is at most $\kappa_*<1$,
choose $\kappa_*^2<\lambda_0<1$. The $\lambda$-density is at
most a fixed multiple of the unregularised density on that
barrier, so the hypothesis of
Corollary~\ref{selected-rational-family} holds. Its quasiconformal
minimiser is stationary for both sets of compact variations.
Lemma~\ref{qc-stationary-hyperbolic-bilip} and
Theorem~\ref{absolute-monotone-comparison} give absolute
minimality and uniqueness. The equalities follow from inclusion
of classes. The $\lambda$ energy convergence is
\eqref{selected-family-energy-convergence}; compactness and
Corollary~\ref{strong-monotone-minimising} give the last assertion.
Analyticity away from the conformal set follows from
Section~\ref{analyticity-scope}.
\end{proof}

The quasiconformal hypothesis on the barrier can now be dropped.

\begin{corollary}[Finite-distortion barriers suffice]
\label{finite-distortion-barrier-attainment}
Let $p_{\mathrm c}<p<2$, and suppose that $\mathcal M_\beta$
contains the inverse of a finite-distortion homeomorphism
$G:\D\to\D$ with $\I_p(G)<\infty$. Then all the conclusions of
Corollary~\ref{upper-absolute-attainment} hold: $\E_p$ has a
unique minimiser $h$ in $\mathcal M_\beta$, it is a quasiconformal
diffeomorphism with the stated regularity, and
\[
 \E_p(h)=m_{\rm mon}=m_{\rm FD}=m_{\rm QC}\leq\I_p(G),
\]
with equality in the last inequality only if $G=h^{-1}$.
\end{corollary}
\begin{proof}
The trace $G|_{\Sone}=\beta^{-1}$ lies in $T^p$ by
Proposition~\ref{finite-distortion-boundary-class}, so $\beta^{-1}$
has a quasiconformal extension $F$ with $\I_p(F)<\infty$, and
$F^{-1}\in\mathcal M_\beta$ is a quasiconformal barrier.
Corollary~\ref{upper-absolute-attainment} applies; the class
$\mathcal M_\beta$ is the same, so its minimiser is unchanged.
The inequality $\E_p(h)\leq\E_p(G^{-1})\leq\I_p(G)$ is
Proposition~\ref{finite-distortion-inverse-calculus} and minimality, and
equality forces $G^{-1}=h$ by uniqueness.
\end{proof}

At $p=2$ the analogous conclusions already follow from the
quadratic variational theorem, and the new comparison applies as
well. At $p=p_{\mathrm c}$,
Theorem~\ref{absolute-monotone-comparison} is still valid, but
the construction of a quasiconformal reference at that endpoint
is not supplied by the strict-range $\lambda$ compactness theorem.

\subsection{The critical exponent}
\begin{proposition}[A precise remaining estimate at the critical exponent]
\label{critical-exponent-reduction}
Fix boundary data with a quasiconformal barrier $h_*$ satisfying
$\E_{p_{\mathrm c}}(h_*)<\infty$. Let $p_n\downarrow p_{\mathrm c}$,
$p_{\mathrm c}<p_n<2$, and let $h_n$ be the unique absolute
minimiser at exponent $p_n$. Suppose, for every $V\Subset\D$,
\begin{equation}\label{critical-local-qc-target}
 \sup_n\|\mu_{h_n}\|_{L^\infty(V)}<1.
\end{equation}
Then the critical energy has a unique minimiser in
$\mathcal M_\beta$, and this is a globally quasiconformal
$C^{2,\alpha}_{\loc}$ diffeomorphism for
$0<\alpha<\sqrt2-1$. It is real analytic off its conformal set.
\end{proposition}
\begin{proof}
The bounds
\[
 \E_2(h_n)\leq\E_{p_n}(h_n)
 \leq\E_{p_n}(h_*)\leq\E_{p_{\mathrm c}}(h_*)
\]
give uniform closed-disk compactness and a monotone limit $h$
mapping interior to interior. Under
\eqref{critical-local-qc-target}, local quasiregular compactness
makes this nonconstant limit open and discrete. Monotonicity then
makes it a homeomorphism, locally quasiconformal.

For each fixed $k$, eventually $p_n\leq p_k$, and for every
$H\in\mathcal M_\beta$ of finite critical energy,
\[
 \E_{p_k}(h_n)\leq\E_{p_n}(h_n)
 \leq\E_{p_n}(H)\leq\E_{p_{\mathrm c}}(H).
\]
Lower semicontinuity at $p_k$, followed by monotone convergence
as $k\to\infty$, proves
$\E_{p_{\mathrm c}}(h)\leq\E_{p_{\mathrm c}}(H)$.
This argument does not assume any critical-energy convergence of
$h_n$. Local quasiconformality permits the forward stationarity
and regularity argument of
Proposition~\ref{forward-regularity-proposition} on compact sets.
The resulting diffeomorphism satisfies the hypotheses of
Proposition~\ref{subquadratic-asymptotic}, which includes the
critical exponent and gives global quasiconformality.
Theorem~\ref{absolute-monotone-comparison} gives uniqueness.
\end{proof}

The estimate \eqref{critical-local-qc-target} is not proved here.
This isolates the endpoint issue: the strict-range scalar gradient
factor degenerates as $p\downarrow p_{\mathrm c}$, whereas the
comparison and its strict equality statement remain valid at
that exponent. A direct regularity proof for the critical
monotone minimiser would be an alternative to this estimate.

\section{Above two}\label{above-two}
\subsection{The original functional has an energy-dependent bound}\label{superquadratic-original}
For $p>2$, the original convention \eqref{funct} is
\begin{equation}\label{superquadratic-original-functional}
 \I_p(f)=\int_\D
       \left(\frac{|\mu_f|}{\sqrt{1-|\mu_f|^2}}\right)^p\eta\area,
 \qquad
 \E_p(h)=\int_\D\eta(h)|h_{\bw}|^pJ_h^{1-p/2}\area.
\end{equation}

\begin{theorem}[The uniformly elliptic logarithmic-distortion equation for $p>2$]\label{superquadratic-original-qc}
Let $p>2$. Suppose $f:\D\to\D$ is an orientation-preserving $C^1$ diffeomorphism, $C^2$ on $\{|\mu_f|>0\}$, with finite \eqref{superquadratic-original-functional}, and its inverse has holomorphic Ahlfors--Hopf differential \eqref{variation}. Put
\[
 r(z)=|\mu_f(z)|,\qquad
 \sigma(z)=\log\frac{r(z)}{\sqrt{1-r(z)^2}}
 \quad\text{on }\{r>0\}.
\]
Then, on that set in the weak sense,
\begin{equation}\label{superquadratic-forward-log-equation}
 \boxed{-\operatorname{div}(\mathsf B_p\nabla\sigma)=8\eta,\qquad
 \mathsf B_p=
 \begin{pmatrix}
 p-1+r^2-2\operatorname{Re}\mu_f&-2\operatorname{Im}\mu_f\\
 -2\operatorname{Im}\mu_f&p-1+r^2+2\operatorname{Re}\mu_f
 \end{pmatrix}.}
\end{equation}
In particular
\begin{equation}\label{superquadratic-forward-ellipticity}
 (p-2)I\leq\mathsf B_p\leq(p+2)I.
\end{equation}
For a disk automorphism $T_z$ with $T_z(0)=z$, set $\mathcal B_z=T_z(\D_{1/2})$. There is a constant $C_p<\infty$, depending only on $p$, such that
\begin{equation}\label{superquadratic-local-energy-bound}
 \boxed{\K_f(z)-1\leq
 C_p\left(\int_{\mathcal B_z}
       \frac{|\mu_f|^p}{(1-|\mu_f|^2)^{p/2}}\eta\area\right)^{2/p}.}
\end{equation}
Consequently $f$ and $h=f^{-1}$ are globally quasiconformal and asymptotically conformal. This proposition does not require a prescribed quasisymmetric boundary extension: its conclusion follows from regularity, stationarity and finite energy.
\end{theorem}
\begin{proof}
Write $r_h=|\mu_h|$. In the inverse Hopf equation the function replacing $F_p$ in \eqref{orthogonal-pullback} is
\[
 \widetilde F_p(r_h)
 =\frac p2 r_h^{p-1}(1-r_h^2)^{1-p/2},\qquad
 |\Phi|=\eta(h)|h_w|^2\widetilde F_p(r_h).
\]
The same orthogonal-metric curvature calculation is valid for this positive function. With $L=(1+r_h)/\sqrt{\widetilde F_p}$ and $M=(1-r_h)/\sqrt{\widetilde F_p}$, its two conductivity eigenvalues are
\begin{equation}\label{superquadratic-inverse-eigenvalues}
 \widetilde\alpha_p(r_h)=\frac1{2r_h}
     \left(1+\frac{p-2}{(1+r_h)^2}\right),\qquad
 \widetilde\beta_p(r_h)=\frac1{2r_h}
     \left(1+\frac{p-2}{(1-r_h)^2}\right).
\end{equation}
Thus $-\operatorname{div}_w(\widetilde{\mathsf A}_p\nabla r_h)
 =4\eta(h)J_h$. Passing to $z=h(w)$ transforms the conductivity to
\[
 \mathsf C_p(z)=
 \left(\frac{Dh\,\widetilde{\mathsf A}_p(Dh)^T}{J_h}\right)(f(z))
\]
and the right side to $4\eta(z)$. Its eigenvalues in the image directions of the natural coordinate are
\[
 \frac{(1+r)^2+p-2}{2r(1-r^2)},\qquad
 \frac{(1-r)^2+p-2}{2r(1-r^2)}.
\]
Since $d\sigma/dr=1/[r(1-r^2)]$, the transformed flux is $\mathsf B_p\nabla\sigma/2$. The inverse Beltrami formula identifies the matrix in \eqref{superquadratic-forward-log-equation}. Its eigenvalues are $(1+r)^2+p-2$ and $(1-r)^2+p-2$, proving \eqref{superquadratic-forward-ellipticity}. The weak curvature justification and the positive-level truncation are the same as for \eqref{scalar-original}; no third derivatives or regularity across Hopf zeros are required.

Here is the quantitative estimate, including the role of energy on escaping balls. Under $z=T(\zeta)$, two-dimensional conductivity transforms by an orthogonal conjugation, preserving \eqref{superquadratic-forward-ellipticity}, while $|T'|^2\eta(T)=\eta$. Thus $\sigma\circ T$ satisfies an equation of the same form. For any real $a$, put $v=(\sigma\circ T-a)_+$, with $v=0$ at conformal points. It belongs to $W^{1,2}(\D_{1/2})$, and positive-part truncation gives
\[
 -\operatorname{div}(\widehat{\mathsf B}_p\nabla v)
 \leq8\eta\leq128/9\quad\text{in }\D_{1/2}.
\]
The uniformly elliptic local boundedness estimate, obtained by the classical power iteration \cite{Moser1960}, yields
\begin{equation}\label{superquadratic-moser-bound}
 v(0)\leq M_p\bigl(\|v\|_{L^2(\D_{1/2})}+1\bigr),
\end{equation}
where $M_p$ depends only on $p$, not on $a$, $T$ or the map. This estimate is applied above a finite level $a$, hence away from all zeros of $\mu_f$.

Set
\[
 e_T=\int_{\D_{1/2}}e^{p\sigma(T(\zeta))}\eta(\zeta)\,dA(\zeta).
\]
If $e_T>0$, take $a=p^{-1}\log e_T$. The elementary inequality
\[
 (s-a)_+^2\leq\frac4{e^2p^2}e^{p(s-a)}
\]
and $\eta\geq1$ give $\|v\|_2\leq2/(ep)$. Equation~\eqref{superquadratic-moser-bound} therefore implies
\[
 \sigma(T(0))\leq\frac1p\log e_T+
 M_p\left(1+\frac2{ep}\right).
\]
Since $\K_f-1=2e^{2\sigma}$, this proves \eqref{superquadratic-local-energy-bound}, for example with
\[
 C_p=2\exp\!\left(2M_p\left(1+\frac2{ep}\right)\right).
\]
If $e_T=0$, continuity makes $\mu_f$ zero on $T(\D_{1/2})$, and the estimate is immediate.

The local energy is at most $\I_p(f)$, proving the global bound \eqref{early-superquadratic-qc-bound}. As $|z|\to1$, the fixed-radius hyperbolic balls $\mathcal B_z$ escape every compact subset. Integrability makes their energies tend to zero, so \eqref{superquadratic-local-energy-bound} proves $|\mu_f(z)|\to0$. The inverse distortion identity and properness transfer quasiconformality and asymptotic conformality to $h$.
\end{proof}

The dependence on the energy is explicit in \eqref{superquadratic-local-energy-bound} and the constant $M_p$ is the uniform elliptic-estimate constant, and probably not best possible. The argument gives no constant depending only on $p$ for arbitrary boundary data. It also does not extend to $p=2$ by substitution, because the lower bound $p-2$ in \eqref{superquadratic-forward-ellipticity} vanishes. The inverse scalar argument above handles that endpoint separately.

\begin{corollary}[A boundary rate from the local excess energy]
\label{superquadratic-boundary-rate}
Under Theorem~\ref{superquadratic-original-qc}, write
\[
 \varepsilon_f(z)=\int_{\mathcal B_z}
       \frac{|\mu_f|^p}{(1-|\mu_f|^2)^{p/2}}\,dA_\eta.
\]
Then
\[
 |\mu_f(z)|\leq
 \left(\frac{C_p\varepsilon_f(z)^{2/p}}
             {2+C_p\varepsilon_f(z)^{2/p}}\right)^{1/2}
 \leq\sqrt{C_p/2}\,\varepsilon_f(z)^{1/p}.
\]
Thus every specified decay bound for this fixed-radius local energy
gives the corresponding rate of asymptotic conformality.
\end{corollary}
\begin{proof}
Use $r^2=(\K_f-1)/(\K_f+1)$ in
\eqref{superquadratic-local-energy-bound}.
\end{proof}

The logarithmic equation is uniformly elliptic on $\{\mu_f\ne0\}$,
so its local estimates apply there. The logarithm itself tends to
$-\infty$ at a conformal point; a Harnack statement across that
set is not inferred. The positive-level truncations in the proof
are what make the distortion estimate valid at all points.

\subsection{Proofs of Theorems~\ref{thm-C} and \ref{thm-D}}
\begin{proof}[Proof of Theorem~\ref{thm-C}]
The first variation gives the holomorphic differential \eqref{variation}. For $1<p\leq2$, apply Proposition~\ref{subquadratic-asymptotic}. Below $p_{\mathrm c}$, Theorem~\ref{weak-variational-cutoff} supplies $|\mu_h|\leq r_{\mathrm v}(p)<r_{\mathrm c}(p)$, since a diffeomorphism has no rank-one set; this also gives $K_h\leq K_{\mathrm v}(p)$. For $p>2$, Theorem~\ref{superquadratic-original-qc} gives the conclusions and \eqref{early-superquadratic-qc-bound}. The identity $K=\K+\sqrt{\K^2-1}$ gives the stated bound, and inversion preserves supremal dilatation.

For each fixed $p_{\mathrm c}<p\leq2$, the smooth boundary maps constructed in Proposition~\ref{lower-range-boundary-obstruction} can require arbitrarily large extension dilatation. They have finite-energy quasiconformal barriers. Corollary~\ref{upper-absolute-attainment} for $p<2$, and Theorem~\ref{quadratic-wp-existence} for $p=2$, give regular minimisers for these data. Their dilatations cannot have a bound depending only on $p$.

\end{proof}

\begin{proof}[Proof of Theorem~\ref{thm-D}]
This is Theorem~\ref{superquadratic-original-qc} and Corollary~\ref{superquadratic-boundary-rate}.

\end{proof}

\section{The boundary class and the canonical Hopf family}\label{besov-section}
\subsection{Besov boundary data and the regular minimiser theorem}
\label{besov-boundary-minimisers}
For $1<p<\infty$, the critical Besov space on the circle is
$B^{1/p}_{p,p}(\IS)$. For a real-valued $u\in L^p(\IS)$ we use the
Slobodeckij seminorm
\begin{equation}\label{boundary-besov-seminorm}
 [u]_{B^{1/p}_{p,p}}^p
 =\int_0^{2\pi}\int_0^{2\pi}
 \frac{|u(\theta)-u(t)|^p}{|e^{i\theta}-e^{it}|^2}
 \,d\theta\,dt.
\end{equation}
The notation $B_p$ in the integrable Teichm\"uller literature denotes
this space, modulo constants. We distinguish it from a class of
boundary homeomorphisms: let $\mathfrak B_p$ consist of the
orientation-preserving quasisymmetric maps
\[
 f_0(e^{i\theta})=e^{i\varphi(\theta)},\qquad
 \varphi(\theta+2\pi)=\varphi(\theta)+2\pi,
\]
whose increasing lift is absolutely continuous, with
$\varphi'>0$ almost everywhere and
\begin{equation} \log\varphi'\in B^{1/p}_{p,p}(\IS)\end{equation}
Three-point normalisation gives the corresponding Teichm\"uller
space; we retain the same notation for its unnormalised boundary
representatives when no quotient is needed.

Wei--Matsuzaki's characterisation gives
\begin{equation}\label{besov-classical-class}
 \mathfrak B_p=T^p,\qquad 1<p<\infty
 \quad\cite[Theorem 5.5]{WeiMatsuzakiIntegrable}.
\end{equation}
The  heat-kernel extension in that paper gives us  a quasiconformal representative
with a hyperbolically $L^p$ Beltrami coefficient
\cite[Theorem 4.5 and Appendix]{WeiMatsuzakiHeat}.
For a quasiconformal map $F$ with $\|\mu_F\|_\infty\leq k<1$,
\begin{equation}\label{besov-barrier-energy-equivalence}
 \int_\D|\mu_F|^p\,dA_\eta
 \leq\I_p(F)
 \leq(1-k^2)^{-p_*/2}
       \int_\D|\mu_F|^p\,dA_\eta.
\end{equation}
Thus $\mathfrak B_p$ is exactly the class admitting a finite-energy quasiconformal barrier, and Proposition~\ref{finite-distortion-boundary-class} identifies it with $\mathcal T_p$. For $p_{\mathrm c}<p<2$, Corollary~\ref{upper-absolute-attainment} supplies the unique inverse minimisers $F_{(p)},H_{(p)}$; at $p=2$ the same conclusion, with an analytic hyperbolic harmonic inverse, is Theorem~\ref{quadratic-wp-existence}. Writing $r_p=|(H_{(p)})_{\bar w}|/|(H_{(p)})_w|$, their canonical Hopf differential is

\begin{equation}\label{besov-canonical-hopf}
 \boxed{\begin{gathered}
 \Phi_p[f_0]
 =\frac12\eta(H_{(p)})\bigl[p+(2-p)r_p^2\bigr]
       r_p^{p-2}(H_{(p)})_w\overline{(H_{(p)})_{\bar w}},\\
 \partial_{\bar w}\Phi_p[f_0]=0.
 \end{gathered}}
\end{equation}
The whole product is defined as zero at $r_p=0$.
At $p=1$ we retain the definition of $T^1$ by a quasiconformal extension with $\mu\in L^1(dA_\eta)$; the intrinsic endpoint description uses the modified Besov space in \cite{Matsuzaki}.

\begin{corollary}[Membership in $T^p$ from the monotone minimiser]\label{regular-minimiser-class-entry}
Let $p_{\mathrm c}<p<2$, let $\beta$ be quasisymmetric, and suppose
that $\mathcal M_\beta$ contains a finite-energy competitor, so that
$\E_p$ has an absolute minimiser $h$ in $\mathcal M_\beta$ by
Theorem~\ref{monotone-existence}. Then $h$ is a quasiconformal
diffeomorphism if and only if $\beta^{-1}\in T^p$. More generally,
for $1<p\leq2$, if $h$ satisfies the hypotheses of
Proposition~\ref{subquadratic-asymptotic}, then $f=h^{-1}$ has $f|_{\Sone}\in T^p$.

\end{corollary}
\begin{proof}
If $\beta^{-1}\in T^p$, then $\mathcal M_\beta$ contains a
quasiconformal barrier and Corollary~\ref{upper-absolute-attainment}
makes the unique minimiser a quasiconformal diffeomorphism.
Conversely, if $h$ is a quasiconformal diffeomorphism with
$\E_p(h)<\infty$, then $f=h^{-1}$ is quasiconformal with
\[
 \int_\D|\mu_f|^p\eta\area
 \leq\int_\D\frac{|\mu_f|^p}{1-|\mu_f|^2}\eta\area
 =\E_p(h)<\infty,
\]
so $f$ is a representative of $f|_{\Sone}=\beta^{-1}$ in the
definition of $T^p$. The same calculation proves the last assertion,
since Proposition~\ref{subquadratic-asymptotic} makes $f$
quasiconformal.
\end{proof}

\subsection{Canonical Hopf differentials and variation of the exponent}\label{canonical-hopf-family}
For $f_0\in T^p$, $p_{\mathrm c}<p\leq2$, fix a three-point normalisation of the unique inverse minimisers and set

\begin{equation}\label{canonical-p-data}
 m_p(f_0)=\I_p(F_{(p)})=\E_p(H_{(p)}),\qquad
 \mathfrak H_p(f_0)=\Phi_p[f_0].
\end{equation}

\subsubsection{The dual exponent and the weighted Hopf norm}
Let $q=p/(p-1)$ and use the space $\mathcal Q^q(\D)$ of Lemma~\ref{weighted-hopf-approximation}.

Corollary~\ref{reference-weighted-hopf-bound} gives
\begin{equation}\label{canonical-hopf-dual-norm}
 \mathfrak H_p:T^p\longrightarrow\mathcal Q^{p/(p-1)}(\D),
 \qquad
 C^{-1}m_p(f_0)\leq
 \|\Phi_p[f_0]\|_{\mathcal Q^{p/(p-1)}}^{p/(p-1)}
 \leq C m_p(f_0),
\end{equation}
where $C$ depends on $p$ and the quasiconformal bound of the minimiser. The exponent $q$ is the H\"older conjugate of $p$. At $p=2$ this recovers the weighted Hilbert norm and canonical energy of \cite[Theorem D]{MYWP}.

\subsubsection{A common class for varying the exponent}
Fix $f_0\in T^1$. Choose one quasiconformal extension $F_*$ with
$\mu_{F_*}\in L^1(dA_\eta)$, and put $M=\I_1(F_*)<\infty$.
Since $0\leq|\mu_{F_*}|<1$,
\begin{equation}\label{common-p-family-barrier}
 f_0\in T^p\quad(p>1),\qquad
 \I_p(F_*)\leq M\quad(1\leq p\leq2).
\end{equation}
The second assertion concerns the denominator power one; no
common bound for the different convention above $p=2$ is being used.
For every $p\in(p_{\mathrm c},2]$ let $F_{(p)},H_{(p)},\Phi_p$ be the
canonical maps and differential with this fixed trace.
Here $T^1$ is a convenient common boundary class, not the assertion
that the variational problem at exponent one has a diffeomorphic
minimum. Corollary~\ref{endpoint-nonattainment} excludes that
conclusion for non-M\"obius data.

\begin{proposition}[Compactness and one-sided continuity in the exponent]
\label{hopf-exponent-stability}
For fixed $f_0\in T^1$, the following hold on $(p_{\mathrm c},2]$.
\begin{enumerate}[label=\textup{(\roman*)}]
\item $0\leq m_p(f_0)\leq M$, and $m_p(f_0)$ is nonincreasing
in $p$. It is strictly decreasing unless $f_0$ is the trace of a
M\"obius automorphism.
\item The family $\{\Phi_p[f_0]:p_{\mathrm c}<p\leq2\}$ is
locally uniformly bounded, and hence relatively compact for
locally uniform convergence of holomorphic coefficients.
\item If $p_j\to p_0\in(p_{\mathrm c},2]$ and
$m_{p_j}(f_0)\to m_{p_0}(f_0)$, then
\begin{equation}\label{exponent-energy-stability}
 \begin{gathered}
 H_{(p_j)}\to H_{(p_0)}\quad\hbox{uniformly on }\overline\D
 \hbox{ and strongly in }W^{1,2}_{\loc}(\D),\\
 F_{(p_j)}\to F_{(p_0)}\quad\hbox{uniformly on }\overline\D,
 \qquad
 \Phi_{p_j}\to\Phi_{p_0}\quad\hbox{locally uniformly on }\D.
 \end{gathered}
\end{equation}
\item These conclusions hold whenever $p_j\downarrow p_0$ with
$p_0<2$. In particular $m_p$ and $\Phi_p$ are right-continuous.
The family in (iii) is continuous at every continuity point of
$m_p$, so there are at most countably many possible exceptional
parameters in the open interval $(p_{\mathrm c},2)$.
\end{enumerate}
\end{proposition}
\begin{proof}
The common barrier proves the energy bound. If $p<q\leq2$,
the minimiser $F_{(p)}$ is an admissible competitor at exponent $q$ and
\[
 m_q\leq\I_q(F_{(p)})\leq\I_p(F_{(p)})=m_p.
\]
The second inequality is strict whenever $F_{(p)}$ is nonconformal:
the integrand decreases strictly on the positive-area set where
$0<|\mu_{F_{(p)}}|<1$. A conformal disk diffeomorphism has M\"obius
boundary trace. This proves (i).

By \eqref{subquadratic-energy-order},
$\E_2(H_{(p)})\leq\E_p(H_{(p)})\leq M$. Therefore
Proposition~\ref{compact-containment}, applied at the fixed exponent
two, gives uniform equicontinuity, a common $W^{1,2}$ bound, and
compact containment of $H_p(V)$ for every $V\Subset\D$.
The modulus of \eqref{besov-canonical-hopf} satisfies
\begin{equation}\label{hopf-family-local-bound}
 |\Phi_p|\leq\eta(H_{(p)})|(H_{(p)})_w|^2,
 \qquad 1<p\leq2.
\end{equation}
Indeed $\tfrac12[p+(2-p)r^2]r^{p-1}\leq1$ for $0\leq r\leq1$.
Consequently the differentials have a common local $L^1$ bound.
The holomorphic mean-value inequality and Montel's theorem prove (ii).

For (iii), take a subsequential uniform and weak $W^{1,2}$ limit
$H$ of $H_{(p_j)}$. It is a proper monotone map with the prescribed
trace. We first record the lower-semicontinuity argument with
varying exponents. For $p,q$ in a compact subinterval of $(1,2]$,
and $|b|\leq|a|$, the inverse density obeys
\begin{equation}\label{varying-p-local-density}
 |W_p(a,b)-W_q(a,b)|
 \leq C|p-q|(|a|^2+|b|^2).
\end{equation}
To see this, write $W_p=|a|^2r^p$ and use
$\sup_{0\leq r\leq1}r^s|\log r|\leq1/(es)$ for $s>0$,
with the continuous value zero at $r=0$.
On compact source sets the weights $\eta(H_{(p_j)})$ are uniformly
bounded. Hence \eqref{varying-p-local-density}, the common
$W^{1,2}$ bound and Lemma~\ref{weighted-lsc} imply
\begin{equation}\label{varying-p-lsc}
 \E_{p_0}(H)\leq\liminf_j\E_{p_j}(H_{(p_j)}).
\end{equation}
This passage first uses compact cutoffs and then exhaustion;
It does not compare two infinite global Jacobian terms.
Under the energy-convergence hypothesis, $H$ is the unique
minimiser $H_{(p_0)}$.

Convergence of the total energies and the same lower bound on
the complementary cutoffs give convergence of the local energies.
Estimate~\eqref{varying-p-local-density} replaces their exponent
by $p_0$ at a cost tending to zero. The weighted Jacobians converge
by \eqref{weighted-jacobian-convergence}. Thus the local integrals
of the convexified density $G_{p_0}$ converge. Its positive quadratic
convexity constant, since $p_0>p_{\mathrm c}$, gives strong local
$W^{1,2}$ convergence exactly as in
Corollary~\ref{strong-monotone-minimising}.
Uniform convergence to the homeomorphism $H_{(p_0)}$ also gives
uniform convergence of the inverses on the closed disk.

The Hopf expression is continuous in $(p,a,b)$ on the
orientation-preserving cone for $p>1$, including $b=0$ and the
zero gradient, and is bounded by a constant times $|a|^2+|b|^2$.
Strong local $L^2$ convergence of the gradients therefore gives
local $L^1$ convergence of the Hopf coefficients, using uniform
integrability. Holomorphy improves this to locally uniform
convergence. Uniqueness of the subsequential limit proves (iii)
for the whole sequence; Hopf zeros require no separate exclusion.

If $p_j\downarrow p_0$, monotonicity and the competitor $F_{(p_0)}$
give $m_{p_j}\leq m_{p_0}$. Applying
\eqref{varying-p-lsc} to a compactness limit gives the reverse
liminf bound. Thus the energies converge and (iii) applies.
Finally a finite monotone real function has at most countably
many discontinuities. This proves (iv).
\end{proof}

\paragraph{What remains for a continuous or smooth parameter family.}
To obtain continuity at every parameter, it suffices to prove the
missing recovery estimate from the stronger-energy side:
\begin{equation}\label{left-exponent-recovery}
 \limsup_{q\uparrow p}\,m_q(f_0)\leq m_p(f_0),
 \qquad p_{\mathrm c}<p\leq2.
\end{equation}
The opposite inequality follows from monotonicity. It would suffice
that, for every $\epsilon>0$, there is a quasiconformal extension
$G_\epsilon$ with the same trace, $\I_1(G_\epsilon)<\infty$, and
$\I_p(G_\epsilon)\leq m_p+\epsilon$, since dominated convergence
then gives $\I_q(G_\epsilon)\to\I_p(G_\epsilon)$ as $q\uparrow p$.
Thus one common barrier does not by itself supply the required recovery estimate.

\subsection{Proof of Theorem~\ref{thm-B}}
\begin{proof}[Proof of Theorem~\ref{thm-B}]
The equivalence of \textup{(a)} and \textup{(c)} follows from \eqref{besov-classical-class} and \eqref{besov-barrier-energy-equivalence}. The implication \textup{(c)}$\Rightarrow$\textup{(b)} is immediate, and \textup{(b)}$\Rightarrow$\textup{(a)} is Proposition~\ref{finite-distortion-boundary-class}. Given \textup{(c)}, the inverse barrier is admissible for Corollary~\ref{upper-absolute-attainment}, which yields the unique absolute monotone minimiser and its regularity. For every finite-energy finite-distortion competitor $G$, Proposition~\ref{finite-distortion-inverse-calculus} gives
\[
 \I_p(F_{(p)})=\E_p(H_{(p)})
 \leq\E_p(G^{-1})\leq\I_p(G).
\]
Equality forces $G^{-1}=H_{(p)}$ by uniqueness. This proves \textup{(d)}, the equality of the infima and the asserted uniqueness; \textup{(d)}$\Rightarrow$\textup{(c)} is immediate. Corollary~\ref{reference-weighted-hopf-bound} gives \eqref{intro-hopf-norm}, and Proposition~\ref{hopf-exponent-stability} gives the continuity assertions.

\end{proof}
\appendix
\section{The second-order Hopf equation and its ellipticity}\label{other-formulations}
For $1<p\leq2$, we record the natural-coordinate equation for the inverse map and the coefficient bounds used in Theorem~\ref{lambda-uniform-selection}.

\subsection{A concise derivation of the second-order Hopf equation}
Set
\[
 S=\frac{1+r^2}{1-r^2},\qquad
 \Psi_\lambda(S)=R_\lambda\left(\frac{S-1}{S+1}\right).
\]
Use the Hopf normalisation
\[
 q_\lambda=\eta(h)\Psi_\lambda'(S)h_w\overline{h_{\bw}}.
\]
This is one half of the normalisation in \eqref{variation}. For a $C^2$ diffeomorphic stationary map with $0<r^2<\lambda$, a natural coordinate makes $q_\lambda=1$ and $\mu_h=r$ real and positive throughout the patch.

Put
\[
 U=h_w,\quad X=h_{ww},\quad V=h_{w\bw},\quad
 \xi=\frac{U}{\overline U},\quad
 g=rU\,\partial_w\log\eta(h),
 \qquad
 a_\lambda=(S^2-1)\frac{\Psi_\lambda''(S)}{\Psi_\lambda'(S)}.
\]
Differentiating the constant Hopf equation gives
\begin{equation}\label{natural-system}
 (1+a_\lambda)V+r\xi\overline V=a_\lambda rX-g.
\end{equation}
Together with its conjugate this yields, when the denominator is nonzero,
\begin{align}\label{natural-second}
 V&=\alpha_\lambda X+\beta_\lambda\xi\overline X+\ell_\lambda,\\
 \alpha_\lambda&=\frac{a_\lambda r(1+a_\lambda)}{(1+a_\lambda)^2-r^2},\qquad
 \beta_\lambda=-\frac{a_\lambda r^2}{(1+a_\lambda)^2-r^2},\notag\\
 \ell_\lambda&=\frac{-(1+a_\lambda)g+r\xi\overline g}
 {(1+a_\lambda)^2-r^2}.\notag
\end{align}
The two coefficients are real; $|\xi|=1$. The familiar sufficient ellipticity condition for the equation on $U=h_w$ is
\begin{equation}\label{ell-norm}
 \kappa_\lambda(r):=|\alpha_\lambda|+|\beta_\lambda|
 =\frac{|a_\lambda|r}{\bigl||1+a_\lambda|-r\bigr|}<1.
\end{equation}
These are principal-part identities for classical solutions away from Hopf zeros.

\subsection{The ellipticity threshold and a necessary branch restriction}
For $\lambda=1$, direct differentiation gives
\begin{equation}\label{alim}
 a_1(r)=-\frac{p(2-p)(1-r^2)}{p+(2-p)r^2}.
\end{equation}
Writing $c=p(2-p)$, the condition in \eqref{ell-norm} is equivalent to
\begin{equation}\label{ell-polynomial}
 H_p(r)>0.
\end{equation}
Indeed, it says $1+a_1>2r/(1+r)$, or $p+(2-p)r^2>p(2-p)(1+r)^2$.

For $(1+\sqrt3)/2<p<2$, cancellation of $1-r$ gives the coefficient bound
\begin{equation}\label{uniform-limit-ell}
 \kappa_1(r)=
 \frac{p(2-p)r(1+r)}{p(p-1)-p(2-p)r+(2-p)r^2}
 \leq\frac{p(2-p)}{(p-1)^2}.
\end{equation}
The right side is strictly less than one precisely when $p>p_{\mathrm c}$. For example, the bound is $19/81$ at $p=1.9$. To check the bound algebraically, put $d=p(p-1)$ and $A=1+p-p^2$. The denominator minus $(p-1)^2r(1+r)$ is $(1-r)(d-Ar)\geq0$ in this range. At $p=p_{\mathrm c}$, every $r<1$ is elliptic, but the supremum is one.

\subsection{What the $\lambda$ barrier does to ellipticity}\label{lambda-ellipticity}
With $S=(1+r^2)/(1-r^2)$ and $\Psi_\lambda(S)=r^p/(\lambda-r^2)$, direct differentiation gives

\begin{equation}\label{alambda}
 \boxed{a_\lambda(r)=p-2+
 \frac{2(2-p)r^2}{p\lambda+(2-p)r^2}
 +\frac{4r^2(1-\lambda)}
 {(\lambda-r^2)(1-r^2)}.}
\end{equation}
 On the allowed branch, $a_\lambda\geq a_1$. Hence for $p_{\mathrm c}<p\leq2$ and $0<\lambda<1$,
\begin{equation}\label{lambda-ell-bound}
 \kappa_\lambda(r)\leq
 \max\left\{\sqrt\lambda,\frac{p(2-p)}{(p-1)^2}\right\}<1.
\end{equation}
To see this, when $a_\lambda\geq0$ the ratio in \eqref{ell-norm} is at most $r$. When $a_1\leq a_\lambda<0$, that ratio decreases as $a_\lambda$ increases and is bounded by \eqref{uniform-limit-ell}. The critical exponent also has a strict bound for each fixed $\lambda<1$, obtained by restricting $r\leq\sqrt\lambda<1$.

Below $p_{\mathrm c}$, nonellipticity at a fixed distortion of the limiting equation persists for $\lambda$ sufficiently close to one. Even above $p_{\mathrm c}$, the bound is not uniform across the moving pole: for $\delta=1-\lambda$ and $r_\lambda^2=\lambda-\delta^2$, one has $a_\lambda\sim4/\delta^2$ and $\kappa_\lambda(r_\lambda)\to1$. The uniform scalar estimates of Section~\ref{scalar-equations} use a different formulation.

\section*{Acknowledgements and author information}
We thank Professor M.~Brakalova for encouraging this investigation after our talk at the third Quasiworld conference.

\noindent\textbf{AI disclosure.} This work was substantially developed before the widespread use of generative AI. Claude and GPT assisted with proofreading, reference checks and exposition, and suggested relevant literature, including work of Riestenberg--Smillie and Ripoll--Tomi.

Work of both authors was partially supported by the New Zealand Marsden Fund. C. Yao was partially supported by the Young Scientist Program of the Ministry of Science and Technology of China (No. 2021YFA1002200), the National Natural Science Foundation of China (Nos. 12401096 and 12101362), the Natural Science Foundation of Shandong Province (Nos. ZR2024QA035 and ZR2022YQ01), and the Fundamental Research Funds for the Central Universities.

\medskip
\noindent Gaven Martin: Institute for Advanced Study, Massey University, Auckland, New Zealand.\\
\texttt{g.j.martin@massey.ac.nz}

\medskip
\noindent Cong Yao: Research Center for Mathematics and Interdisciplinary Sciences, Shandong University, Qingdao 266237, and Frontiers Science Center for Nonlinear Expectations, Ministry of Education, P. R. China.\\
\texttt{c.yao@sdu.edu.cn}

\end{document}